\documentclass{article}
\usepackage{amssymb,amsmath,amsfonts,amsthm,latexsym}
\usepackage{multicol}
\usepackage{graphicx}
\usepackage[latin1]{inputenc}

\newtheorem{theo}{Theorem}

\newtheorem{ques}{Question}

\newtheorem{lemm}{Lemma}
\newtheorem{coro}[lemm]{Corollary}
\newtheorem{prop}[lemm]{Proposition}

\newtheorem{rema}[lemm]{Remark}
\newtheorem{defi}[lemm]{Definition}

\newtheorem*{lemm*}{Lemma}
\newtheorem*{clai*}{Claim}

 \def\NN{{\mathbb N}}  \def\PP{{\mathbb P}}
 \def\RR{{\mathbb R}}

\def\La{\Lambda}

\def\cA{{\cal A}}  \def\cG{{\cal G}}  
\def\cB{{\cal B}}   \def\cN{{\cal N}} \def\cT{{\cal T}}
\def\cC{{\cal C}}  \def\cI{{\cal I}}  \def\cU{{\cal U}}
    
\def\cE{{\cal E}}    
   \def\cR{{\cal R}} \def\cX{{\cal X}}

\def\set#1{\left\{\, #1 \,\right\}}

\def\norm #1{\Vert \,#1\, \Vert\,}

\newcommand{\flow}[1][t]{\phi_{#1}}

\title{Star flows and multisingular hyperbolicity}
\author{Christian Bonatti, Adriana da Luz}

\begin{document}

\maketitle
\begin{abstract} A vector field $X$ is called a star flow if every  periodic
orbit, of any vector field $C^1$-close to $X$, is hyperbolic.
It is known that the chain recurrence classes of a generic star flow $X$ on a $3$ or $4$
manifold are either  hyperbolic, or \emph{singular hyperbolic}  (see \cite{MPP} for $3$-manifolds  and \cite{GLW} on $4$-manifolds).

As it is defined, the notion of singular hyperbolicity forces the singularities in a same class to have the same index.
However in higher dimension (i.e $\geq 5$) \cite{BdL} shows that  singularities of different indices may be robustly in the same chain recurrence class of a star flow.
Therefore the usual notion of singular hyperbolicity is not enough for characterizing the star flows.

We present a  form of
hyperbolicity (called \emph{multisingular hyperbolic}) which makes compatible the hyperbolic structure of regular orbits together
with the one of singularities  even if they have different indices. We show that  multisingular hyperbolicity implies that the flow is  star,
and conversely we prove that  there is a $C^1$-open and dense subset of the open set of star flows
which are multisingular hyperbolic.

More generally, for most of the hyperbolic structures (dominated splitting, partial hyperbolicity etc...) well defined on regular orbits,
we propose a way for generalizing it to a compact set containing singular points.
\end{abstract}



\section{Introduction}

\subsection{General setting and historical presentation}

Considering the infinite diversity of  the dynamical behaviors, it is natural to have a   special interest on the \emph{robust properties} that is, properties that are
impossible to break  by small perturbations of a system; in other words,
a dynamical property is robust if its holds on a (non-empty) open set of diffeomorphisms or flows.

One important starting point of the dynamical systems as a mathematical field has been the characterization of the structural stability (i.e. systems whose topological dynamics are
unchanged under small perturbations) by the hyperbolicity (i.e. a global structure expressed in terms of transversality and of uniform expansion and contraction).
This characterization, first stated by the stability conjecture \cite{PaSm}, was proven for diffeomorphisms in the $C^1$ topology by Robin and Robinson in \cite{R1}, \cite{R2}
(hyperbolic systems are structurally stable) and Ma\~{n}\'e \cite{Ma2} (structurally stable systems are hyperbolic).
The equivalent result for flows (also for the $C^1$ topology) leads to extra difficulties and was proven by \cite{H2}.

 We can see in this case how the robustness of the properties is related with the structure in the tangent space: in this case, a very strong robust property
 is related to a very strong uniform structure. However, hyperbolic systems are not dense in the set of diffeomorphisms or flows; instability and non-hyperbolicity may be robust.
 In order to describe a larger set of systems, one is lead to consider less rigid robust properties, and to try to characterize them by (weaker)
structures that limit the effect of the small perturbations.

In this spirit there are several results for diffeomorphisms in the $C^1$ topology:
\begin{enumerate}
          \item A system is \emph{robustly transitive} if every $C^1$-close system is transitive.
          \cite{Ma} proves that robustly transitive surface
          diffeomorphisms are globally hyperbolic (i.e. are Anosov diffeomorphism). This is no more true in higher dimensions (see examples in \cite{Sh,Ma1}). \cite{DPU, BDP} show that robustly transitive
          diffeomorphisms admits a structure called \emph{dominated splitting}, and their  \emph{finest dominated splitting} is \emph{volume partially hyperbolic}.  This result extends
          to \emph{robustly transitive sets}, and to \emph{robustly chain recurrent sets}.
\item One says that a system is \emph{star} if  all periodic orbits are hyperbolic in a robust fashion: every periodic orbit of every $C^1$-close system is hyperbolic.
For a diffeomorphism, to be star  is equivalent to be hyperbolic
(an important step is done in \cite{Ma}  and has been completed in \cite{H}\label{star}).
\end{enumerate}

Now, what is the situation of these results for flows?
The dynamics of flows seems to be very related with the dynamics of diffeomorphisms. Even more, the dynamics of vector fields in
dimension $n$ looks like the one of diffeomorphisms in dimension $n-1$. Several results can be translated from one setting to the other, for instance by considering the suspension.
For example, \cite{D} proved that robustly transitive flows on $3$-manifolds are Anosov flows, generalizing Ma\~ n\'e's result for surface diffeomorphisms. More generally, in any dimension,
if a vector field is robustly transitive (or chain recurrent) then \cite{Vi} shows that it is non singular, and its \emph{linear Poincar\'e flow} (that is, the natural action of the differential on the
normal bundle) admits a dominated splitting which is volume partially hyperbolic. On the other hand, if one consider the suspension of robustly transitive diffeomorphism without partially hyperbolic splitting
(as built in \cite{BV}) one gets a robustly transitive vector field $X$ whose flow $\{\phi^t\}$ does not admit any dominated splitting. This leads to the fundamental observation that,

\vskip 2mm
\emph{for flows, the hyperbolic structures are living on the normal bundle for the linear Poincar\'e flow, and not on the tangent bundle.}
\vskip 2mm

However, there is a phenomenon which is really specific to vector fields: the existence of singularities (zero of the vector field) accumulated, in a robust way, by regular recurrent orbits.
Then, some of the previously mentioned results may fail to be translated to the flow setting.

The first example with this behavior has been indicated by Lorenz in \cite{Lo} under  numerical evidences.  Then \cite{GW} constructs
a $C^1$-open set of vector fields in a $3$-manifold, having a topological transitive attractor containing  periodic orbits (that are all hyperbolic) and one
singularity.  The examples in \cite{GW} are known as the geometric Lorenz attractors.

The Lorenz attractor is also an example of a robustly non-hyperbolic  star flow,  showing that the result in \cite{H} is no more true for flows.
In dimension $3$ the difficulties introduced by the robust coexistence of singularities and periodic orbits is now almost fully understood.
In particular, Morales, Pacifico and Pujals (see \cite{MPP} ) defined the notion of \emph{singular hyperbolicity}, which requires some compatibility
between the hyperbolicity of the singularity and the hyperbolicity of the regular orbits. They prove that,
\begin{itemize}
 \item for  $C^1$-generic star flows on $3$-manifolds,
every chain recurrence class is singular hyperbolic. 
It was conjectured in \cite{GWZ} that the same result could hold without the generic assumption.
However \cite{BaMo} built a star flow on a $3$-manifold having a chain recurrence class which is not singular hyperbolic, contradicting the conjecture.
We exhibit a very simple such example in Section~\ref{s.example}.
\item any robustly transitive set containing a singular point of a flow on a $3$-manifold is either a
singular hyperbolic attractor or a singular hyperbolic repeller.
\end{itemize}

The singular hyperbolic structure for a compact invariant set $K$ of a vector field $X$ on a $3$-manifold
is equivalent to the existence of a volume partially hyperbolic splitting
of the tangent bundle for the flow $\phi^t$, for $t\neq 0$.

Let us make two observations
\begin{itemize}
 \item The singular hyperbolic structure is living on the tangent bundle (and not on the normal bundle) contradicting our fundamental observation above.
 \item When the compact set is singular, the splitting has only $2$ bundles, one of dimension $1$  and the other of dimension $2$;  this asymmetry forces the index of the
 singularities contained in $K$.  In other words, all the singularities contained in a singular hyperbolic compact set have the same index, by definition of the singular hyperbolicity.
 The examples of star flows in  \cite{BaMo}, as well as the examples presented in Section~\ref{s.example}, contain singular points with distinct indices, and therefore are not singular hyperbolic.
\end{itemize}

\subsection{Discussion on the notion of singular hyperbolicity in dimensions $>3$}

\subsubsection{The natural generalization of the singular hyperbolicity}
In higher dimensions, we are far from understanding the relation between robust properties and hyperbolic structures for  singular flows. Nevertheless,
the notion of singular hyperbolicity defined by \cite{MPP} in dimension $3$ admits a straightforward generalization in higher dimension:
following \cite{MM,GWZ}, a chain recurrence class  is called \emph{singular hyperbolic} if the tangent bundle over this class admits a dominated, partially hyperbolic splitting
in two bundles, one being uniformly contracted (resp. expanded) and the other is sectionally area expanded (resp. sectionally area contracted).

\begin{rema}Far from the singular points, the singular hyperbolicity is equivalent to the hyperbolicity,
as the direction spanned by the vector field is neither contracted nor expanded: the uniform area expansion means the uniform expansion of the normal directions.
\end{rema}

This notion has been very helpful for the study of singular star flows.
If the chain recurrence set of a vector field $X$  can be covered by filtrating sets $U_i$ in which the maximal invariant set $\La_i$ is singular hyperbolic, then
$X$ is a star flow.  Conversely, \cite{GLW} and \cite{GWZ}  prove that this property characterizes the  generic star flows on  $4$-manifolds.
In \cite{GSW} the authors prove the singular hyperbolicity of generic star flows in any dimensions assuming
an extra property:  if two singularities are in the same chain recurrence class then
they must have the same \emph{$s$-index} (dimension of the stable manifold). Indeed, as already mentioned above,
the  singular hyperbolicity implies directly this extra property.

However, in dimension $\geq 4$, singularities of different indices may coexist $C^1$-robustly in the same class, and these class may have a robust property which requires
a notion of hyperbolicity.  For instance:

\begin{itemize}\item  In dimension $4$,  \cite{BLY} build  a flow  having a robustly chain recurrent attractor
containing saddles of different indices. In particular, this attractor is not singular hyperbolic in the meaning of \cite{GSW}.
\item in \cite{BdL}, an example is announced ,of a star flow  in  dimension $5$ admitting singularities of different
indices which belong robustly to the same chain recurrence class. This example cannot satisfy the singular hyperbolicity used in \cite{GSW}.
\end{itemize}

\subsubsection{A local solution for a local problem}

If we want to explain robust properties of chain recurrent classes containing singularity, one needs to understand how to turn compatible the hyperbolic structure on the regular orbits
with the local dynamics in  the  neighborhood of the singular point:

\emph{why the regular orbits do not loose their hyperbolic structure when crossing a small neighborhood of the singularity ?}

That is a local problem.

The singular hyperbolicity, as defined in \cite{MM,GWZ,GSW}, is a global way for fixing this local problem.  As a consequence, if several singularities coexist in the same class,
the global solution needs to solve the local problem corresponding to each singularity: it is therefore natural that the singular
hyperbolicity assume that the singularities are of the same kind.
This explains why the singular hyperbolicity could not characterize all the star flows but only those for which singular points of distinct
indices are assumed to belong in distinct chain recurrence classes.

This paper provides a local answer to these local problem:
the way for fixing the hyperbolic structure of the regular orbits with the one of a
given singular point needs to be independent to what we do in the neighborhood of the
other singular points.  For that:
\begin{itemize}\item the main new tool will be Theorem~\ref{t.reparametrization} which associates a cocycle
to any singularity of a vector field.
\item Another important tool built in \cite{GLW} is the generalized linear Poincar\'e flow, and we need to recall its construction for presenting our results. .
\item the last tool will be the notion of extended maximal invariant set. Such a notion has been already defined and used in \cite{GLW,GSW};
we propose  here a slightly different notion and we compare it (see Theorem~\ref{t.nioki})  with the one in \cite{GLW,GSW}
\end{itemize}

Given any usual notion of hyperbolic structure  (hyperbolicity, partial hyperbolicity, volume hyperbolicity, etc..),
well defined on compact invariant sets far from the singularities,
we propose a notion of \emph{(multi)singular hyperbolic structure}
generalizing it to compact invariant sets containing singular points.

Then we will illustrate the power of this notion by paying a special attention to star flows. In this particular setting, the usual structure (for regular orbits)
one wants to generalize to the singular
setting is the uniform hyperbolicity. In order to avoid confusion with the singular structure defined in \cite{GLW}, we will call \emph{multi-singular hyperbolicity}
our way for generalizing the uniform hyperbolicity to singular sets. Then,
Theorem~\ref{t.converse} proves that this multisingular hyperbolicity characterizes the star flows in any dimensions:
\begin{itemize}
\item the multi-singular hyperbolic flows are star flows,
\item conversely, a  $C^1$-open and dense subset of the star flows consists in multi-singular hyperbolic flows.
\end{itemize}

We notice that the example in Section~\ref{s.example} as well as the ones in \cite{BaMo} are multisingular hyperbolic.

In the same spirit, generalizing the results in \cite{BDP}, the second author announces in \cite{dL} that every $C^1$-robustly chain recurrent class of a singular flow
is \emph{singular volume partially hyperbolic}: in particular the example of robustly chain recurrent attractors in \cite{BLY} are singular volume partially hyperbolic.

\section{Presentation of our results}

\subsection{The extended linear Poincar\'e flow}

The hyperbolic structure we will define does not lie on the tangent bundle, but on the normal bundle.
However the flows we consider are singular and so that the normal bundle (and therefore the linear Poincar\'e flow)
is not defined at the singularities. In \cite{GLW},  the authors  define the notion of \emph{extended linear Poincar\'e flow} defined on some sort of blow-up of the singularities.
Our notion of \emph{multisingular hyperbolicity} will be expressed in term  of this extended linear Poincar\'e flow (see the precise definition in Section~\ref{s.extended});
we present it roughly below.

\begin{itemize}
\item The linear Poincar\'e flow is the natural linear cocycle over the flow,
on the normal bundle to the flow. The dynamics on the fibers is the quotient of the derivative of the flow by the direction of the flow: this is possible as the direction of the flow is invariant.

\item We denote by $\PP M$ the projective tangent bundle of $M$, that is a point $L$ of $\PP M$ corresponds to a line of the tangent space at a point of $M$. The flow $\phi^t$ of $X$ induces (by the action of its derivative $D\phi^t$) a
topological flow $\phi^t_\PP$ on $\PP M$ which extends the flow of $X$.

\item The projective tangent bundle $\PP M$  admits a natural
bundle called the \emph{normal bundle $\cN$}: the fiber over  $L\in\PP M$ (corresponding to a line $L\subset T_xM$)  is the quotient $N_L=T_xM/L$.
The derivative of the flow $D \phi^t$ of $X$  passes to the quotient on the normal bundle $\cN$
in a linear cocycle over $\phi^t_\PP$, called the
\emph{extended linear Poincar\'e flow} and we note it  as $\psi^t_{\cN}$.

\item The vector field $X$ provides an embedding of  $M\setminus Sing(X)$ into the projective tangent bundle $\PP M$: to a point $x$ one associates the line directed by $X(x)$.
We denote by $\La_X\subset \PP M$ the compact subset defined by the union
$$ \La_X=\{<X(x)>\in \PP T_xM, x\in M\setminus Sing(X)\}\cup \bigcup_{y\in Sing(X)} \PP T_yM.$$
$\La_X$ is a compact set, invariant under the topological flow $\phi^t_\PP$.

\item The restriction of $\psi^t_\cN$ to the fibers of $\PP M$   over $\{<X(x)>\in \PP T_xM, x\in M\setminus Sing(X)\}$
is naturally conjugated to the linear Poincar\'e flow over $M\setminus Sing(X)$;
thus the restriction of $\psi^t_\cN$ to $\La_X$ is a natural extension, over the singular points, of the linear Poincar\'e flow.

The hyperbolic structures we will defined are defined  for $\psi^t_\cN$ over $\phi^t_\PP$-invariant compact subsets of $\La_X$.
\end{itemize}

\subsection{A local cocycle associated to a singular point}

Let $X$ be a vector field and $\phi^t$ its flow.  A (multiplicative) cocycle over $X$ is
a continuous map $h\colon \La_X\times \RR\to\RR$, $h(L,t)=h^t(L)$,  satisfying the cocycle relation
$$h^{r+s}(L)= h^{r}(\phi^s_\PP(L))\cdot h^s(L)$$

\begin{rema}
For instance, fix a Riemannian metric $\| \cdot \|$ on $M$.  For $L\in\PP T_xM$ and $t\in \RR$ one denote
$$h^t_X(L)= \|(D_x\phi^t)|_L\|,$$
where $(D_x\phi^t)|_L$ is the restriction to $L$ of the derivative at $x$ of the flow $\phi^t$.

The map $h_X\colon \La_X\times \RR\to \RR$,  $(L,t)\mapsto h^t_X(L)$,  is a cocycle that we will call \emph{the expansion in the direction of the flow}.

This cocycle depends on the choice of the metric $\|\cdot \|$.  However if $h'_X$ is obtained by another choice $\|\cdot \|'$ of the metric then
$$\frac{h^t_X(L)}{{h'_X}^t(L)}$$
is uniformly bounded on $\La_X\times \RR$.  Thus the cocycle of expansion in the flow direction is \emph{unique up to bounded cocycles}.
\end{rema}

Let $\sigma\in Sing(X)$ be an isolated singular point. We will say that $h^t$ is a  \emph{local cocycle} at $\sigma$
if for any neighborhood $U$ of $\PP T_\sigma M$ in $\La_X$ there is a constant $C>1$ so that
$$\frac 1C< h^t(L)< C$$
for every $(L,t)\in \La_X$  with $L\notin U$ and $\phi^t_\PP(L)\notin U$.

\begin{defi} Let $\sigma\in Sing(X)$ be an isolated singular point.  A \emph{local reparametrization cocycle associated to $\sigma$} is
a cocycle $h^t\colon \La_X\times \RR\to \RR$ so that
\begin{itemize}
 \item  $h^t$ is a local cocycle at $\sigma$;
 \item there is a neighborhood $U$ of $\sigma$ and $C>1$  so that
 $$ \frac 1C <\frac{h^t(L)}{h_X^t(L)}<C$$
 for any $t\in \RR$ and $L\in \La_X$ so that $L\in \PP T_xM$ with  $x\in U$ and $\phi^t(x)\in U$,  where $h^t_X$ is the cocycle of expansion in the flow direction.
\end{itemize}
\end{defi}

The main idea of this paper is based the following result :

\begin{theo}\label{t.reparametrization} $X$ be a vector field on a closed manifold and let $\sigma$ be a simple singularity of $X$ (that is the derivative of $X$ at
$\sigma$ is invertible). Then
\begin{itemize}
 \item There is  a local reparametrization cocycle $h_{\sigma}\colon\La_X\to \RR$ associated to $\sigma$
 \item If $h'_\sigma$ is another reparametrization cocycle then $\frac{h_\sigma^t}{(h'_\sigma)^t}$ is uniformly bounded (thus $h_\sigma$ is unique up to bounded cocycles)
 \item There is a $C^1$-neighborhood $\cU$ of $X$ and a continuous map $Y\in\cU\mapsto h_{Y,\sigma_Y}$
 where $h_{Y,\sigma_Y}$ is a reparametrization cocycle for $Y$ and the continuation $\sigma_Y$ of $\sigma$ for $Y$.
\end{itemize}
\end{theo}
The proof of Theorem~\ref{t.reparametrization} is the aim of  Section~\ref{ss.reparametrization}.

Notice that the product of two cocycles is a cocycle, and the power of a cocycle is a cocycle.
\begin{defi} Let $X$ be a vector field on a compact manifold so that the zeros of $X$ are all simple.
We call \emph{reparametrization cocycle} every cocycle $h$ so that there are numbers $\alpha(\sigma)\in\RR$, $\sigma\in Sing(X)$ so that
$$h^t=\Pi_{\sigma\in Sing(X)} (h^t_\sigma)^{\alpha(\sigma)}$$
where $h_\sigma$ is a local reparametrization cocycle associated to $\sigma\in Sing(X)$.
\end{defi}

\subsection{Hyperbolic structures}

\subsubsection{The usual hyperbolic structures on non-singular sets}\label{ss.non-singular}

We start by recalling what are the usual notions of hyperbolic structures on non-singular compact sets (see Section~\ref{ss.hyp} and~\ref{ss.robust} for precise definitions).
A \emph{hyperbolic structure} for a flow on a non-singular compact $\phi^t$-invariant set $K$ is
\begin{itemize}
 \item a dominated splitting  $N=E_1\oplus_{_<}\cdots\oplus_{_<} E_k$ of the normal bundle over $K$ for the linear Poincar\'e flow.
 \item for some of the bundles $E_i$,and a number $1\leq d_i\leq  dimE_i$ then $Det (J(\psi^t_\cN|_{D_i}))$
 presents  a uniform contraction or expansion for any subspace $D_i\subset E_i$ of dimension $d_i$.
\end{itemize}

For instance:
\begin{itemize}
 \item $K$ is \emph{hyperbolic} if there is a dominated splitting $N=E\oplus_{_<} F$ so that the vectors in $E$ are uniformly contracted and the ones in $F$ are uniformly expanded.
 \item $K$ is \emph{volume partially hyperbolic} if there is a dominated splitting $N=E_1\oplus_{_<}\cdots\oplus_{_<} E_k$ so that the volume in $E_1$ is uniformly contracted
 and the one in $E_k$ is uniformly expanded.
\end{itemize}

\subsubsection{Hyperbolic structures over compact subsets of $\PP M$}\label{ss.hyperbolicstructures}

Consider now a vector-field $X$ on a compact manifold $M$ and $\La_X\subset \PP M$.  We assume that every singularity of $X$ is simple.
Let $K\subset \La_X$ be a $\phi^t_\PP$-invariant compact set.  A \emph{singular hyperbolic structure on $K$} is

\begin{itemize}
 \item a dominated splitting  $\cN=E_1\oplus_{_<}\cdots\oplus_{_<} E_k$ of the normal bundle over $K$ for the extended linear Poincar\'e flow.

 \item for some of the bundles $E_i$,and a number $1\leq d_i\leq  dimE_i$ there is a reparametrization cocycle $h_i^t=\Pi_{\sigma\in Sing(X)} (h^t_\sigma)^{\alpha_i(\sigma)}$ so that $$Det (J(h_i^t\cdot (\psi^t_\cN|_{D_i}))$$
 presents  a uniform contraction or expansion for any subspace $D_i\subset E_i$ of dimension $d_i$.
\end{itemize}

For instance

\begin{itemize}
 \item $K$ is  \emph{multi-singular hyperbolic} if there is a dominated splitting $\cN=E\oplus_{_<} F$ for $\psi^t_\cN$  and there are two reparametrization cocycles $h_s^t$ and $h_u^t$
 so that the vectors in $E$ are uniformly contracted by the flow $h^t_s\cdot \psi^t_\cN$  and the ones in $F$ are uniformly expanded by the flow $h_u^t\cdot \psi^t_\cN$.

 \item $K$ is \emph{volume partially hyperbolic} if there is a dominated splitting $\cN=E_1\oplus_{_<}\cdots\oplus_{_<} E_k$  for $\psi^t_\cN$
 and two reparametrization cocycles $h_1^t$ and $h_k^t$ so that the volume in $E_1$ is uniformly contracted   by the flow $h_1^t\cdot \psi^t_\cN$
 and the one in $F_k$ is uniformly expanded by the flow $h_k^t\cdot \psi^t_\cN$.
\end{itemize}

We show that singular hyperbolic structures are robust in the following  sense:

\begin{lemm}Let $X$ be a vector field on a compact manifold. If $K\subset \La_X\subset \PP M$ is a $\phi_\PP$-invariant compact set admitting a singular hyperbolic structure, then
there is a $C^1$-neighborhood of $X$ and a neighborhood $U$ of $K$ in $\PP M$ so that, for any $Y$ in $\cU$ the maximal invariant set of $\phi_{Y,\PP}$ in $\La_Y\cap U$ admits
the same singular hyperbolic structure.
\end{lemm}
This lemma is a straightforward consequence of the fact that the reparametrization cocycles used for defining the singular hyperbolic
structures admit a continuous choice with respect to the vector field (last item of Theorem~\ref{t.reparametrization}).

\subsection{The extended maximal invariant set}\label{MIS}

The next difficulty is to define the set on which we would like to define an hyperbolic structure.

We are interested in the dynamics of $X$ in  a compact region $U$ on $M$, that is, to describe the maximal
invariant set $\La(X,U)$ in $U$. An important property is that the maximal invariant set
depends upper-semi continuously on the vector field $X$.  This property is fundamental in the fact that ``having a hyperbolic structure''
is a robust property.

Therefore we needs to consider a compact part of $\PP M$, as small as possible, such that
\begin{itemize}
\item it is invariant under the flow $\phi^t_\PP$
\item it contains all the direction spanned by $X(x)$ for  $x\in\La(X,U)\setminus Sing(X)$,
\item  it varies upper semi-continuously with $X$
\end{itemize}

We denote by $\La_{U,\PP}(X)$ the closure in $\PP M$ of $\{<X(x)>, x\in \La(X,U)\setminus Sing(X)\}$: it is a $\phi^t_\PP$-invariant compact set, but in general it fails to
upper semicontinuously with $X$.

The smallest compact set satisfying all the required properties is

$$\widetilde{\La(X,U)}=\limsup_{Y\underrightarrow{C^1} X} \La_{U,\PP}(Y). $$

\begin{defi} We will say that \emph{$X$ has a singular hyperbolic structure}  in a compact region $U$ if the compact set $\widetilde{\La(X,U)}\subset \La_X\subset \PP M$
satisfies this singular hyperbolic structure, as defined in
Section~\ref{ss.hyperbolicstructures}.
\end{defi}

As a straightforward consequence of the upper semi-continuous dependence of $\widetilde{\La(X,U)}$ on the vector field $X$ one gets the robustness of the singular hyperbolic structure of
$X$ in  $U$

\begin{lemm} If $X$ has a singular hyperbolic structure in a compact region $U$ then the same singular hyperbolic
structure holds for every vector fields $Y$ $C^1$-close to $X$.
\end{lemm}

\begin{rema} If $X$ in non-singular on $U$ then the singular hyperbolic structure of $X$ is equivalent to the corresponding (non-singular) hyperbolic structure.

More generally, if $X$ has a singular hyperbolic structure on $U$ then every $\phi_t$-invariant compact set $K\subset U\setminus Sing(X)$
has the corresponding (non-singular) hyperbolic structure.
\end{rema}

The set $\widetilde{\La(X,U)}$  is a fundamental tool for defining the singular hyperbolic structures.
However, it may be hard to calculate  because it depends not only on $X$ but
also on every $C^1$-small perturbations of $X$. That is a little bit unsatisfactory: hyperbolic structures are invented for controlling the effect of small perturbations, however,
for knowing if $\widetilde{\La(X,U)}$ admits some hyperbolic structure, we need to understand the effect of the perturbations of $X$.

\vskip 2mm

In what follows, we propose another set, much simpler to compute, since it does not depend on perturbations of $X$.

In Section~\ref{ss.centerspace} we define the notion of \emph{central space $E^c_{\sigma,U}$ of a singular point $\sigma\in U$}.
Then we call \emph{extended maximal invariant set } and we  denote
by $B(X,U)\subset \PP M$ the set  of all the lines $L$ contained
\begin{itemize}\item either in the central
space of a singular point in $\bar U$
\item or directed by the vector $X(x)$ at a regular point $x\in \La(X,U)\setminus Sing(X)$.
\end{itemize}

Proposition~\ref{p.extended} proves that $B(X,U)$ varies upper semi-continuously with the vector field $X$. In particular, once again, the existence of a dominated splitting
$\cN_L=E_L\oplus F_L$ of the normal bundle $\cN$ over $B(X,U)$ is a robust property,
as well as the existence of a singular hyperbolic structure.  Further  more remark \ref{r.inclusion} shows it is larger than  $\widetilde{\La(X,U)}$:
$$\widetilde{\La(X,U)}\subset B(X,U).$$

\subsection{Hyperbolic structures over a chain recurrence class}.

Another set on which one is interested to define hyperbolic structures are the chain recurrence classes.  For the non-singular classes that is the classical notions.
So we are interested to define hyperbolic structure on the chain recurrence class $C(\sigma)$  of a singular point $\sigma$.

Conley theory asserts that  any chain recurrent  $C$ class admits a basis of neighborhood which are nested filtrating neighborhood $U_{n+1}\subset U_n$, $C=\bigcap U_n$
(see Section~\ref{ss.chainrecurrent} for the definitions).  We define

$$\widetilde{\La(C)}=\bigcap_n\widetilde{\La(X,U_n)}  \mbox{ and }B(C)=\bigcap_n B(X,U_n).$$
These two sets are independent of the choice of the sequence $U_n$.  Clearly $\widetilde{\La(C)}\subset B(C)$.

\begin{defi}
We say that a chain recurrence class $C$ has a given singular hyperbolic structure if $\widetilde{\La(C)}$ carries this singular hyperbolic structure.
\end{defi}

\begin{rema}
  If $C$ is a chain recurrence class which has a singular hyperbolic structure then $X$ has this singular hyperbolic structure  on a small filtrating neighborhood of $C$.
\end{rema}

If $\sigma\in Sing(X)$ is an hyperbolic singular point we define $E^c_\sigma=\bigcap_n E^c_{\sigma,U_n}$ and we call it the center space of $\sigma$.
We denote $\PP^c_\sigma =\PP E^c_\sigma$, its projective space.

\begin{rema}Consider the open and dense set of vector fields whose singular points are all hyperbolic.  In this open set the singularities depend continuously on the field.
Then for every singular point $\sigma$, the projective center space $\PP^c_\sigma$ varies upper semi continuously, and in particular the dimension  $dim E^c_\sigma$ varies upper semi-continuously.
As it is a non-negative integer, it is locally minimal and locally constant on a  open and dense subset.

We will say that such a singular point has \emph{locally minimal center space}.
\end{rema}

We prove
\begin{theo}\label{t.nioki} Let $X$ be a vector field on a closed manifold, whose singular points are
\begin{itemize}
\item  hyperbolic \item with locally minimal center spaces \item the finest dominated splitting of the center spaces is in one or two dimensional subspaces.
\end{itemize} Then for every
$\sigma\in Sing X$, every hyperbolic structure on $\widetilde{\La(C(\sigma))}$ extends on $B(C(\sigma))$.
\end{theo}

\subsection{The (multi)singular hyperbolicity and the star flows}

We say that a vector field $X$, whose singularities are all hyperbolic is \emph{multisingular hyperbolic} if every chain recurrence class is multisingular hyperbolic.
Recall that a vector field is a \emph{star flow} if it belongs to the $C^1$-interior of the vector fields whose periodic orbits are all hyperbolic.

\begin{rema}
 \begin{itemize}
  \item If $C$ is a non-singular chain recurrence class which is multisingular hyperbolic then it is uniformly hyperbolic and therefore is a hyperbolic basic set and a  homoclinic class.
  \item If $C$ is a chain recurrence class which is multisingular hyperbolic then $X$ is multisingular hyperbolic on a small filtrating neighborhood of $C$.
 \end{itemize}

\end{rema}

One check easily:
\begin{lemm} If $X$ is multisingular hyperbolic, then $X$ is a star flow.
\end{lemm}

We will show that, conversely

\begin{theo}\label{t.converse} There is a $C^1$-open and dense subset $\cU$ of $\cX^1(M)$ so that is $X\in \cU$ is a star flow then the chain recurrent set $\cR(X)$ is contained in the
union of finitely many pairwise disjoint filtrating regions in which $X$ is multisingular hyperbolic.
\end{theo}

Indeed we will get a more precise result: our notion of singular hyperbolic structures allows \emph{a priori} many possible choices of
reparametriztions cocycles.
Here one can fix \emph{a priori} the reparametrizations cocycle, by using the works of  \cite{GSW} on star flow.

More precisely, according to \cite{GSW} for an open and dense subset in the set of star flows, one has the following properties:
\begin{enumerate}
\item any chain recurrence class $C$ admits a (unique) dominated splitting $\cN= E\oplus F$ for the extended linear Poincar\'e flow on $\widetilde{\La(C)}$ which is the limit of
the hyperbolic splitting of the periodic orbits for the $C^1$-nearby flows.

\item every  singularity $\sigma$ with non-trivial
chain recurrence class has only two possibilities
\begin{itemize}
 \item either the stable space  $E^s_\sigma$ has the same dimension as the bundle $E$ of the dominated splitting of the extended linear Poincar\'e flow over
 $\widetilde{\La(C(\sigma))}$
 (and thus $dim E^u_\sigma = dim(F)+1$).  We denote by $S_{E}\subset Sing(X)$ the set of such singular points.

 \item or  $dim E^u_\sigma = dim F$ and  $dim E^s_\sigma=dim E+ 1$.  We denote by $S_F$ the set of singular points satisfying this second possibility.
\end{itemize}
\end{enumerate}

In particular, the indices of the singularities in a  same chain recurrence classes may differ at most of $1$.

Then one consider the reparametrization cocycles $h^t_E$ and $h^t_F$ defined as

$$ h^t_E= \prod_{\sigma\in S_E} h^t_\sigma \mbox{ and } h^t_F= \prod_{\sigma\in S_F} h^t_\sigma .$$

Now, Theorem~\ref{t.converse} is a straightforward corollary of:

\begin{theo}\label{t.star-multi}There is a $C^1$  open and dense subset $\cU$ of the open set of star flows so that, for any $X$ in $\cU$ every chain
recurrence class admits a dominated splitting $\cN=E\oplus_{_\prec}F$ for the extended linear Poincar\'e flow $\psi^t_\cN$ over $B(C)$ and so that  the reparametrized
flow
$$(h^t_E \psi^t_\cN|_E, h^t_F \psi^t_\cN|_F)$$
is uniformly hyperbolic.

In other words,  $X$ is multisingular hyperbolic and its reparametrization cocycles are $(h^t_E,h^t_F)$.

\end{theo}

\begin{rema} If all the singular points in a chain recurrence class $C$  have the same index, that is, if $S_E\cap C$ or $S_F\cap C $ is empty, then the multisingular
hyperbolicity is  the singular hyperbolicity as in \cite{GSW}.
\end{rema}

The proof of Theorem~\ref{t.star-multi} follows closely the proof in \cite{GSW} that star flows with only singular points of the same index are singular hyperbolic.

\begin{ques} Can we remove the generic assumption, at least in dimension $3$, in Theorem~\ref{t.converse}?  In other word, is it true that, given any star $X$ flow
(for instance on a $3$-manifold)
every chain recurrence class of $X$ is multisingular hyperbolic?
\end{ques}

\section{Basic definitions and preliminaries}

\subsection{Chain recurrent set}\label{ss.chainrecurrent}
The following notions and theorems are due to Conley \cite{Co} and they can be found in several other references (for example \cite{AN}).
\begin{itemize}
  \item We say that a pair of sequences $\set{x_i}_{0\leq i\leq k}$ and  $\set{t_i}_{0\leq i\leq k-1}$, $k\geq 1$,  are an \emph{ $\varepsilon$-pseudo orbit from $x_0$ to $x_k$} for a flow $\phi$,
  if for every $0\leq i \leq k-1$ one has
  $$ t_i-t_{i-1}\geq 1 \mbox{ and }d(x_{i+1},\phi^{t_i}(x_i))<\varepsilon.$$

  \item A compact invariant set $\Lambda$ is called \emph{chain transitive} if  for any $\varepsilon > 0$, and for any $x, y \in\Lambda$
there is an $\varepsilon$-pseudo orbit from $x$ to $y$.
\item  We say that $x\in M$ is \emph{chain recurrent} if  for every $\varepsilon>0$, there is an $\varepsilon$-pseudo orbit from $x$ to $x$. We call the set of chain recurrent points, the\emph{ chain recurrent set}
and we note it $\mathfrak{R}(M)$.
  \item We say that   $x, y \in \mathfrak{R}(M)$ are chain related if,  for every $\varepsilon>0$, there are $\varepsilon$-pseudo orbits form $x$ to $y$ and from $y$ to $x$. This is an equivalence relation.
 The equivalent classes of this equivalence relation are called \emph{ chain recurrence classes}.
\end{itemize}
\begin{figure}[htb]
\begin{center}
\includegraphics[width=0.63\linewidth]{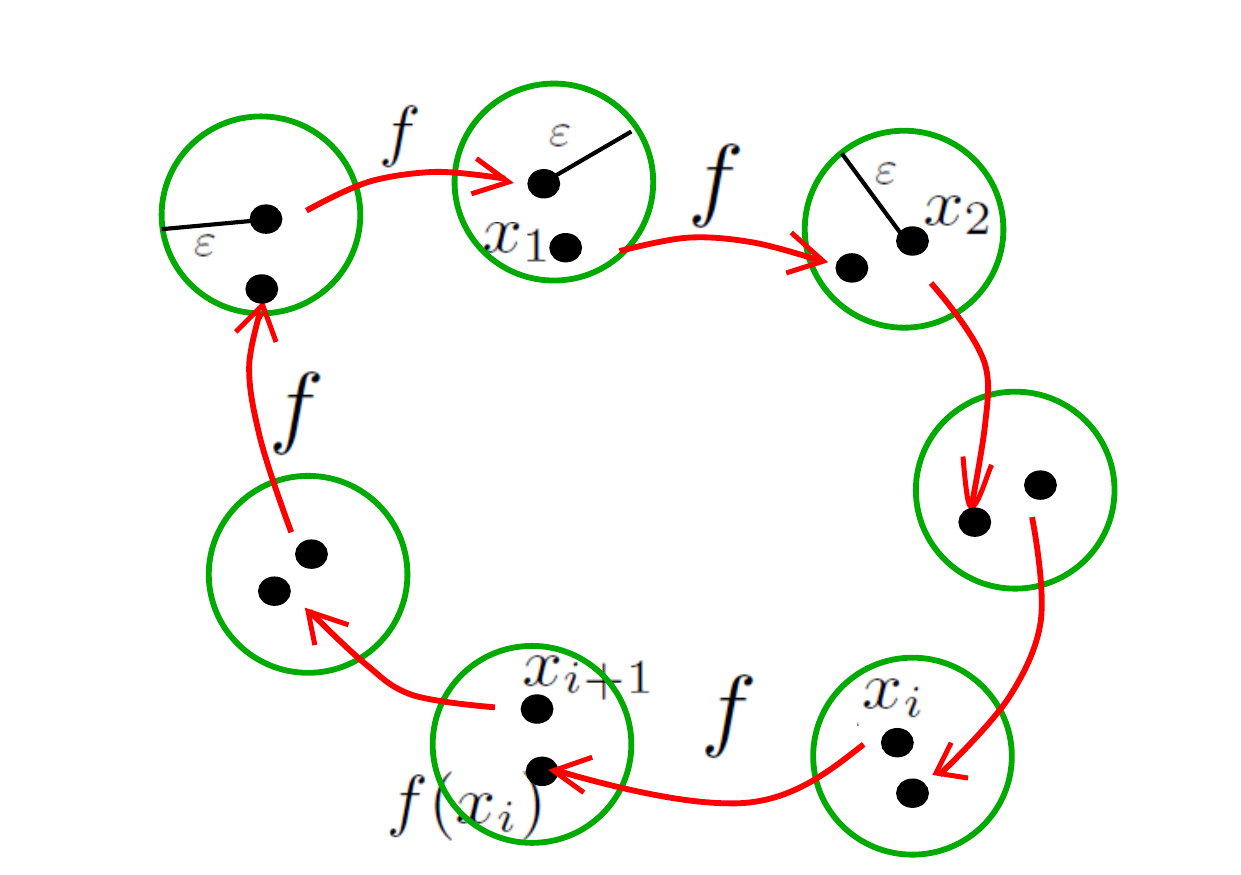}
\end{center}
\caption{an $\varepsilon$-pseudo orbit }
\end{figure}
\begin{defi}
\begin{itemize} \item An \emph{attracting region} (also called \emph{trapping region} ) is a compact set $U$ so that $\phi^t(U)$ is contained in the interior
of $U$ for every $t>0$. The maximal invariant set in an attracting region is called an \emph{attracting  set}.  A repelling region is an attracting region for $-X$, and the maximal invariant set is called a repeller.

\item A \emph{filtrating region} is the intersection of an attracting region with a repelling region.

\item Let $C$ be a chain recurrent class of $M$ for the flow $\phi$.
A \emph{filtrating neighborhood } of $C$ is a (compact) neighborhood which is a filtrating region.
\end{itemize}
\end{defi}
\begin{figure}[htb]
\begin{center}
\includegraphics[width=0.63\linewidth]{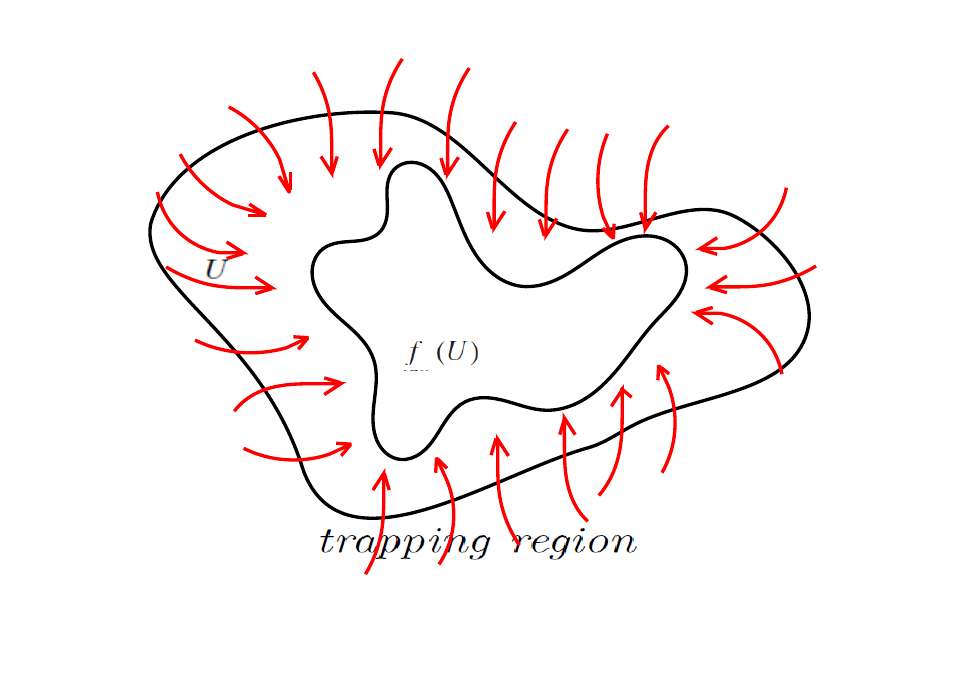}
\end{center}
\caption{ A trapping region or attracting region. }
\end{figure}
\begin{defi}
Let $\{\phi^t\}$  by a flow on a Riemannian manifold $M$.
A \emph{complete Lyapunov function} is a continuous function $\mathcal{L}:M\to\RR$ such that
\begin{itemize}
  \item $\mathcal{L}(\phi^t(x))$ is decreasing for $t$ if $x\in M\setminus\mathfrak{R}(M)$
  \item Two points $x, y \in M$ are chain related if and only if $\mathcal{L}(x)=\mathcal{L}(y)$
  \item  $\mathcal{L}(\mathfrak{R}(M))$ is nowhere dense.
\end{itemize}
\end{defi}

Next result is called the fundamental theorem of dynamical systems by some authors:

\begin{theo}[Conley \cite{Co}]

Let $X$  be a $C^1$ vector field on a compact manifold $M$. Then  its flow $\{\phi^t\}$ admits a complete Lyapunov function.
\end{theo}

Next corollary will be often used in this paper:

\begin{coro}\label{c.filtrating}
Let $\phi$  be a $C^1$-vector field on a compact manifold $M$. Every chain class $C$  of $X$ admits a  basis of filtrating neighborhoods, that is, every neighborhood
of $C$ contains a filtrating neighborhood of $C$.
\end{coro}

\begin{lemm}[Connecting lemma]\label{l.contecting} \cite{BC}.
Given $\flow$ induced by a vector field $X\in \mathcal{X}^1(M)$ such that all periodic orbits of $X$ are hyperbolic. For any $C^1$ neighborhood $\cU$ of $X$ and $x,y\in M$ if $y$ is chain attainable from $x$, then there exists $Y\in U$ and $t>0$ such that $\flow^Y(x)= y$. Moreover the following holds: For any $k\geq 1$, let $\{x_{i,k},t_{i,k}\}_{i=0}^{n_k}$ be an $(1/k,T)$-pseudo orbit from $x$ to $y$ and denote by
$$\Delta_k = \bigcup_{i=0}^{n_k-1} \phi_{[0,t_{i,k}]}(x_{i,k}).$$
Let $\Delta$ be the upper Hausdorff limit of $\Delta_k$. Then for any neighborhood $U$ of  $\Delta$, there exists $Y\in U$ with $Y = X$ on $M\setminus U$ and $t>0$ such that $\flow^Y(x) = y$.
\end{lemm}

For a generic vector field $X\in \mathcal{X}^1(M)$ we also have:

\begin{theo}\label{ConConLem} \cite{C}
There exists a $G_{approx}\subset \mathcal{X}^1(M)$ a generic set such that for every $X\in G_{approx}$ and for every $C$ a chain recurrence class there exists a sequence of periodic orbits $\gamma_n$ which converges to $C$ in the Hausdorff topology.
\end{theo}

\subsection{Linear cocycle}
Let  $\phi=\{\phi^t\}_{t\in\RR}$ be a topological flow on a compact metric space  $K$.
A \emph{linear cocycle over   $(K,\phi)$}   is a continuous map   $A^t\colon E\times \RR\to E$
defined by $$A^t(x, v) = (\phi^t(x), A_t(x)v)\,,$$ where
\begin{itemize}
\item  $\pi\colon E\to K$ is a $d$ dimensional  linear bundle over $K$;
\item  $A_t:(x,t)\in K\times \RR \mapsto GL(E_x,E_{\phi^t(x)})$ is a continuous map  that satisfies the \emph{cocycle relation } :
 $$A_{t+s}(x)=A_t(\phi^s(x))A_s(x),\quad  \mbox{for any } x\in K \mbox{ and }t,s\in\RR  $$

\end{itemize}

Note that $\cA=\{A^t\}_{t\in\RR}$ is a flow on the space $E$ which projects on $\phi^t$.
$$\begin{array}[c]{ccc}
E &\stackrel{A^t}{\longrightarrow}&E\\
\downarrow&&\downarrow\\
K&\stackrel{\phi^t}{\longrightarrow}&K
\end{array}$$

If $\Lambda\subset K$ is a $\phi$-invariant subset,  then $\pi^{-1}(\Lambda)\subset E$ is $\cA$-invariant, and  we  call
\emph{the restriction of $\cA$ to $\Lambda$}  the restriction of $\{A^t\}$ to $\pi^{-1}(\Lambda)$.

 \subsection{Hyperbolicity, dominated splitting on linear cocycles}\label{ss.hyp}

 \begin{defi}
 Let $\phi$ be a topological flow on a compact metric space $\La$. We consider a vector bundle  $\pi\colon E\to \La$ and
   a linear cocycle $\cA=\{ A^t\}$ over $(\La,X)$.

 We say that $\cA$ admits a \emph{dominated splitting over $\Lambda$} if
 \begin{itemize}\item there exists a splitting $E=E^1\oplus\dots\oplus E^k$ over $\lambda$ into $k$ subbundles
 \item  the dimension of the subbundles is constant, i.e. $dim(E^i_x)=dim(E^i_y)$ for all $x,y\in\Lambda$ and $i\in \set{1\dots k}$,
     \item the splitting is invariant, i.e. $A^t(x)(E^i_x)=E^i_{\phi^t(x)}$ for all $i\in \{1\dots k\}$,

\item there exists a $t>0$ such that for every $x\in \Lambda$ and any pair of non vanishing vectors $v\in E^i_x$ and $u\in E^j_x$, $i<j$ one has
 \begin{equation}\label{e.dom}
 \frac{\norm{A^t(u)}}{\norm u}\leq \frac{1}{2}\frac{\norm{ A^t(v)}}{\norm v}
 \end{equation}

 We denote $E^1\oplus_{_\prec}\dots\oplus_{_\prec} E^k$.
 the splitting is \emph{$t$-dominated}.
\end{itemize}
 \end{defi}

 A classical result (see for instance \cite[Appendix B]{BDV})  asserts that the bundles of a dominated splitting are always continuous.
 A given cocycle may admit several dominated
  splittings.  However, the dominated splitting is unique if one prescribes the dimensions $dim(E^i)$.

Associated to the dominated splitting we define a family of cone fields $C^{iu}_a$ around each space $E^i\oplus\dots\oplus E^k$ as follows.
Let  us write the vectors $v\in E$ as $v=(v_1,v_2)$ with $v_1\in E^1\oplus\dots\oplus E^{i-1}$ and $v_2\in E^i\oplus\dots\oplus E^k$.
Then the conefield $C^{iu}_a$ is the set $$C^{iu}_a=\set{v=(v_1,v_2) \text{ such that } \norm{v_1}<a\norm{v_2} }.$$
These  are called the\emph{ family of unstable conefields} and the domination gives us that they are strictly invariant for times larger than $t$: i.e. the cone $C^{iu}_a$ at $T_xM$ is taken by $A^t$ to the interior of the cone $C^{iu}_a$ at $T_{\phi^t x}M$.

Analogously we define the \emph{stable family of conefields $C^{is}_a$} around
$E^1\oplus\dots\oplus E^i$ and the domination gives us that they are strictly invariant for negative times smaller than $-t$.

\begin{figure}[htb]\label{f.cone}
\begin{center}
\includegraphics[width=0.8\linewidth]{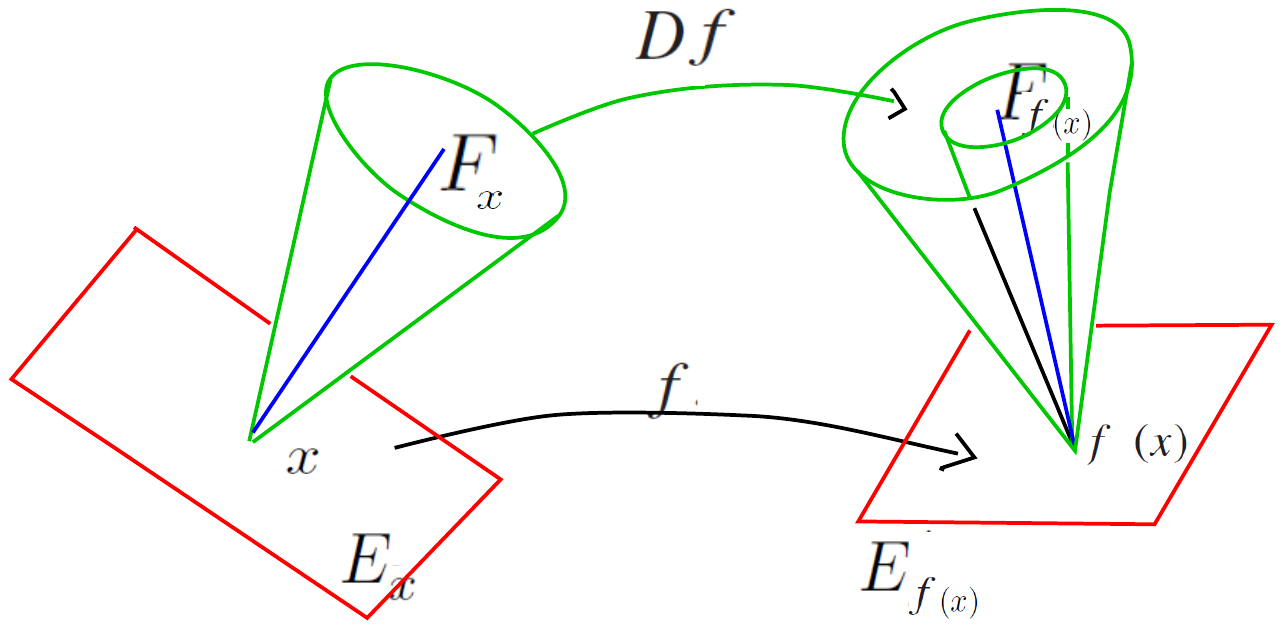}
\end{center}
\end{figure}

 One says that one of the bundle $E^i$ is \emph{(uniformly) contracting} (resp. \emph{expanding}) if there is $t>0$ so that for every
 $x\in\La$ and every non vanishing vector   $u\in E^i_x$ one has $\frac{\norm{A^t(u)}}{\norm u}<\frac 12$
 (resp. $\frac{\norm{A^{-t}(u)}}{\norm u}<\frac 12$). In both cases one says that $E^i$ is \emph{hyperbolic}.

 Notice that if $E^j$ is contracting (resp. expanding) then the same holds for any $E^i$, $i<j$  (reps. $j<i$) as a consequence of the domination.

\begin{defi}
 We say that the linear cocycle  $\cA$ is \emph{hyperbolic over $\Lambda$} if
  there is a dominated splitting $E=E^s\oplus_{_\prec} E^u$ over $\Lambda$ into $2$  hyperbolic subbundles so that $E^s$ is uniformly contracting and $E^u$ is
  uniformly expanding.

  One says that $E^s$ is the \emph{stable bundle}, and $E^u$ is the \emph{unstable bundle}.
 \end{defi}

 The existence of a dominated splitting or of a hyperbolic structure is an open property in the following sense,

 \begin{prop}\label{p.robustcocycle} Let $K$ be a compact metric space, $\pi\colon E\to K$ be a $d$-dimensional  vector bundle, and $\cA$ be a linear cocycle over $K$.
 Let $\La_0$ be a $\phi$-invariant compact set. Assume that the restriction of $\cA$ to $\La_0$ admits a dominated splitting
 $E^1\oplus_{{_\prec}_t}\dots\oplus_{{_\prec}_t} E^k$, for some $t>0$.

 Then there is a compact neighborhood $U$ of $\La_0$ with the following property. Let $\La=\bigcap_{t\in\RR} \phi^t(U)$ be the maximal invariant set  of $\phi$ in $U$.
 Then the dominated splitting admits a unique extension as a $2t$-dominated splitting over $\La$.  Furthermore if one of the subbundle $E^i$ is hyperbolic over $\La_0$ it is still hyperbolic
 over $\La$.

 As a consequence, if $\cA$ has a hyperbolic structure over $\La_0$ then (up to shrink $U$ if necessary) it also has a hyperbolic structure over $\La$.
 \end{prop}

 \subsection{Robustness of hyperbolic structures}\label{ss.robust}

 The aim of this section is to explain that Proposition~\ref{p.robustcocycle} can be seen as a robustness property.

 Let $M$ be a  manifold and $\phi_n$ be a sequence of flows in $M$ tending to $\phi_0$ as $n\to +\infty$, in the $C^0$-topology on compact subsets: for any compact set
 $K\subset M$  and any $T>0$, the restriction of $\phi^t_n$ to $K$, $t\in[-T,T]$, tends uniformly (in $x\in K$ and $t\in[-T,T]$) to $\phi^t_0$.

 Let $\La_n$ be compact $\phi_n$-invariant subsets of $M$, and assume that the upper limit of the $\La_n$ for the Hausdorff topology is contained in $\La_0$:
 more precisely,  any neighborhood of $\La_0$ contains all but finitely many of the $\La_n$.  Let us present another way to see this property:

 Consider the subset $\cI=\{0\}\cup\{\frac 1n,n\in\NN\setminus\{0\}\}\subset \RR$ endowed with the induced topology.
 Consider $M_\infty= M\times \cI$.  Let $\La_\infty$ denote
 $$\La_\infty=\La_0\times \{0\}\cup\bigcup_{n>0}\La_n\times\{\frac1n\}\subset M_\infty.$$
 With this notation, the upper limit of the $\La_n$ is contained in $\La_0$ if and only if $\La_\infty$ is a compact subset.

 Let $\pi\colon E\to M$ be a vector bundle. We denote $E_\infty=E\times \cI$ the vector bundle $\pi_\infty\colon E\times \cI\to M\times \cI$.
 We denote by $E_\infty|_{\La_\infty}$ the restriction of $E_\infty$ on the compact subset $\La_\infty$.

Assume now  that $\cA_n$ are linear cocycles over the restriction of $E$ to $\La_n$.
 We denote by $\cA_\infty$ the map defined on the restriction  $E_\infty|_{\La_\infty}$ by:
 \begin{eqnarray*}
 A^t_\infty (x,0)&=&(A^t_0(x), 0), \mbox{ for } (x,0)\in \La_0\times\{0\} \mbox{ and }\\A^t_\infty (x,\frac 1n)&=&(A^t_n(x), \frac1n),
 \mbox{ for } (x,\frac1n)\in \La_n\times\{\frac1n\}.\end{eqnarray*}

\begin{defi}\label{d.robust} With the notation above, we say that the family of cocycles $\cA_n$ tends to $\cA_0$ as $n\to \infty$ if the map $\cA_\infty$ is continuous, and therefore is
a linear cocycle.
\end{defi}

As a consequence of Proposition~\ref{p.robustcocycle} we get:
\begin{coro}\label{c.robust} Let $\pi\colon E\to M$ be a linear cocycle over a manifold $M$ and let $\phi_n$ be a sequence of flows on $M$
converging to $\phi_0$ as $n\to \infty$. Let $\La_n$ be a sequence of $\phi_n$-invariant compact subsets so that the upper limit of the $\La_n$, as $n\to\infty$,
is contained in $\La_0$.

Let $\cA_n$ be a sequence of linear cocycles over $\phi_n$ defined  on the restriction of $E$ to $\La_n$. Assume that $\cA_n$ tend to $\cA_0$ as $n\to\infty$.

Assume that $\cA_0$ admits a dominated splitting $E=E^1\oplus_{_\prec}\dots\oplus_{_\prec} E^k$ over $\La_0$.  Then, for any $n$ large enough, $\cA_n$ admits a dominated splitting with
the same number of sub-bundles and the same dimensions of the sub-bundles.  Furthermore, if $E^i$ was hyperbolic (contracting or expanding) over $\La_0$ it is still hyperbolic
(contracting or expanding, respectively) for $\cA_n$ over $\La_n$.
\end{coro}
The proof just consist in applying Proposition~\ref{p.robustcocycle} to a neighborhood of $\La_0\times\{0\}$ in $\La_\infty$.

\subsection{Reparametrized cocycles, and hyperbolic structures}
Let $\cA=\{A^t(x)\}$ and $\cB=\{B^t(x)\}$ be two linear cocycles on the same linear bundle $\pi\colon \cE\to \La$ and over the  same flow $\phi^t$ on a compact invariant set $\La$ of a manifold $M$. We say that
$\cB$ is  a \emph{reparametrization} of $\cA$ if there is a continuous map $h=\{h^t\}\colon \La\times \RR\to (0,+\infty)$ so that for every $x\in\La$ and $t\in\RR$ one has
$$B^t(x)=h^t(x) A^t(x).$$
The reparametrizing map $h^t$  satisfies the cocycle relation
$$h^{r+s}(x)=h^r(x)h^s(\phi^r(x)),$$
and is called a \emph{cocycle}.

One easily check the following lemma:
\begin{lemm}
Let $\cA$ be a linear cocycle and $\cB$ be a reparametrization of $\cA$.  Then  any dominated splitting for $\cA$ is a dominated splitting for $\cB$.
\end{lemm}

\begin{rema} \begin{itemize}\item If $h^t$ is a cocycle, then for any $\alpha\in\RR$  the power $(h^t)^\alpha: x\mapsto (h^t(x))^\alpha$ is a cocycle.
\item If $f^t$ and $g^t$ are cocycles then $h^t=f^t\cdot g^t$ is a cocycle.
\end{itemize}
\end{rema}

A cocycle $h^t$ is called a \emph{coboundary} if there is a continuous function $h\colon \La\to (0,+\infty)$ so that

$$h^t(x)= \frac{h(\phi^t(x))}{h(x)}.$$

A coboundary cocycle in uniformly bounded.
Two cocycles $g^t$, $h^t$ are called \emph{cohomological} if $\frac{g^t}{h^t}$ is a coboundary.

\begin{rema}\label{r.cohomological}
The cohomological relation is an equivalence relation among the cocycle and is compatible with the product: if $g^t_1$ and $g^t_2$ are cohomological
and $h^t_1$ and $h^t_2$ are cohomological then $g^t_1h^t_1$ and $g^t_2 h^t_2$ are cohomological.
\end{rema}

\begin{lemm}Let $\cA=\cA^t$ be a linear cocycle, and $h=h^t$ be a cocycle which is bounded.  Then
$\cA$ is uniformly contracted (resp. expanded) if and only if the  cocycle $\cB=h\cdot \cA$ is uniformly contracted (resp. expanded).
\end{lemm}

As a consequence one gets that the hyperbolicity of a reparametrized cocycle  only depends on the cohomology class of the reparametrizing cocycle:
\begin{coro}
if $g$ and $h$ are cohomological then $g\cdot \cA$ is hyperbolic if and only if  $h\cdot \cA$ is hyperbolic.
\end{coro}

\section{The extended linear Poincar\'e flow}\label{s.extended}
 \subsection{The linear Poincar\'e flow}\label{ss.Poincare}

 Let $X$ be a $C^1$ vector field on a compact manifold $M$.  We denote by $\phi^t$ the flow of $X$.

 \begin{defi} The \emph{normal bundle} of $X$ is the vector sub-bundle $N_X $ over $M\setminus Sing(X)$ defined as follows: the fiber $N_X(x)$ of $x\in M\setminus Sing(X)$ is
 the quotient space of $T_xM$ by the vector line $\RR.X(x)$.
 \end{defi}
 Note that, if $M$ is endowed with a Riemannian metric, then $N_X(x)$ is canonically identified with the orthogonal space of $X(x)$:
 $$N_X=\{(x,v)\in TM, v\perp X(x)\} $$

Consider $x\in M\setminus Sing(X)$ and $t\in \RR$.  Thus $D\phi^t(x):T_xM\to T_{\phi^t(x)}M$ is a linear automorphism mapping $X(x)$ onto $X(\phi^t(x))$. Therefore
$D\phi^t(x)$ passes to the quotient as an linear automorphism $\psi^t(x)\colon N_X(x)\to N_X(\phi^t(x))$:

$$\begin{array}[c]{ccc}
T_xM&\stackrel{D\phi^t}{\longrightarrow}&T_{\phi^t(x)}M\\
\downarrow&&\downarrow\\
N_X(x)&\stackrel{\psi^t}{\longrightarrow}&N_X(\phi^t(x))
\end{array}$$
where the vertical arrow are the canonical projection of the tangent space to the normal space parallel to $X$.
\begin{figure}[htb]
\begin{center}
\includegraphics[width=0.63\linewidth]{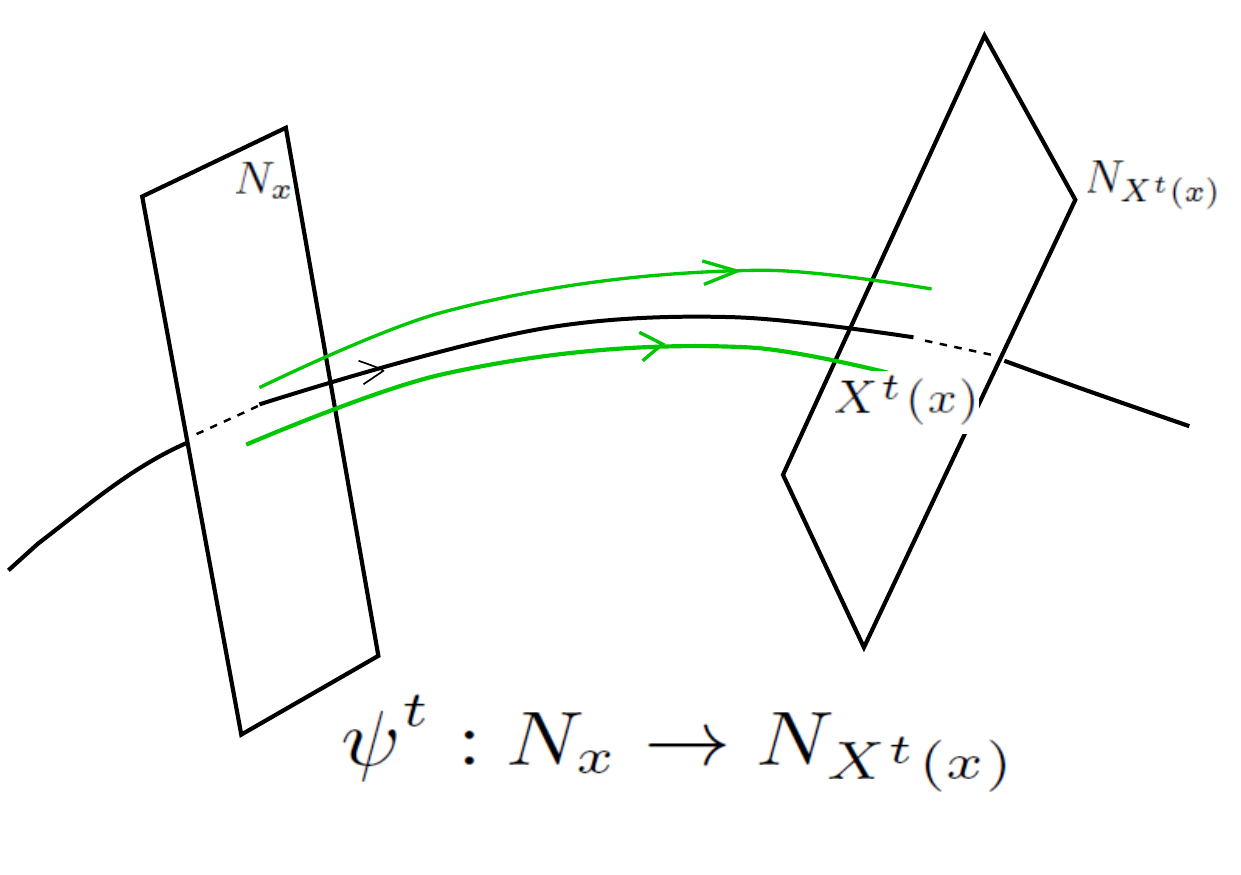}
\end{center}
\caption{\label{f.nonsingular}  $\psi^t$  is the differential of the holonomy or Poincar\'e map}
\end{figure}

 \begin{prop}\label{p.lpf} Let $X$ be a $C^1$ vector field on a manifold $M$, and $\La$ be a compact invariant set of $X$.  Assume that $\La$ does not contained any singularity  of $X$.
 Then $\La$ is hyperbolic if and only if the linear Poincar\'e flow over $\La$ is hyperbolic.
 \end{prop}

 Notice that the notion of dominated splitting for non-singular flow is sometimes better expressed in term of Linear Poincar\'e flow: for instance,
 the linear Poincar\'e flow of a robustly transitive vector field
 always admits a dominated splitting, when the flow by itself may not admit any dominated splitting.
 See for instance  the suspension of the example in \cite{BV}.

 \subsection{The extended linear Poincar\'e flow}\label{ss.extended}

 We are dealing with singular flows and the linear Poincar\'e flow is not defined on the singularity of the vector field $X$.
 However we can extend the linear Poincar\'e  to a
 flow, as  in \cite{GLW}, on a larger set, and for which the singularities of $X$ do not play a specific role. We call this the \emph{extended linear Poincar\'e flow}.

 This flow will be a linear co-cycle define on some linear bundle over a manifold, that we define now.

 \begin{defi}Let $M$ be a  $d$ dimensional manifold.
 \begin{itemize}
 \item We call \emph{the projective tangent bundle of $M$}, and denote by $\Pi_\PP\colon \PP M\to M$, the fiber bundle whose fiber $\PP_x$ is
 the projective space of the tangent space $T_xM$: in other words, a point $L_x\in \PP_x$ is a $1$-dimensional vector subspace of $T_xM$.
  \item We call \emph{the tautological bundle of $\PP M$}, and we  denote by $\Pi_\cT\colon \cT M\to \PP M$, the $1$-dimensional vector bundle over $\PP M$
  whose fiber $\cT_{L}$, $L\in \PP M$,  is the line $L$ itself.
  \item We call \emph{normal bundle of $\PP M $} and we denote by $\Pi_\cN\colon \cN \to \PP M$, the $(d-1)$-dimensional vector bundle over $\PP M$ whose fiber $\cN_{L}$ over
  $L\in \PP_x$
  is the quotient space $T_x M/L$.

  If we endow $M$ with riemannian metric, then $\cN_L$  is identified with the orthogonal hyperplane of $L$ in $T_xM$.
 \end{itemize}
 \end{defi}

 Let $X$ be a $C^r$ vector field on a compact manifold  $M$, and $\phi^t$ its flow. The natural actions of the derivative of $\phi^t$ on $\PP M$, $\cT M$ and $\cN$ define
 flows on these manifolds.  More precisely, for any $t\in \RR$,

 \begin{itemize}
  \item We denote by $\phi_{\PP}^t\colon\PP M\to \PP M$ the flow defined by $$\phi_{\PP}^t(L_x)= D\phi^t(L_x)\in \PP_{\phi^t(x)}.$$

  \item  We denote by $\phi_{\cT}^t\colon\cT M\to \cT M$ the topological flow  whose restriction to a fiber $\cT_L$ is the linear automorphisms onto
  $\cT_{\phi^t_{\PP}(L)}$ which is the restriction of $D\phi^t$ to the line $\cT_L$.

  \item We denote by $\psi_{\cN}^t\colon\cN\to \cN$ the flow whose restriction to a fiber $\cN_L$, $L\in \PP_x$,
  is the linear automorphisms onto
  $\cN_{\phi^t_{\PP}(L)}$ defined as follows: $D\phi^t(x)$ is a linear automorphism from $T_xM$ to $T_{\phi^t(x)}M$, which maps the line $\cT_L\subset T_xM$
  onto the line $\cT_{\phi^t_{\PP}(L)}$.  Therefore it passe to the quotient in the announced linear automorphism.
  $$\begin{array}[c]{ccc}
T_xM &\stackrel{D\phi^t}{\longrightarrow}&T_{\phi^t(x)}M\\
\downarrow&&\downarrow\\
\cN_L&\stackrel{\psi^t_{\cN}}{\longrightarrow}&\cN_{\phi^t_{\PP}(L)}
\end{array}$$

 \end{itemize}

 Note that $\phi^t_\PP$, $t\in \RR$ defines a flow on $\PP_M$  which is a cocycle over $\phi^t$ whose action on the fibers is by projective maps.

 The one parameter families  $\phi^t_\cT$ and $\psi^t_\cN$ define flows on $\cT M$ and $\cN$, respectively, are linear cocycles over $\phi^t_\PP$.
 We call $\phi^t_\cT$ the \emph{tautological flow}, and we call $\psi^t_\cN$ the \emph{extended linear Poncar\'e flow}.
 We can summarize  by the following diagrams: 

  \begin{multicols}{2}

$$\begin{array}[c]{ccc}
\cN &\stackrel{\psi^t_{\cN}}{\longrightarrow}&\cN \\
\downarrow&&\downarrow\\
\PP M &\stackrel{\phi_{\PP}^t}{\longrightarrow}&\PP M\\
\downarrow&&\downarrow\\
M&\stackrel{\phi^t}{\longrightarrow}&M
\end{array}$$

$$\begin{array}[c]{ccc}
\cT M&\stackrel{\phi^t_{\cT}}{\longrightarrow}&\cT M\\
\downarrow&&\downarrow\\
\PP M &\stackrel{\phi_{\PP}^t}{\longrightarrow}&\PP M\\
\downarrow&&\downarrow\\
M&\stackrel{\phi^t}{\longrightarrow}&M
\end{array}$$
\end{multicols}

\begin{rema} The extended linear Poincar\'e flow is really an extension of the linear Poincar\'e flow defined in the previous section; more precisely:

 Let $S_X\colon M\setminus Sing(X)\to \PP M$ be the section of the projective bundle defined as $S_X(x)$ is the line $\langle X(x)\rangle\in \PP_x$ generated by $X(x)$.Then

\begin{itemize}\item the fibers $N_X(x)$ and $\cN_{S_X(x)}$ are canonically identified and
\item the linear automorphisms $\psi^t\colon N_X(x)\to N_X(\phi^t(x))$ and $\psi^t_\cN\colon \cN_{S_X(x)}\to \cN_{S_X(\phi^t(x))}$ are equal (under the identification of the fibers)
\end{itemize}
\end{rema}

\subsection{Maximal invariant set and lifted maximal invariant set}

Let $X$ be a vector field on a manifold $M$ and $U\subset M$ be a compact region.  The \emph{maximal invariant set} $\La=\La_U$ of $X$ in $U$ is the intersection
$$\La(X,U)=\bigcap_{y\in \RR} \phi^t(U).$$

 We say that a compact $X$-invariant set $K$ is \emph{locally maximal} if there exist an open neighborhood $U$ of $K$ so that $K=\La(X,U)$.

\begin{defi}We call \emph{lifted maximal invariant set in $U$}, and we denote by $\La_{\PP,U}\subset \PP M$
(or simply $\La_\PP$ if one may omit the dependence on $U$), the closure of the set of lines $\langle X(x)\rangle$ for regular
 $x\in \La_U $:

 $$\La_{\PP,U}=\overline{S_X(\La_U\setminus Sing(X)}\subset \PP M,$$
 where $S_X\colon M\setminus Sing(X) \to \PP M$ is the section defined by $X$.
 \end{defi}



\section{The extended maximal invariant set}

\subsection{Strong stable, strong unstable and center spaces associated to a hyperbolic singularity.}
Let $X$ be a vector field and $\sigma\in Sing(X)$ be a hyperbolic singular point of $X$.
Let
 $\lambda^s_k\dots\lambda^s_2<\lambda^s_1<0<\lambda^u_1<\lambda^u_2\dots \lambda^u_l$ be the Lyapunov exponents of $\phi_t$ at $\sigma$ and let
 $E^ s_k\oplus_{_<}\cdots E^s_2\oplus_{_<}E^s_1\oplus_{_<}E^u_1\oplus_{_<}E^u_2\oplus_{_<}\cdots \oplus_{_<}E^s_l$ be the corresponding (finest) dominated
 splitting over $\sigma$.

 A subspace $F$ of $T_\sigma M$ is called a \emph{center subspace} if it is of one of the possible form below:
 \begin{itemize}
 \item Either $F=E^ s_i\oplus_{_<}\cdots E^s_2\oplus_{_<}E^s_1$
 \item Or $F=E^u_1\oplus_{_<}E^u_2\oplus_{_<}\cdots \oplus_{_<}E^s_j$
 \item Or else $F=E^ s_k\oplus_{_<}\cdots E^s_2\oplus_{_<}E^s_1\oplus_{_<}E^u_1\oplus_{_<}E^u_2\oplus_{_<}\cdots \oplus_{_<}E^s_l$
 \end{itemize}

 A subspace  of $T_\sigma M$ is called a \emph{strong stable space}, and we denote it  $E^{ss}_{i}(\sigma)$,  is there in $i\in\{1,\dots, k\}$ such that:
 $$E^{ss}_{i}(\sigma)=E^ s_k\oplus_{_<}\cdots E^s_{i+1}\oplus_{_<}E^s_i$$

 A classical result from hyperbolic dynamics asserts that for any $i$ there is a unique infectively immersed manifold $W^{ss}_i(\sigma)$, called a
 \emph{strong stable manifold}
 tangent at $E^{ss}_i(\sigma)$ and invariant by the flow of $X$.

 We define analogously the \emph{strong unstable spaces} $E^{uu}_j(\sigma)$ and the \emph{strong unstable manifolds} $W^{uu}_j(\sigma)$ for $j=1,\dots ,l$.


\subsection{The lifted maximal invariant set and the singular points}\label{ss.centerspace}
The aim of this section is to add  to
 the lifted maximal invariant set $\La_{\PP,U}$, some set over the singular
 points in order to recover some upper semi-continuity properties. As mentioned in section \ref{MIS} we want to define a set, that is as small as possible, but that can be defined without any information of the perturbations of our vector field.

We define the \emph{escaping stable space} $E^{ss}_{\sigma,U}$ as the biggest strong stable space $E^{ss}_j(\sigma)$ such that the corresponding strong stable manifold
$W^{ss}_j(\sigma)$  is \emph{escaping}, that is:  $$\Lambda_{X,U}\cap W^{ss}_j(\sigma)=\{\sigma\}.$$

We define the \emph{escaping  unstable space} analogously.

We define the \emph{central space $E^c_{\sigma,U}$ of $\sigma$ in $U$}  the center space such that
$$T_\sigma M=E^{ss}_{\sigma,U}\oplus E^ c_{\sigma,U}\oplus E^ {uu}_{\sigma,U}$$

 We denote by
 $\PP^i_{\sigma,U}$ the projective space of $E^ i(\sigma,U)$ where $i=\set{ss,uu,c}$.

 \begin{lemm}\label{l.escaping} Let $U$ be a compact region and $X$  a vector field whose singular points are hyperbolic.
 Then, for any $\sigma\in Sing(X)\cap U$, one has :
 $$\La_{\PP,U}\cap\PP^{ss}_{\sigma,U}=\La_{\PP,U}\cap\PP^{uu}_{\sigma,U}=\emptyset.$$

 \end{lemm}
\begin{proof}  Suppose (arguing by contradiction) that $L\in \La_{\PP,U}\cap\PP^{ss}_{\sigma,U}$. This means that there exist a sequence $x_n\in \La_{X,U}\setminus Sing(X)$
converging to $\sigma$,   such that $L_{x_n}$ converge to $ L$, where
$L_{x_n}$ is the line $\RR X(x_n)\in \PP_{x_n}$.

We fix a small neighborhood $V$ of $\sigma$ endowed with local coordinates so that the vector field is very close to its linear part in these coordinates:
in particular, there is a small cone $C^{ss}\subset V$ around $W^{ss}_{\sigma,U}$
so that the complement of
this cone   is strictly invariant in the following sense: the positive orbit of a point out of $C^{ss}$ remains out of $C^{ss}$ until it leaves $V$.

For $n$ large
enough the points $x_n$ belong to $V$.

As $\RR X(x_n)$ tend to $L$, this implies that the point $p_n$, for $n$ large, are contained in the cone $C^{ss}$.

In particular, the point $x_n$ cannot belong to
$W^u(\sigma)$. Therefore they admits negative iterates $y_n= \phi^{-t_n}(x_n)$ with the following property.
\begin{itemize}
 \item $\phi^{-t}(x_n)\in V$ for all $t\in[0,t_n]$,
 \item $\phi^{-t_n-1}(x_n)\notin V$,
 \item $t_n\to +\infty$.

\end{itemize}
Up to consider  a subsequence one may assume that the point $y_n$ converge to a point $y$,
and one easily check that the point $y$ belongs to $W^s(\sigma)\setminus\sigma$. Furthermore
all the points $y_n$ belong to $\La_{X,U}$ so that $y\in \La_{X,U}$.

We conclude the proof by showing that $y$ belongs to $W^{ss}_{\sigma,U}$, which is a contradiction with the definition of the escaping strong stable manifold
$W^{ss}_{\sigma,U}$. If $y\notin W^{ss}_{\sigma,U}$ then it positive orbit arrive to $\sigma$ tangentially to the weaker
stable spaces: in particular, there is $t>0$ so that $\phi^t(y)$ does not belong to the cone $C^{ss}$.

Consider $n$ large, in particular $t_n$ is larger than $t$ and $\phi^t(y_n)$ is so close to $y$ that $\phi^t(y_n)$
does not belong to $C^{ss}$: this contradict the fact that $x_n=\phi^{t_n}(y_n)$ belongs to $C^{ss}$.

We have proved  $\La_{\PP,U}\cap\PP^{ss}_{\sigma,U}=\emptyset$; the proof that  $\La_{\PP,U}\cap\PP^{uu}_{\sigma,U}$ is empty is analogous.

\end{proof}

As a  consequence we get the following characterization of the central space of $\sigma$ in $U$:
\begin{lemm} Let $U$ de a compact region and $X$ a vector field whose singular points are hyperbolic. Then for any $\sigma\in Sing(X)\cap U$
the central space  $E^c_{\sigma,U}$   is the smallest center space containing $\La_{\PP,U}\cup \PP_\sigma M$.
\end{lemm}
\begin{proof}
 The proof that $E^c_{\sigma,U}$ contains $\La_{\PP,U}\cap \PP_\sigma$ is very similar to the end of the proof of Lemma~\ref{l.escaping} and we just sketch it: by
 definition of the strong escaping manifolds, they admit  an neighborhood of a fundamental domain which is disjoint from the maximal invariant set.  This implies that
 any point in $\La_{X,U}$ close to $\sigma$ is contained out of arbitrarily large cones around the escaping strong direction. Therefore the vector $X$ at these points is almost
 tangent to $E^c_{\sigma,U}$.

 Assume now for instance  that:
 \begin{itemize}
 \item $E^c_{\sigma,U}=E^s_i\oplus E^s_{i-1}\oplus\dots\oplus E^s_1\oplus E^u_1\oplus\dots \oplus E^u_j$: in particular $W^{ss}_{i+1}(\sigma)$ is the escaping strong
 stable manifold, and
  \item $\La_{\PP,U}\cap \PP_\sigma$ is contained in
 the smaller center space
 $$ E^s_{i-1}\oplus\dots\oplus E^s_1\oplus E^u_1\oplus\dots \oplus E^u_j.$$
 \end{itemize}
  We will show that the strong stable manifold $W^{ss}_{i}(\sigma)$ is escaping, contradicting the maximality  of the escaping strong stable manifold
  $W^{ss}_{i+1}(\sigma)$.
  Otherwise, there is $x\in W^{ss}_i(\sigma)\setminus\{\sigma)\cap \La_{X,U}$.  The positive orbit of $x$ tends to $\sigma$ tangentially to
  $E^s_k\oplus\dots \oplus E^s_i$ and thus
  $X(\phi^t(x))$ for $t$ large is almost tangent to $E^s_k\oplus\dots\oplus E^s_i$: this implies that $\La_{\PP,U}\cap \PP_\sigma$ contains at least a direction in
  $E^s_k\oplus\dots \oplus E^s_i$ contradicting the hypothesis.
\end{proof}

\begin{lemm}\label{l.lower} Let $U$ be a compact region. Given $\sigma$  a hyperbolic singular point in $U$,  the point $\sigma$ has a continuation $\sigma_Y$ for vector fields
$Y$ in a $C^1$-neighborhood of $X$. Then both  escaping strong stable and unstable spaces $E^{ss}_{\sigma_Y,U}$ and $E^{uu}_{\sigma_Y,U}$ depend lower semi-continuously on
$Y$.

As a consequence the central space $E^c_{\sigma_Y,U}$  of $\sigma_Y$ in $U$ for $Y$
depends upper semi-continuously on $Y$,  and the same happens for its projective space
$\PP^{ss}_{\sigma_Y,U}$.
\end{lemm}
\begin{proof} We will make only the proof for the escaping strong stable space, as the proof for the escaping strong unstable space is identical.

As $\sigma$ is contained in the interior of $U$, there is $\delta>0$ and a $C^1$-neighborhood $\cU$ of $X$  so that, for any $Y\in \cU$, one has:
\begin{itemize}
\item $\sigma$ has a hyperbolic continuation $\sigma_Y$ for $Y$;
 \item the finest dominated splitting of $\sigma_X$ for $X$ has a continuation for $\sigma_Y$ which is a dominated splitting (but maybe not the finest);
 \item the local stable manifold of size $\delta$ of $\sigma_Y$ in contained in $U$ and depends continuously on $Y$
 \item for any strong stable space $E^{ss}(\sigma)$ the corresponding local strong stable manifold $W^{ss}(\sigma_Y)$ varies continuously with $Y\in \cU$.
\end{itemize}

Let denote $E^{ss}$ denote the escaping strong stable space of $\sigma$ and $W^{ss}_\delta(\sigma)$ be the corresponding local strong stable manifold.
We fix a sphere $S_X$ embedded in the interior of  $W^{ss}_\delta(\sigma)$, transverse to $X$ and cutting every orbit in  $W^{ss}_\delta(\sigma)\setminus \sigma$.
By definition of escaping strong stable manifold, for every $x\in S_X$ there is $t(x)>0$ so that $\phi^{t(x)}(x)$ is not contained in $U$.

As $S_X$ is compact and the complement of $U$ is open, there is a finite family $t_i, i=0,\dots, k$,  an open covering $V_0,\dots, V_k$ and  a $C^1$-neighborhood
$\cU_1$ of $X$ so that, for every $x\in U_i$ and every $Y\in\cU_1$ the point $\phi^{t_i}_Y(x)$ does not belong to $U$.

For $Y$ in a smaller neighborhood $\cU_2$ of $X$, the union of the $V_i$ cover a sphere $S_Y\subset W^{ss}_\delta(\sigma_Y,Y)$ cutting every orbit in
$W^{ss}_\delta(\sigma_Y,Y)\setminus \sigma_Y$.

This shows that $W^{ss}_\delta(\sigma_Y,Y)$ is contained in the escaping strong stable manifold of $\sigma_Y$, proving the lower semi continuity.

\end{proof}

\subsection{The extended maximal invariant set}

We are now able to define the subset of $\PP M$ which extends the lifted maximal invariant set and depends upper-semicontinuously on $X$.

 \begin{defi}Let $U$ be a compact region and $X$  a vector field whose singular points are hyperbolic.
 Then the set
 $$B(X,U)=\La_{\PP,U}\cup\bigcup_{\sigma\in Sing(X)\cap U} \PP^c_{\sigma,U} \subset \PP M$$
 is called the \emph{extended maximal invariant set of $X$ in $U$}
 \end{defi}

 \begin{prop}\label{p.extended} Let $U$ be a compact region and $X$  a vector field whose singular points are hyperbolic.
Then the extended maximal invariant set $B(X,U)$ of $X$ in $U$ is a compact subset  of $ \PP M$, invariant under the flow $\phi_\PP^t$.
Furthermore, the map $X\mapsto B(X,U)$ depends upper semi-continuously on $X$.
 \end{prop}
 \begin{proof}First notice that the singular points of $Y$ in $U$ consists in finitely many hyperbolic singularities varying continuously with $Y$ in a  neighborhood of $X$.
The extended maximal invariant set is compact as being the union of finitely many compact sets.

Let $Y_n$ be a sequence of vector fields tending to $X$ in the $C^1$-topology, and let $(x_n,L_n)\in B(Y_n,U)$.  Up to considering a subsequence we may assume that
$(x_n,L_n)$ tends to a point $(x,L)\in\PP M$ and we need to prove that $(x,L)$ belongs to $B(X,U)$.

First assume that $x\notin Sing(X)$.  Then, for $n$ large, $x_n$ is not a singular point
for $Y_n$ so that $L_n= <Y_n(x_n)>$ and therefore $L=<X(x)>$ belongs to $B(X,U)$,
concluding.

Thus we may assume $x=\sigma\in Sing(X)$. First notice that, if for infinitely many $n$, $x_n$ is a  singularity of $Y_n$ then $L_n$ belongs to $\PP^c_{\sigma_{Y_n},U}$.  As
$\PP^c_{\sigma_{Y},U}$ varies upper semi-continuously  with $Y$, we conclude that $L$ belongs to $\PP^c_{\sigma_{X},U}$, concluding.

So we may assume that $x_n\notin Sing(Y_n)$.

We fix a neighborhood $V$ of $\sigma$ endowed with coordinates, so that $X$ (and therefore $Y_n$ for large $n$) is very close to its linear part in $V$.
Let $S_X\subset W^s_{loc}(\sigma)$ be a sphere transverse to $X$ and cutting every orbit in $W^s_{loc}(\sigma)\setminus\{\sigma\}$, and let $W$ be a small
neighborhood of $S_X$.

First assume that, for infinitely many $n$, the point $x_n$ does not belong  $W^u(\sigma_{Y_n})$.
There is a sequence $t_n>0$ with the following property:
\begin{itemize}
\item   $\phi^{-t}_{Y_n}(x_n)\in V$ for all $t\in [0,t_n]$
 \item $\phi^{-t_n}_{Y_n}(x_n)\in W$
 \item $t_n$ tends to $+\infty$ as $n\to \infty$.
\end{itemize}

Up to considering a subsequence, one may assume that the points $y_n=\phi^{-t_n}_{Y_n}(x_n)$ tend to a point $y\in W^s(\sigma)$.
\begin{clai*} The point $y$ does not belong to $W^{ss}_{\sigma,U}$.
\end{clai*}
\begin{proof}By definition of the escaping strong stable manyfold, for every $y\in W^{ss}_{\sigma,U}$ there is $t$ so that $\phi^t(y)\notin U$; thus $\phi^t_{Y_n}(y_n)$
do not belong to $U$ for $y_n$ close enough to $y$; in particular $y_n\notin \La_{Y_n,U}$.
\end{proof}

Thus $y$ do not belong to $W^{ss}_{\sigma, U}$.  Choosing $T>0$ large enough, one gets that the line $<X(z)>$, $z=\phi^T(y)$, is almost tangent to
$E^{cu}=E^{c}_{\sigma,U}\oplus E^{uu}_{\sigma, U}$.
As a consequence,
for $n$ large, one gets that $<Y_n(z_n)>$, where $z_n=\phi^T_{Y_n}(y_n)$, is almost tangent to the continuation $E^{cu}_n$ of $E^c$ for $\sigma_n, Y_n$.
As $x_n= \phi^{t_n-T}_{Y_n}(y_n)$ , and as $t_n-T\to +\infty$, the dominated splitting implies that $L_n=<Y_n(x_n)>$ is almost tangent  to $E^{cu}_n$.

This shows that  $L$ belongs to $E^{cu}$. Notice that this also holds if $x_n$ belong to the unstable manifold of $\sigma_{Y_n}$.

Arguing analogously we get that $L$ belongs to $E^{cs}=E^{c}_{\sigma,U}\oplus E^{ss}_{\sigma, U}$.  Thus $L$ belongs to $E^{c}_{\sigma,U}$ concluding.
 \end{proof}
\begin{rema}\label{r.inclusion}
The lower semi continuity of the strong escaping stable and unstable manifolds of a vector field $X $, and the upper semicontinuity of $E^c_{\sigma}$, implies that there is a $\cC^1$ neighborhood $\cU$ of $X$, such that for any $Y$ in $U$ there are no regular orbits approaching the singularity $\sigma$ tangent to the escaping spaces. In fact, domination implies that any regular orbit approaching $\sigma$ becomes tangent to $E^c_{\sigma}$. This implies that $$\widetilde{\La(X,U)}\subset B(X,U).$$
\end{rema}

\section{Reparametrizing cocycle associated to a singular point}\label{ss.reparametrization}

Let $X$ be a $C^1$ vector field, $\phi^t$ its flows,  and $\sigma$ be a hyperbolic zero of $X$.
We denote by $\Lambda_X\subset\PP M$ the union
$$\Lambda_X = \overline{\{\RR X(x), x\notin Sing(X)\}}\cup\bigcup_{x\in Sing(X)} \PP_x M.$$

It can be shone easily that this set  is upper semi-continuous, as in the case of $B(X,U)$ (see \ref{p.extended})

\begin{lemm} $\La_X$ is a compact subset of $\PP M$ invariant under the flow $\phi_\PP^t$, and the map $X\mapsto \La_X$ is upper semi-continuous.
Finally, if the zeros of $X$ are hyperbolic then, for any compact regions one has $B(X,U)\subset \La_X$.
\end{lemm}

Let $U_\sigma$ be a compact neighborhood of $\sigma$ on which
$\sigma$ is the maximal invariant.

Let $V_\sigma$ be a compact neighborhood of $Sing(X)\setminus\{\sigma\}$ so that $V_\sigma\cap U_\sigma=\emptyset$.
We fix a ($C^1$) Riemmann metric $\|.\|$ on $M$ so that
$$\|X(x)\|=1 \mbox{ for all } x\in M\setminus (U_\sigma\cup V_\sigma).$$

Consider the map $h\colon \Lambda_X\times \RR\to (0,+\infty)$, $h(L,t)=h^t(L)$, defined as follows:
\begin{itemize}
 \item if $L\in \PP T_xM$ with   $x\notin U_\sigma$ and $\phi^t(x)\notin U_\sigma$, then $h^t(L)=1$;
 \item if $L\in \PP T_xM$ with $x\in U_\sigma$ and $\phi^t(x)\notin U_\sigma$ then $L= \RR X(x)$  and
 $h^t(L)= \frac 1{\|X(x)\|}$;
 \item if $L\in \PP T_xM$ with $x\notin U_\sigma$ and $\phi^t(x)\in U_\sigma$ then $L= \RR X(x)$  and
 $h^t(L)= \|X(\phi^t(x))\|$;
 \item  if $L\in \PP T_xM$ with $x\in U_\sigma$ and $\phi^t(x)\in U_\sigma$ but $x\neq \sigma$ then $L= \RR X(x)$  and
 $h^t(L)= \frac {\|X(\phi^t(x)\|}{\|X(x)\|}$;
 \item if $L\in \PP T_\sigma M$ then $h^t(L)= \frac {\|\phi_\PP^t(u)\|}{\|u\|}$ where $u$ is a vector in $L$.
\end{itemize}
Note that the case in which $x$ is not the singularity and  $x\in U_\sigma$  can be written as in the last item by taking $u=X(x)$.
\begin{figure}[htb]\label{f.h}
\begin{center}
\includegraphics[width=0.8\linewidth]{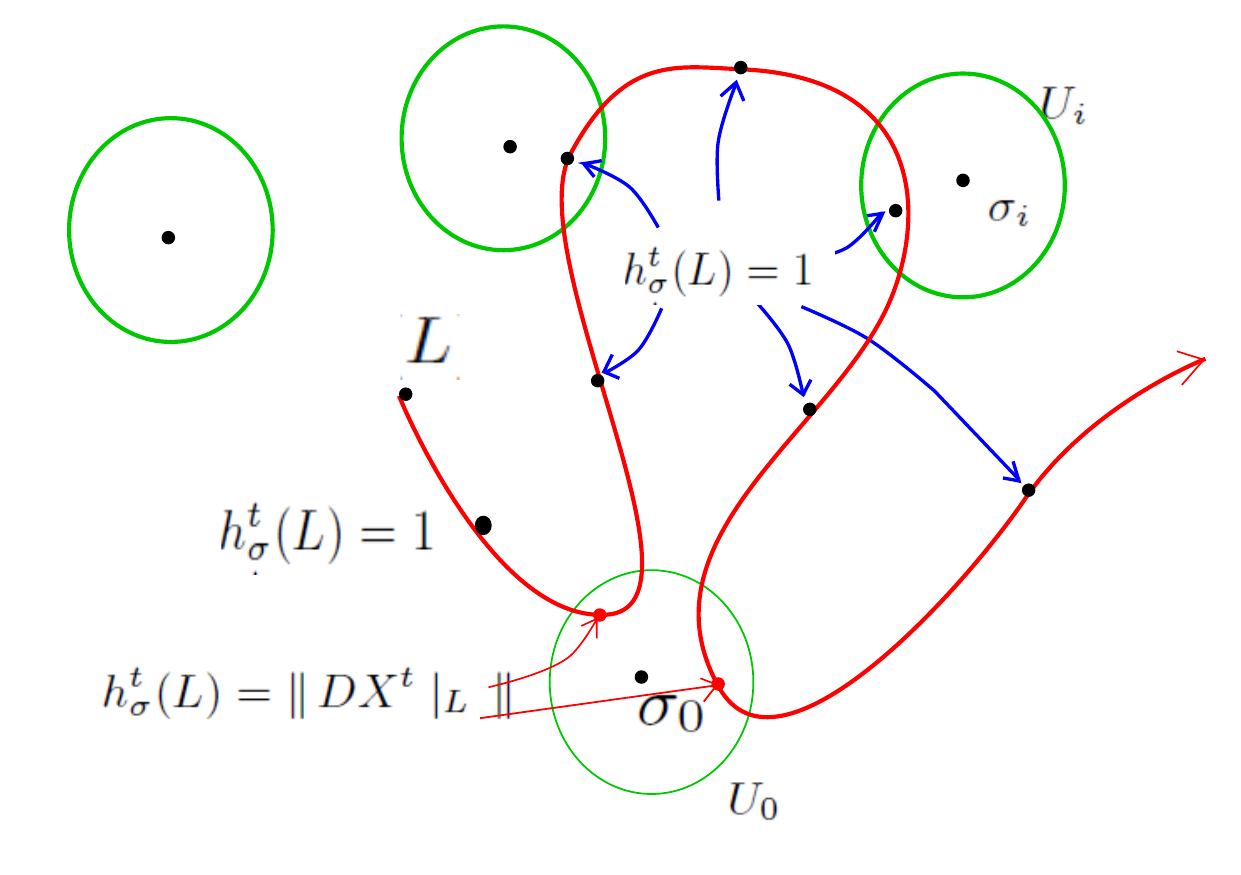}
\end{center}
\caption{the local cocycle $h^t_\sigma$ associated to the singularity $\sigma=\sigma_0$}
\end{figure}
\begin{lemm}\label{l.hcocycle}
With the notation above, the map $h$ is a (continuous) cocycle on $\La_(X,U)$.

\end{lemm}
\proof
The continuity of $h$ comes from the continuity of the norm and the fact that the neighborhoods $U$ and $V$ do not intersect. Now we aim to show that $h$ verifies the
cocycle relation : $$h^{t}(\phi_{\mathbb{P}}^s(L))h^s(L)=h^{t+s}(L)$$
\begin{itemize}
 \item if $L\in \PP T_xM$ with   $x\notin U_\sigma$ , $\phi^{s}(x)\notin U_\sigma$  $\phi^{s+t}(x)\notin U_\sigma$, then $h^{t+s}(L)=h^{t}(\phi_{\mathbb{P}}^s(L))h^s(L)=1$;
\item Let $L\in \PP T_xM$ with   $x\notin U_\sigma$ , $\phi^{s}(x)\notin U_\sigma$  $\phi^{s+t}(x)\in U_\sigma$. Since $ \|X(\phi^s(u))\|=1$ then $h^s(L)=1$ and,
\begin{eqnarray*}
h^{t}(\phi_{\mathbb{P}}^s(L))h^s(L)&=& \|X(\phi^t(\phi^s(u)))\|\\
&=& \|X(\phi^{t+s}(u))\|\\
&=& h^{t+s}(L),
\end{eqnarray*}

\item if $L\in \PP T_xM$ with $x\notin U_\sigma$, $\phi^s(x)\in U_\sigma$ and $\phi^{t+s}(x)\notin U_\sigma$  then $L= \RR X(x)$, $h^s(L)=\|X(\phi^s(x))\|$ and
\begin{eqnarray*}
h^{t}(\phi_{\mathbb{P}}^s(L))h^s(L)&=& \frac 1{\|X(\phi^s(x))\|}\|X(\phi^s(x))\|\\
&=& 1\\
&=& h^{t+s}(L)
\end{eqnarray*}
\item if $L\in \PP T_xM$ with $x\notin U_\sigma$, $\phi^s(x)\in U_\sigma$ and $\phi^{t+s}(x)\in U_\sigma$  then $L= \RR X(x)$, $h^s(L)=\|X(\phi^s(x))\|$ and
\begin{eqnarray*}
h^{t}(\phi_{\mathbb{P}}^s(L))h^s(L)&=&  \frac {\|X(\phi^t(\phi^s(x)))\|}{\|X(\phi^s(x)\|}\|X(\phi^s(x))\|;\\
&=&\|X(\phi^t(\phi^s(x)))\|\\
&=&\|X(\phi^{t+s}(x))\|\\
&=& h^{t+s}(L)
\end{eqnarray*}

 \item if $L\in \PP T_xM$ with   $x\in U_\sigma$ , $\phi^{s}(x)\notin U_\sigma$  $\phi^{s+t}(x)\notin U_\sigma$, then $ h^t(\phi_{\mathbb{P}}^s(L))^s(L))=1$
  \begin{eqnarray*}
h^t(\phi_{\mathbb{P}}^s(L))h^s(L)&=& \frac 1{\|X(x)\|}\\
&=& h^{t+s}(L),
\end{eqnarray*}

\item Let $L\in \PP T_xM$ with   $x\in U_\sigma$ , $\phi^{s}(x)\notin U_\sigma$  $\phi^{s+t}(x)\in U_\sigma$. Since $h^s(L)=\frac 1{\|X(x)\|}$ then,
\begin{eqnarray*}
h^t(\phi_{\mathbb{P}}^s(L))h^s(L)&=& \|X(\phi^t(\phi^s(x)))\|\frac 1{\|X(x)\|}\\
&=& \frac {\|X(\phi^{t+s}(x))\|}{\|X(x)\|}\\
&=& h^{t+s}(L),
\end{eqnarray*}
\item if $L\in \PP T_xM$  with $x\in U_\sigma$, $\phi^s(x)\in U_\sigma$ and $\phi^{t+s}(x)\notin U_\sigma$  then $L= \RR X(x)$, $h^t(\phi_{\mathbb{P}}^s(L))=\frac 1{\|X(\phi^s(x))\|}$ and
\begin{eqnarray*}
h^t(\phi_{\mathbb{P}}^s(L))h^s(L)&=& \frac {\|X(\phi^s(x))\|}{\|X(x)\|}\frac 1{\|X(\phi^s(x))\|}\\
&=&\frac 1{\|X(x)\|}\\
&=& h^{t+s}(L).
\end{eqnarray*}
\item if $L\in \PP T_xM$ with $x\in U_\sigma$, $\phi^s(x)\in U_\sigma$ and $\phi^{t+s}(x)\in U_\sigma$  then $L= \RR X(x)$,
\begin{eqnarray*}
h^t(\phi_{\mathbb{P}}^s(L))h^s(L)&=&  \frac {\|X(\phi^t(\phi^s(x)))\|}{\|X(\phi^{s}(x))\|}\frac{\|X(\phi^{s}(x))\|}{\|X(x)\|} ;\\
&=& \frac {\|X(\phi^t(\phi^s(x)))\|}{\|X(x)\|}\\
&=& \frac {\|X(\phi^{t+s}(x))\|}{\|X(x)\|}\\
&=& h^{t+s}(L)
\end{eqnarray*}

 \item if $L\in \PP T_\sigma M$, let  $u$ be a  vector in $L$, then
 \begin{eqnarray*}
 h^{t+s}(L)=\frac{ \|D\phi_\PP^{t+s}(u)\|}{\|u\|} ;\\
 &=&\frac{\|D\phi_\PP^{t+s}(u)\|}{\|D\phi_\PP^s(u)\|}\frac{\|D\phi_\PP^s(u)\|}{\|u\|}\\
&=&\frac{\|D\phi_\PP^t(D\phi_\PP^s(u))\|}{\|D\phi_\PP^s(u)\|}\frac{\|D\phi_\PP^s(u)\|}{\|u\|}\\
&=&h^t(\phi_{\mathbb{P}}^s(L)) h^{s}(L)
\end{eqnarray*}

\end{itemize}
\endproof

\begin{lemm}\label{l.hsigma}
The cohomology class of a cocycle $h$ defined as above,  is independent from the choice of the metric $\|.\|$ and of the neighborhoods
 $U_\sigma$ and $V_\sigma$.
\end{lemm}
\begin{proof}
Let $\|.\|$ and $\|.\|'$ be two different metrics  and 2 different sets of neighborhoods of $\sigma$ and $Sing(X)\setminus\{\sigma\}$ such that:
\begin{itemize}
\item $V_\sigma\cap U_\sigma=\emptyset$.
\item $V'_\sigma\cap U'_\sigma=\emptyset$.
\item $V'_\sigma\cap U_\sigma=\emptyset$ and  $V_\sigma\cap U'_\sigma=\emptyset$ .
\item $\|X(x)\|=1 \mbox{ for all } x\in M\setminus (U_\sigma\cup V_\sigma),$
\item $\|X(x)\|'=1 \mbox{ for all } x\in M\setminus (U'_\sigma\cup V'_\sigma).$
\end{itemize}
 We define  $h$ as above for the metric $\|.\|$ and   $h'$  as above for the metric  $\|.\|'$.
 We define a function $g\colon B(X,U)\to (0,+\infty)$
 \begin{itemize}
\item If $L\in \PP T_xM$ with $x\notin V'_\sigma\cup V_\sigma$ then  $g(L)=\frac{\|u\|'}{\|u\|} $ with $u$ a non vanishing vector in $L$,
\item  and if $L\in \PP T_xM$ with  $x\in V'_\sigma\cup V_\sigma$, then $g(L)=1 $  .
\end{itemize}
\begin{clai*}
The function $g\colon B(X,U)\to (0,+\infty)$ defined above is continuous
\end{clai*}

\proof
Since $V'_\sigma\cap U_\sigma=\emptyset$ and  $V_\sigma\cap U'_\sigma=\emptyset$, the continuity in the boundary of $V_\sigma\cup V_\sigma'$ comes from the fact that $\|.\|$ and $\|.\|'$ are 1 out of $U_\sigma\cup U_\sigma'$.
Also since $V'_\sigma\cap U_\sigma=\emptyset$ and  $V_\sigma\cap U'_\sigma=\emptyset$,
 The continuity of the norms  $\|.\|$ and $\|.\|'$, and the fact that they are 1 out of $V_\sigma\cup V_\sigma'$ , gives us the continuity in the boundary of $U_\sigma\cup U_\sigma'$.
\endproof
The following claim  will show us that the functions $h$ and $h'$ differ in a coboundary defined as
 $g^t(L)=\frac {g(D\phi_\PP^t(u))}{g(u)}$.
\begin{clai*}
The functions $h$ and $h'$ are such that

 $h'^t(u)=h^t(u)\frac {g(D\phi_\PP^t(u))}{g(u)}\,.$
 \end{clai*}
\proof
\begin{itemize}
 \item For the metric  $\|.\|'$  and $L\in \PP T_xM$ with   $x\notin U_\sigma\cup U'_\sigma $ and $\phi^t(x)\notin  U_\sigma\cup U'_\sigma $, $g^t(L)=1 $. On the other side $h'^t(L)=1$ as desired.
\item If $L\in \PP T_xM$ with $x\in U_\sigma\cup U'_\sigma$ and $\phi^t(x)\notin  U_\sigma\cup U'_\sigma $ then   $g^t(L)=\frac{\|u\|}{\|u\|'} $. Take $u=X(x)$
$h^t(L)= \frac 1{\|X(x)\|}$; $$h'^t(L)=h^t(L)\frac{\|X(x)\|}{\|X(x)\|'}\,.$$

 \item If $L\in \PP T_xM$ with $x\notin  U_\sigma\cup U'_\sigma$ and $\phi^t(x)\in  U_\sigma\cup U'_\sigma$ then $L= \RR X(x)$  and  take $u=X(x)$.
 $g^t(L)=\frac{\|D\phi_\PP^t(u)\|'}{\|D\phi_\PP^t(u)\|} $.Then since $h^t(L)=\|D\phi_\PP^t(u)\|$,then
 $$h'^t(L)=h^t(L)\frac{\|D\phi_\PP^t(u)\|'}{\|D\phi_\PP^t(u)\|} \,.$$.
 \item  If $L\in \PP T_xM\cap B(X,U)$ with $x\in U_\sigma$ and $\phi^t(x)\in U_\sigma$  .  Take $u=X(x)$, then
  $$g^t(L)\frac {\|D\phi_\PP^t(u)\|'\|u\|}{\|D\phi_\PP^t(u)\|\|u\|'}\,,$$ and
$h^t(L)= \frac {D\phi_\PP^t(u)}{\|u\|}$. So  $h'^t(L)=h^t(L)g^t(L)$
\end{itemize}

\endproof

Now in order to finish the proof we need to show that assuming the condition that the norms where such that $V'_\sigma\cap U_\sigma=\emptyset$ and  $V_\sigma\cap U'_\sigma=\emptyset$ does not make us loose generality.
For this, suppose we started with any  other norm $\|.\|''$ and that there exist 2 neighborhood such that.
\begin{itemize}
\item $V''_\sigma\cap U''_\sigma=\emptyset$.
\item $\|X(x)\|''=1 \mbox{ for all } x\in M\setminus (U''_\sigma\cup V''_\sigma).$
\end{itemize}
Let us choose a smaller neighborhood $V'_\sigma\subset V_\sigma''$. This satisfies  $V'_\sigma\cap U''_\sigma=\emptyset$.
Analogously   $U'_\sigma \subset U_\sigma''$ will satisfy  $V''_\sigma\cap U'_\sigma=\emptyset$.
Now if we choose this neighborhoods $V'$ and $U'$ as small as we want, and a norm
 norm $\|.\|'$ such that
 $\|X(x)\|'=1 \mbox{ for all } x\in M\setminus (U''_\sigma\cup V''_\sigma).$
the claims above  implies that the corresponding $h''$ and $h'$ differ in a coboundary.
Therefore $h'$ can be chosen so that  $h''$ and $h$ differ in a coboundary.
\end{proof}
We denote by $[h(X,\sigma)]$ the cohomology class of any cocycle defined as $h$ above.

\begin{lemm}\label{l.representativecocycle} Consider a vector field $X$ and a hyperbolic zero $\sigma$ of $X$. Then there is a $C^1$-neighborhood $\cU$ of $X$ so that $\sigma$ has a well defined hyperbolic
continuation $\sigma_Y$ for $Y$ in $\cU$ and for any $Y\in\cU$ there is a map $h_Y\colon \La_Y\times \RR\to (0,+\infty)$ so that
\begin{itemize}
 \item for any $Y$, $h_Y$ is a cocycle belonging to the cohomology class $[h(Y,\sigma_Y)]$
 \item $h_Y$ depends continuously on $Y$:  if $Y_n\in \cU$ converge to $Z\in \cU$ for the $C^1$-topology and if $L_n\in \La_{Y_n}$ converge to $L\in \La_Z$ then
 $h^t_{Y_n}(L_n)$ tends to $h^t_Z(L)$ for every $t\in \RR$; furthermore this convergence is uniform in $t\in[-1,1]$.
\end{itemize}

\end{lemm}
\begin{proof} The manifold $M$ is endowed with a Riemmann metric $\|.\|$.  We fix the neighborhoods  $U_\sigma$ and $V_\sigma$ for $X$ and $\cU$ is a $C^1$-neighborhood of $X$ so that
$\sigma_Y$ is the maximal invariant set for $Y$ in $U_\sigma$ and  $Zero(Y)\setminus\{\sigma_Y\}$ is contained in the interior of $V_\sigma$.
Up to shrink $\cU$ if necessary, we also assume that there are compact neighborhoods $\tilde U_\sigma$ of $\sigma_Y$ contained in the interior of $U_\sigma$ and $\tilde V_\sigma$ of
$Zero(Y)\setminus \{\sigma_Y\}$ contained in the interior of $V_\sigma$.

We fix a continuous function $\xi\colon M\to [0,1]$ so that $\xi(x)= 1$ for $x\in M\setminus (U_\sigma\cup V_\sigma)$ and $\xi(x)=0$ for $x\in\tilde U_\sigma\cup \tilde V_\sigma$.

For any $Y\in \cU$ we  consider the map $\eta_Y\colon M\to (0,+\infty)$ defined by
$$\eta_Y(x)=\frac{\xi(x)}{\|X(x)\|}+ 1-\xi(x).$$
This map is a priori not defined on $Zero(Y)$ but extends by continuity on $y\in Zero(Y)$ by $\eta_Y(y)=1$, and is continuous.

This maps depends continuously on $Y$.   Now we consider the metric $\|.\|_Y= \eta_Y\|.\|$.
Note that $\|Y(x)\|_Y=1$ for $x\in M\setminus (U_\sigma\cup  V_\sigma$).

Now $h_Y$ is the cocycle built at lemma~\ref{l.hsigma} for $U_\sigma, V_\sigma$ and $\|.\|_Y$.
\end{proof}

Notice that, according to Remark~\ref{r.cohomological}, if $\sigma_1,\dots,\sigma_k$ are hyperbolic zeros of $X$ the homology class of the
product cocycle
$h_{\sigma_1}^t\cdots h_{\sigma_k}^t$ is well defined, and admits representatives varying continuously with the flow.
\section{Extension of the Dominated splitting}
\subsection{The  dominated splittings over the singularities }

The aim of this section is to prove Theorem \ref{t.nioki}

.


\begin{rema}\label{r.proy}
Let us suppose that the finest Lyapunov decomposition of the singularity is $$T_{\sigma}M=E_{1}\oplus\dots\oplus E_i\oplus \dots \oplus E_j \oplus \dots \oplus  E_{l}\,.$$
If we pick a direction $L\in\PP_\sigma M$ such that the closure of its orbit under $\phi^t_{\PP}$, $\overline{O(L)}$ is contained in $ E_i\oplus \dots \oplus E_j$, then the angle between $\phi^t_{\PP}(L)$ and  any space $$\sum_{h<i} E_h \text{ and } \sum_{h>j}E_h\,,$$  is uniformly away from zero.

We can then identify $\pi_L(E_{h})$ with $E_{h}$ for all $h< i$ and $h>j$, and also we can identify
$$\phi^t_{\PP}(\sum_{h<k} E_h)\text{ with } \sum_{h<k} E_h$$ for any $k<i$ and also  $$\phi^t_{\PP}(\sum_{k<h} E_h)\text{ with } \sum_{k<h}E_h$$ for $k>j$.
\end{rema}

 \begin{lemm}\label{l.refdom}
We consider $X$ a vector field in a $d$ dimensional manifold $M$, with a  hyperbolic singularity $\sigma$ where the finest splitting of the tangent space is $$T_{\sigma}M=E_{1}\oplus\dots\oplus  E_{l}\,.$$
Then, we   consider a metric such that $E_i\perp E_j$ for all $i\neq j$.
 Let $L $ be  such that $L=<u>$ where $u$ belongs to some $E_{i}$.
Then $$\cN_L= E_{1}\oplus\dots\oplus \pi_L(E_{i})\oplus \dots\oplus E_{l}\,.$$ is the finest dominated splitting over the closure of the orbit of $L$.
 \end{lemm}
 \proof
Suppose that $\lambda<0$, the other case is analogous.
Let us consider a vector $w$ in $\pi_L (E_{i})$
 Since $D\phi^t(u)$ and $\psi^t_{\cN}(w)$ are perpendicular, then $$J(D\phi^t)\mid_{E_i}=J(\psi^t_{\cN})\mid_{\pi_L (E_{i}) }\norm{D\phi^t(u)}\,.$$

We have that $J(D\phi^t)=t\lambda^d$  and the biggest Lyapunov exponent of  $\psi^t_{\cN}$, cannot be bigger than $\lambda$. Therefore there are two possibilities:
\begin{enumerate}
  \item there exist a constant $C$ such that for $t$ big enough $J(\psi^t_{\cN})= C t\lambda^{(d-1)}$
  \item or the ratio of contraction of  $u$ is bigger than $\lambda$.
\end{enumerate}
since $D\phi^t\mid_{E_i}$ has all the  Lyapunov exponent equal to $\lambda$ then for $t $ big enough $$\norm{D\phi^t(u)}\leq C_1\norm{u}e^{\lambda t}\,,$$ reaching a contradiction.

 \endproof

The following lemma together with Lemma \ref{l.central} are very similar to lemma 4.3 in \cite{GLW}. Since the context is slightly different and the statement is splitted in two parts, we add the proof anyway.
 \begin{lemm}\label{l.domalg}
We consider $X$ a flow with a  hyperbolic singularity $\sigma$ where the finest hyperbolic decomposition of the tangent space is
$$T_{\sigma}M=E_{1}\oplus\dots E_i\oplus \dots \oplus E_j\oplus \dots \oplus  E_{l}\,.$$

Let us consider a stable space $E_i$,  and an unstable space $E_j$.    Let denote $k_i=\sum_1^{i-1} dim (E_k)$ and $h_j= \sum_{j+1}^l dim(E_k)$.

Consider a direction $L=<u>$ where $u$ is a vector in  $\left[\left( E_i\oplus E_j \right)\setminus \left( E_i\cup E_j\right) \right]$.
Assume $E\oplus_{_<} F$ is a dominated splitting for  $\psi_\cN$ over the closure  $\overline{O(L)}$ of $L$ for $\phi_\PP$. Then:
\begin{itemize}
 \item either $dim E\leq k_i$.  In that case, there is $1\leq i'<i$ so that, for any $L'\in \overline{O(L)}$, one has
 $$E=\pi_{L'}\left(\sum_1^{i'}E_k\right)\simeq  \sum_1^{i'}E_k,\mbox{  and } F= \pi_{L'}\left(\sum_{i'+1}^l E_k\right).$$
 In particular $F$ contains the projection of the sum of the $E_k$ for $k\geq i$.
 \item or  $dim F\leq h_j$.  In that case, there is $j<j'\leq l$ so that, for any $L'\in \overline{O(L)}$, one has
 $$F=\pi_{L'}\left(\sum_{j'}^l E_k\right)\simeq  \sum_{j'}^l E_k,\mbox{ and } E= \pi_{L'}\left(\sum_1^{j'-1} E_k\right).$$
 In particular $E$ contains the projection of the sum of the $E_k$ for $k\leq j$.
\end{itemize}

 \end{lemm}
 \proof
 First note that, as $L$ is contained in $E_i\oplus E_j$ but not in $E_i$ nor $E_j$, then the $\omega$-limit of $L$ for $\phi_\PP$ is contained in
 $\PP E_j$ and its $\alpha$-limit is contained in $\PP E_i$.

 We argue by contradiction, assuming that there is a dominated splitting $E\oplus F$ over $\overline{O(L)}$ so that $dim E>k_i$ and $dim F>h_j$.

 According to Lemma~\ref{l.refdom}, for any  $L_\omega\in \overline{O(L)}\cap \PP E_j$ the finest dominated splitting of $\overline{O(L_\omega)}$ is
 obtained from $E_{1}\oplus\dots E_i\oplus \dots \oplus E_j\oplus \dots \oplus  E_{l}$  by substituting the space $E_j$ by its projection and keeping all the other unchanged
 (modulo their identification with their projection).  Thus the splitting $E\oplus F$ is just given by the dimension. So there is $i\leq r<j$ so that for any such $L_\omega$ one has:
 $$E(L_\omega)= E_1\oplus\cdots\oplus E_r \mbox{ and } F(L_\omega)= E_{r+1}\oplus\cdots E_{j-1}\oplus \pi_{L_{\omega}}(E_j)\oplus\cdots \oplus E_l.$$

 The same argument show that there is $i<s\leq j$ so that for any $L_\alpha$ in $\PP E_i\cap \overline{O(L)}$ ($\alpha$-limit of $L$) one has

 $$E(L_\alpha)= E_{1}\oplus\cdots E_{i-1}\oplus \pi_{L_{\alpha}}(E_i)\oplus\cdots \oplus E_s \mbox{ and } F(L_\alpha)= E_{s+1}\oplus\cdots\oplus E_l.$$

 This allows us to know the spaces $E(L)$ and $F(L)$. For that, consider an unstable cone around the space $F(L_\omega)$ and
 extend it by continuity in a small neighborhood of the $\omega$-limit of $L$. Then $E(L)$ is exactly the set of vectors which do not enter in the unstable cone for large positive iterates by the
 extended linear Poincar\'e flow. One deduces that

 $$E(L)= E_1\oplus \cdots \oplus E_{i-1}\oplus \pi_L(E_i\oplus L)\oplus E_{i+1}\oplus \cdots \oplus E_r.$$
 In the same way, $F(L)$ consist in the vectors which do not enter in the stable cone defined on the $\alpha$-limit set of $L$ by large negative iterates of the extended linear Poincar\'e flow.
 One deduces that
 $$F(L)= E_s\oplus \cdots \oplus E_{j-1}\oplus \pi_L(E_j\oplus L)\oplus E_{j+1}\oplus \cdots \oplus E_l.$$

Consider the positive iterates $\psi^t_\cN(F(L))= F(\phi^t_\PP(L))$  of $F(L)$. Denote $L_t=\phi^t_\PP(L)$.  Then $F(L_t)$ contains $\pi_{L_t}(E_j\oplus L_t)$ which has the same dimension as $E_j$.
Recall that $L_t$ is contained, by hypothesis, in $E_i\oplus E_j$.  Thus $\pi_{L_t}(E_j\oplus L_t)$ converges in $\cN(L_\omega)$  to some subspace of
$\pi_{L_\omega}(E_i\oplus E_j)\simeq E_i\oplus \pi_{L_\omega}(E_j)$ containing $\pi_{L_\omega}(E_j)$ and having the same dimension has $E_j$. This implies that the limit of $\pi_{L_t}(E_j\oplus L_t)$
for $t\to +\infty$ contains vectors in $E_i$ but is contained in $F(L_\omega)$. This contradicts the fact that $E_i$ is contained in $E(L_\omega)$.

This contradiction implies that $dim E\leq k_i$ or  $dim F\leq h_j$. We now conclude the proof in the first case, the other case being similar.

Assume $dim E \leq k_i$. Then looking at the finest dominated splitting at $L_\omega$ one deduces that there is $1\leq i'<i$ so that $dim E= \sum_1^{i'}dim E_k$.
Then the splitting  $\cN(L')= \tilde E(L')\oplus \tilde F(L')$ defined as
$$\tilde E(L')=\pi_{L'}\left(\sum_1^{i'}E_k\right)\simeq  \sum_1^{i'}E_k,\mbox{  and } \tilde F (L')= \pi_{L'}\left(\sum_{i'+1}^l E_k\right)$$
is invariant, has constant dimension for $L'\in \overline{O(L)}$ and coincides with a dominated splitting over the $\omega$-limit set and over the $\alpha$-limit set.
Therefore this splitting is a dominated splitting so that $dim \tilde E= dim E$ and so $\tilde E=E$ and $\tilde F= F$ concluding.

\endproof

\subsection{Relating the central space of the singularities with the dominated splitting  on $\widetilde{\La}$}

Now let us go back to our dynamical context. Let us consider a vector field with a  chain class $C$ and a singularity $\sigma\in C$. We consider the following splitting of its tangent space:
$$E^{ss}\oplus E^c\oplus E^{uu}\,,$$ noting the stable escaping, the unstable escaping and the central spaces.
We can suppose that the singularities are hyperbolic and that the dimension of the central space is locally constant. This are open and dense conditions. 
 Let us consider the hyperbolic eigen values of the hyperbolic splitting  restricted to the central space:
$$\lambda_{1}<\dots<\lambda_{l}\,$$
and the spaces associated to them $$E^c=E_{1}\oplus\dots\oplus E_{l}\,.$$

Note that it follows from remark \ref{r.inclusion} we know that $\widetilde{\La}\subset B(X,U)$, and from Theorem \ref{ConConLem} we have that $\La_{\PP}(X,U)\subset\widetilde{\La}$.

 \begin{rema}
By lema \ref{l.lower} we have that there is a $C^1$ open and dense set such that the dimension of the central space is locally constant.
By definition of central space there is always a direction  $L_1$  in $\widetilde{\La}\cap\PP_{\sigma}M$,  such that $L_1=<u>$ where $u$ belongs to $E^{1}$ and  $L_l$  in $\widetilde{\La}\cap\PP_{\sigma}M$,  such that $L_l=<v>$ where $v$ belongs to $E^{l}$.
 \end{rema}

\begin{lemm}\label{l.conect}
We consider $X$ a vector field such that :

\begin{itemize}
\item There is a hyperbolic singularity $\sigma$ and the splitting over the tangent space of the singularity into escaping spaces and central space is:
$$T_{\sigma}M=E^{ss}\oplus E^c\oplus E^{uu}\,.$$
\item The central space splits into $E^c=E_{1}\oplus\dots\oplus E_{l}\,.$
\item The class of $\sigma$, $C_{\sigma}$, is not trivial.
\item  The dimension of $E^c$ is locally constant.(i.e. the dimension of $E^c(\sigma,Y)$ is constant for $Y$ in a $C^1$ open neighborhood of $X$)

\end{itemize}

Then for any $ C^1$ open set $ \cU$ of $X$, there is $ Y$ in $\cU$  such that there is a homoclinic orbit $\gamma\subset C(\sigma)$, that approaches the singularity tangent to the $E_{1}$ direction and for the past, tangent to  $E_l$.
 \end{lemm}
 \proof

  Let us consider the fines hyperbolic decomposition of the central space of $\sigma$ for this vector field:
 $$E^c=E_{1}\oplus\dots\oplus  E_{l}\,.$$
 By definition, there is an orbit  in the stable manifold tangent to  $E^{ss}\oplus E_{1}$ that is contained in $C_{\sigma}$ and there is an orbit  in the unstable manifold tangent to $E_{l}\oplus E^{uu}$ that is contained in $C_{\sigma}$.
 In the open set around $X$ such that the dimension of the central space is constant, we choose $Y$ such that all periodic orbits of $Y$ are hyperbolic, and the orbit in the stable manifold tangent to $E^{ss}\oplus E_{1}$ approaches the singularity in the direction of $E_{1}$ and the orbit in the unstable manifold tangent to $E^{l}\oplus E_{uu}$ approaches the singularity in the direction of $E_{l}$.
 By theorem \ref{l.contecting} we can get another vector field $Y_1$ arbitrarily close to $Y$ that has an homoclinic  orbit $\Gamma$ such that it approaches the singularity in the direction of $E_{1}$ for the future and in the direction of  $E_{l}$ for the past (observe that for $Y_1$ the dimension of the central space is the same as for $X$).

 \endproof

 \begin{coro}\label{c.dir}
 We consider $X$ a vector field such that :

\begin{itemize}
\item There is a singularity $\sigma$ and the tangent space of the singularity splits into escaping spaces and central space as follows:
$$T_{\sigma}M=E^{ss}\oplus E^c\oplus E^{uu}\,.$$
\item The central space splits into $E^c=E_{1}\oplus\dots\oplus E_{l}\,.$
\item The class of $\sigma$, $C_{\sigma}$, is not trivial.
\item  The dimension of $E^c$ is locally constant.(i.e. the dimension of $E^c(\sigma,Y)$ is constant for $Y$ in a $C^1$ open neighborhood of $X$)

\end{itemize}

Then
\begin{itemize}
  \item  There is $L_1\in \widetilde{\La}(C(\sigma))\cap \pi_{\PP} (E_{1})$ and $L_l\in \widetilde{\La}(C(\sigma))\cap  \pi_{\PP}(E_{l})$.
  \item  There is $L\in \widetilde{\La}(C(\sigma))$ such that $L=<u>$ where $u$ is a vector in  $$\left[\left( E_1\oplus E_l \right)\setminus \left( E_1\cup E_l\right) \right]$$.
\end{itemize}

 \end{coro}
\proof
The first item is a direct consequence of Lemma \ref{l.conect}.

For the second item,
 from lema \ref{l.conect}, we can find a vector field $Y$ having a homoclinic orbit $\gamma$ such that it approaches the singularity $\sigma$ tangent to $L_1$ and it approaches the singularity for the past, tangent to a direction $L_l$ in $E_{l}$.

We consider now a linearized neighborhood of the singularity that we call $U_{\sigma}$, and choose two regular points $x$, $y$ such
that $x\in W^s_{loc}\sigma\cap\gamma$ and  $y\in W^u_{loc}\sigma\cap\gamma$
 Then we can choose $x_n \to x$ and $y_n \to y$, such
that $\phi_{t_n}(x_n) = y_n$ and $\set{\phi_{t}(x_n) \text{ for every } 0\leq t\leq t_n}$ is tangent to $E_1\oplus E_l$, note that for $n$ big enough we can  suppose that the segment of orbit from $x_n$ to $y_n$ is in $U_{\sigma}$ and in the linearized neighborhood so actually $\set{\phi_{t}(x_n) \text{ for every } 0\leq t\leq t_n}\subset E_1\oplus E_l$.

\begin{figure}[htb]
\begin{center}
\includegraphics[width=0.55\linewidth]{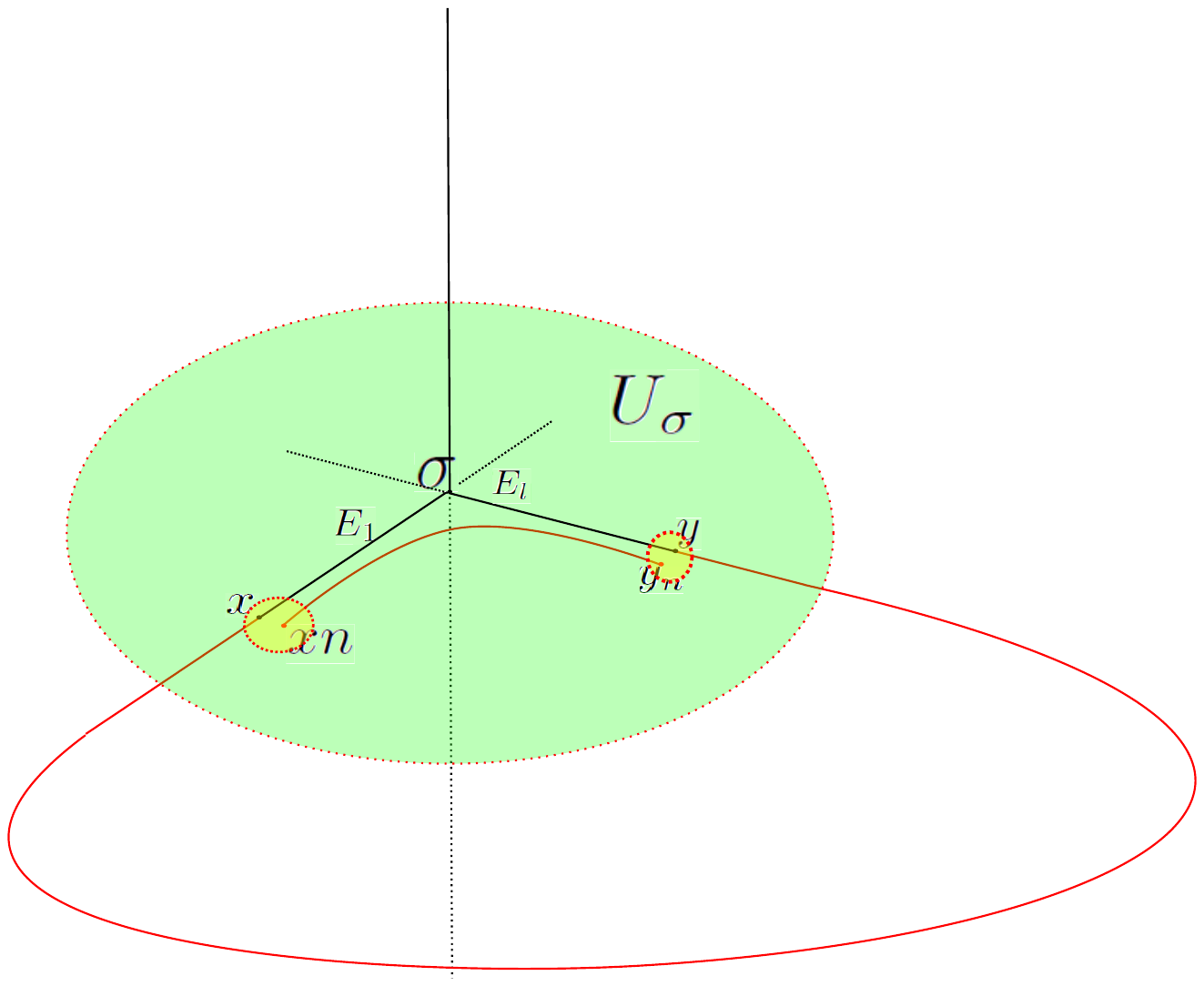}
\end{center}
\label{f.loop}
\caption{perturbation to get $\gamma_n$  }
\end{figure}

We now perturb our vector field $X$ to a new vector field  $X_n\to X$  so that there is a closed orbit $\gamma_n$ formed by the segment of orbit between $x_n$ and $y_n$ in $U_{\sigma}$, and the segment of orbit of $\gamma$ outside $U_{\sigma}$ ( see figure \ref{f.loop})

We can now find a $p_n$ in  $\gamma_n$  satisfying $p_n \to \sigma$ and if $L_n$ is such that $L_n=<X_n(p_n)>$, then
 the upper limit of $L_n$ is a subset of $\widetilde{\Lambda}(C(\sigma))$, ( i.e. all
limit points of $L_n$ are in $\widetilde{\Lambda}(C(\sigma))$). In fact the limit points of $L_n$ are a one dimensional subspace of $ E_1\oplus E_l$.

Taking subsequence when necessary, we may assume that $L_n\to L \in \left[\left( E_1\oplus E_l \right)\setminus \left( E_1\cup E_l\right) \right]$.

\endproof

 In this section we  suppose that  the extended linear Poincar\'{e} flow over $\widetilde{\La}(C(\sigma))$ has a  dominated splitting,
$$\cN_L=\cN^E\oplus\cN^F\,,$$
where $L$ is a direction in $\widetilde{\La}(C(\sigma))$.

We call $\pi_L:T_xM\to \cN_L$ where $L\in\mathbb{P}_xM$ the projection over the normal space at a given direction $L$.

\begin{lemm}\label{l.central}
Let $X$ be a vector field having  a singular chain class $C_{\sigma}$.  We denote $S= Sing(X)\cap C_{\sigma}$ and we suppose that\begin{itemize}
                                                           \item  every $\sigma\in S$ that is hyperbolic,
                                                           \item the dimension of the central space of $\sigma\in S$ is locally constant,

                                                           \item the extended linear Poincar\'{e} flow over $\widetilde{\La}(C(\sigma))$ has a  dominated splitting,
$$\cN_L=\cN^E\oplus\cN^F\,,$$
where $L$ is any direction in $\widetilde{\La}(C(\sigma))$.
                                                         \end{itemize}

Let $L $ be a direction in $\widetilde{\La}\cap\PP_{\sigma}M$, 
Then, $$\pi_L ( E_{\sigma}^c)\subset\cN_L^E\,,$$ or
$$\pi_L (E_{\sigma}^c)\subset\cN_L^F\,.$$
 \end{lemm}
 \proof
 Since $\sigma$ is hyperbolic we can suppose that the tangent space of $\sigma$  splits into $$T_{\sigma}M=E^{ss}\oplus E^c\oplus E^{uu}\,,$$ the escaping spaces and the central space. We consider as well the finest hyperbolic splitting over the singularity of $E^c=E_{1}\oplus\dots\oplus  E_{l}$.

 Let us suppose that $dim(\cN^E_L)=n$.
 If the $dim(E^{ss})\geq n$, then since $dim(E^{ss}) =dim(\pi_L (E^{ss}))$ we have that $$\cN^E_L\subset \pi_L (E^{ss})\,.$$
 This implies that $$\pi_L ( E_{\sigma}^c)\subset\cN_L^F\,.$$

 Let us suppose that $dim(\cN^F_L)=m$.
 If the $dim(E^{uu})\geq m$, then since $dim(E^{uu}) =dim(\pi_L (E^{uu}))$ we have that $$\cN^F_L\subset \pi_L (E^{uu})\,.$$
 This implies that $$\pi_L ( E_{\sigma}^c)\subset\cN_L^E\,.$$

Suppose now that $dim(E^{ss})< dim(\cN^E_{L})$ from Corollary \ref{c.dir} for every  $L$  we have that $\pi_{L}(E_1)\subset\cN^E_{L}$.

We suppose by contradiction that there exist a direction $L_u=<u>$ such that $\cN^E_{L_u}$ contains some vector $v\in\pi_{L_u}( E_{1}\oplus\dots\oplus E_{l})$ such that $v\notin\cN^E_{L_u}$ and that if $L_v=<v>$ then  $L_u\in\cN^E_{L_v}$.

 We can assume without loss of generality that $v\in E_{l}$ and that $u\in E_1$. Then Lema \ref{l.refdom} gives us that $\cN^E_{L_u}\cap E_{l}=\emptyset$.

This implies that the dimension of $\cN^E_{L}$ is lesser than $dim(\pi_L(E^{ss}\oplus E^c))$ for any $L$ in $\widetilde{\Lambda}(C(\sigma))$.
Since the singularity is not isolated then $E_1$ is a contracting space and $E_l$ expanding.
We can now apply Corollary \ref{c.dir} to get a direction
 $L$  in the conditions of Lemma \ref{l.domalg} and this allows us to conclude.
 We can do this for all singularities in $S$
\endproof

\begin{coro}\label{c.dom}
Let $X$ be a vector field having a singular chain class $C(\sigma)$.  We denote $S= Sing(X)\cap C(\sigma)$, and let us suppose that \begin{itemize}
                                                           \item  all $\sigma\in S$ are hyperbolic
                                                           \item the dimension of the central space of any singularity $\sigma$ is locally constant.

                                                         \end{itemize}
Then the extended linear Poincar\'{e} flow  has a dominated splitting over $\widetilde{\Lambda}(C(\sigma))$ if and only if  the extended linear Poincar\'{e} flow  has a dominated splitting over $B(C(\sigma))$  of the same dimension.

 \end{coro}

 \proof
 Suppose that $B(C(\sigma))$ has a dominated splitting, then $\widetilde{\Lambda}(C(\sigma))$ has a dominated splitting of the same dimension since it is a compact invariant subset.
 Suppose that there is a dominated splitting of the normal bundle in  $\widetilde{\Lambda}(C(\sigma))$ $$\cN_L=\cN^E\oplus\cN^F\,,$$
 Then according to the previous lemma  we have 2 possibilities
  $$\pi_L ( E_{\sigma}^c)\subset\cN_L^E\,,$$ or
$$\pi_L (E_{\sigma}^c)\subset\cN_L^F\,.$$

The tangent space of $\sigma$  splits into $$T_{\sigma}M=E^{ss}\oplus E^c\oplus E^{uu}\,,$$ the escaping spaces and the central space. We consider as well the finest hyperbolic splitting over the singularity of \begin{itemize}
                                                                                                           \item$E^c=E_{1}\oplus\dots\oplus  E_{l}$
                                                                                                           \item$E^{ss}=E_{s1}\oplus\dots\oplus  E_{sk}$
                                                                                                           \item $E^{uu}=E_{u1}\oplus\dots\oplus  E_{ur}$
                                                                                                           \end{itemize}
So if we are in the case where $\pi_L ( E_{\sigma}^c)\subset\cN_L^E\,,$  lemma \ref{l.refdom} implies that there exist an $i$  such that $$\cN_L^F= E_{ui}\oplus\dots\oplus  E_{ur}$$ and
$$\cN_L^E= E_{ss}\oplus\pi_L( E^c)\oplus E_{u1}\dots\oplus  E_{ui-1}$$. This same dominated splitting can be defined for any $L\in B(C(\sigma))$. The other case is analogous.

 We can do the same for every singularity in the class.
 \endproof

\begin{lemm}\label{propinc}
Let us consider  a chain recurrence class $C(\sigma)$ of a vector field $X$, having a singularity $\sigma$ where the tangent space splits into $$T_{\sigma}M=E^{ss}\oplus E^c\oplus E^{uu}\,,$$ the escaping spaces and the central space. We consider as well the finest lyapunov splitting over the singularity of the central space is
$E^c=E_{1}\oplus\dots\oplus  E_{l}$.
If the dimension of the  central space is locally constant  then
$$\pi_{\PP}(E_i)\cap\widetilde{\La}(C(\sigma))\neq\emptyset$$ for all the $E_i$,  Lyapunov spaces of the hyperbolic splitting.
More over  if  the spaces $E_i$ of the central space $E^c$ are only one or two dimensional, then $$\pi_{\PP}(E_i)\subset\widetilde{\La}(C(\sigma))$$ for all the $E_i$, Lyapunov spaces of the hyperbolic splitting.
\end{lemm}
\proof
Let us consider $\pi_{\PP}(E_{1})$ in $\PP^c_{\sigma}$. By definition of central space, there is an orbit $\gamma_1$ tangent to $E^{ss}\oplus E_{1}$, that is not tangent to $E^{ss}$. This implies that $\pi_{\PP}(E_{1})\cap\widetilde{\La}(C(\sigma))\neq\emptyset $

 Consider a small filtrating neighborhood of $C(\sigma)$, $U$

  First we perturb $X$ to a vector field $Y'$ that is Kupka-Smale. We can make the perturbation small enough so that the vector field $Y'$  is in the hypothesis of the lemma as well, since our assumptions are robust.

  By \ref{l.contecting} we perturb  $Y'$  to $Y$  so that  $\gamma_1$ is a homoclinic connection of the singularity and without changing the fact that $\gamma_1$ becomes tangent to $E^{ss}\oplus E_{1}$ as it approaches the singularity.
 Now we perturb $Y$  to $Y_1$ braking the homoclinic connection in the direction of $E_2$ so that is no longer tangent to  $E^{ss}\oplus E_{1}$ but is tangent to  $E^{ss}\oplus E_{1}\oplus E_{2}$.
 The domination implies that the orbit will become tangent to $E_{2}$ as it approaches $\sigma$.
 We can do this perturbation so that $\gamma_1$ remains the same out of the linear neighborhood of the singularity and so that the $\alpha$-limit also remains the same ($\sigma$).
 Therefore, $\gamma_1$ still belongs to $C(\sigma)_{Y_1}$ . Thus there is a direction $L_2\in E^2$ so that  $L^2\in\La_{\PP,U}(Y_1)$ for any $U$.
 We can continue this process for all $ 1\leq i\leq l.$

 We conclude that in any small enough $C^1$ neighborhood of $X$ there are a vector fields $Y_{i-1}$ such that $$\La_{\PP,U}(Y_{i-1})\cap E_i\neq\emptyset$$.
 Since the  $C^1$ neighborhood of $X$ can be taken arbitrarily small then, $$\widetilde{\La}(X,U)\cap E_i\neq\emptyset.$$

 Since this is true for any small enough filtrating neighborhood, then $$\widetilde{\La}(C(\sigma))\cap E_i\neq\emptyset.$$

 If the central space splits into only one or two dimensional spaces let us take $L\in\pi_{\PP}(E_i)\cap\widetilde{\La}$ where $E_i$ is two dimensional with complex lyapunov exponents . Since $\widetilde{\La}(C(\sigma))$ then the orbit of $L$ under $\phi_{\PP}^t$, that we note $O(L)$, is such that $O(L)\subset \widetilde{\La}(C(\sigma))$. Since $E_i$ has  complex lyapunov exponents, the direction $L $ is not invariant and $O(L)$ covers all directions of  $E_i$  and therefore $\pi_{\PP}(E_i)\subset\La_{\PP}(X,U)$.

 \endproof

The next corollary implies Theorem  \ref{t.nioki}
\begin{coro}
Let $X$ be a vector field having  a singular chain class $C_{\sigma}$.  We denote $S= Sing(X)\cap C_{\sigma}$ and we suppose that\begin{itemize}
                                                           \item  every $\sigma\in S$  is hyperbolic,
                                                           \item the dimension of the central space of $\sigma\in S$ is locally constant, and the finest Lyapunov  splitting is  into one or two dimensional spaces,
 \end{itemize}

 Then if $\widetilde{\La}(C(\sigma))$ has a hyperbolic structure on the  normal bundle $\cN_L=\cN_1\oplus\dots \oplus\cN_i\oplus\dots\oplus\cN_r$  for the extended linear Poincar\'{e} flow. Then  $B(C(\sigma))$ has the same hyperbolic structure.

 \end{coro}
 \proof
From corollary \ref{c.dom} the dominated splitting in $\widetilde{ \Lambda}$, extends to $B(X,U)$.
So let us consider 	the space $\cN_i$ and the number $d_i$ so that  $$Det (J(h_i^t\cdot (\psi^t_\cN|_{D_i}))$$  presents  a uniform contraction or expansion for any subspace $D_i\subset E_i$ of dimension $d_i$ over the orbits in $\widetilde{ \Lambda}$. We suppose, without loss of generality, that it is a contraction.

Since  $B(C(\sigma))$ and $\widetilde{ \Lambda}$ coincide on the directions that  are not over the singularities, we only                         need to check that for every $\sigma$ and every orbit of an  $L\in B(C(\sigma))\cap \PP_{\sigma}M$ the $Det (J(h_i^t\cdot (\psi^t_\cN|_{D_i}))$ expands uniformly   for any subspace $D_i\subset E_i$ of dimension $d_i$.

 The tangent space of $\sigma$  splits into $$T_{\sigma}M=E^{ss}\oplus E^c\oplus E^{uu}\,,$$ the escaping spaces and the central space. We consider as well the finest Lyapunov splitting over the singularity of $E^c=E_{1}\oplus\dots\oplus  E_{l}$, from lemma \ref{propinc} we have that for every $i\in\set{1\dots l}$, $\pi_{\PP}(E_i)\subset\widetilde{ \Lambda} $.

Then we consider $L\in B(X,U)$ and $u$ a vector in the direction of $L$.
We write $u$ in coordinates of the central space $u=(u_1,\dots,u_i,\dots,u_j\dots u_l)$. We suppose as well that $u_h$ is the first non zero coordinate of $u$ and $u_j$ is the last.
Domination implies that for $t$ sufficiently negatively large $\phi_{\PP}^t(L)$ is in a small cone around $\pi_{\PP}(E_h)$ and remains there, there after.
 For the future $\phi_{\PP}^t(L)$ is in a small cone around $\pi_{\PP}(E_j)$ and remains there, there after.
  Since the contraction and expansion rates extend to the cones around  $\pi_{\PP}(E_j)$  and  $\pi_{\PP}(E_h)$ , and the orbit is outside of this cones only finite time, we get our conclusion.


We are now ready for defining our notion of multisingular hyperbolicity.
\section{Multisingular hyperbolicity}

\subsection{Definition of multisingular hyperbolicity}

\begin{defi}Let $X$ be a $C^1$-vector field on a compact manifold and let $U$ be a compact region.
One says that $X$ is \emph{multisingular hyperbolic} on $U$ if
\begin{enumerate}
 \item Every zero of $X$ in $U$ is hyperbolic. We denote $S= Sing(X)\cap U$.
 \item The restriction of the extended linear Poincar\'e flow $\{\psi_\cN^t\}$ to the extended maximal invariant set $B(X,U)$ admits
 a dominated splitting $N_L=E_L\oplus F_L$.
 \item There is a subset $S_E\subset S$ so that the reparametrized cocycle $h_E^t\psi_\cN^t$ is uniformly contracted in restriction to the
 bundles $E$ over  $B(X,U)$ where $h_E$ denotes
 $$h_E=\Pi_{\sigma\in S_E} h_\sigma.$$
 \item There is a subset $S_F\subset S$ so that the reparametrized cocycle $h_F^t\psi_\cN^t$ is uniformly expanded in restriction to the
 bundles $F$ over  $B(X,U)$ where $h_F$ denotes
 $$h_F=\Pi_{\sigma\in S_F} h_\sigma.$$
\end{enumerate}
\end{defi}

\begin{rema}
 The subsets $S_E$ and $S_F$ are not necessarily uniquely defined, leading to several notions of multisingular hyperbolicity.
 We can also modify slightly this definition allowing to consider the product of power of the $h_\sigma$.  In that case
 $\tilde h_E$ would be on the form
 $$h^t_E=\Pi_{\sigma\in S_E} (h^t_\sigma)^{\alpha_E(\sigma)}$$
 where $\alpha_E(\sigma)\in\RR$.
\end{rema}

\begin{theo}Let $X$ be a $C^1$-vector field on a compact manifold $M$ and let $U\subset M$ be a compact region.  Assume that  $X$ is multisingular
hyperbolic on $U$. Then $X$ is a star flow on $U$, that is, there is a $C^1$-neighborhood $\cU$ of $X$ so that every
periodic orbit contained in $U$ of a vector field  $Y\in \cU$ is hyperbolic.  Furthermore $Y\in \cU$ is multisingular hyperbolic in $U$.
\end{theo}
\begin{proof}
Recall that the extended maximal invariant set $B(Y,U)$ varies upper semicontinuously with $Y$  for the $C^1$-topology .
Therefore, according to Proposition~\ref{p.robustcocycle} there is a $C^1$-neighborhood $\cU_0$ of $X$ so that, for every $Y\in\cU_0$
the extended linear Poincr\'e flow $\psi^t_{\cN,Y}$ admits a dominated splitting $E\oplus_< F$ over $B(Y,U)$, whose dimensions are independent of $Y$ and whose bundles
vary continuously with $Y$.

Now let $S_E$ and $S_F$ be the sets of singular point of $X$ in the definition of singular hyperbolicity.  Now Lemma~\ref{l.representativecocycle} allows us to choose two
 families of cocycles $h_{E,Y}^t$ and $h_{F,Y}^t$ depending continuously on $Y$ in a small neighborhood $\cU_1$ of $X$ and which belongs to the product of the cohomology class
 of cocycles associated to the singularities in $S_E$ and $S_F$, respectively.  Thus the linear cocycles
 $$h_{E,Y}^t\cdot \psi_{\cN,Y}^t|_{E,Y}, \quad \mbox{over } B(Y,U)$$
 varies continuously with $Y$ in $\cU_1$, and is uniformly contracted for $X$.  Thus, it is uniformly contracted for $Y$ in a $C^1$-neighborhood of $X$.

 One shows in the same way that $$h_{F,Y}^t\cdot \psi_{\cN,Y}^t|_{F,Y}, \quad \mbox{over } B(Y,U)$$
 is uniformly expanded for $Y$ in a small neighborhood of $X$.

 We just prove that there is a neighborhood $\cU$ of $X$ so that $Y\in\cU$ is multisingular hyperbolic in $U$.

 Consider a (regular) periodic orbit $\gamma$ of $Y$ ant let $\pi$ be its period.  Just by construction of the cocycles $h_E$ and $h_F$, one check that
 $$h_E^\pi(\gamma(0))=h_F^\pi(\gamma(0))=1.$$
 One deduces that the linear Poincar\'e flow is uniformly hyperbolic along $\gamma$ so that $\gamma$ is hyperbolic, ending the proof.

\end{proof}

\subsection{The multisingular hyperbolic structures over a singular point}

The aim of this subsection is next proposition
\begin{prop}\label{p.singularites} Let $X$ be a $C^1$-vector field on a compact manifold and $U\subset M$ a compact region. Assume that
$X$ is multisingular hyperbolic in $U$ and let $i$ denote the dimension of the stable bundle of the reparametrized extended linear Poincar\'e flow.

let $\sigma$  be a zero of $X$.
Then

\begin{itemize}
 \item either at least one entire invariant (stable or unstable) manifold of $\sigma$ is escaping from $U$.
 \item or $\sigma$ is Lorenz like, more precisely
 \begin{itemize}
 \item either the stable index is $i+1$, the central space $E^c_{\sigma,U}$ contains exactly one stable direction $E^s_1$ (dim $E^s_1=1$)
 and $E^s_1\oplus E^u(\sigma)$ is sectionally dissipative;  in this case $\sigma\in S_F$.
\item or the stable index is $i$, the central space $E^c_{\sigma,U}$ contains exactly one unstable direction $E^u_1$ (dim $E^u_1=1$)
 and $E^s(\sigma)\oplus E^u_1$ is sectionally contracting; then $\sigma\in S_E$.
 \end{itemize}
\end{itemize}
\end{prop}

Note that in the first case of this proposition  the class of the singularity must be trivial. If it was not, the regular orbits of the class that accumulate on  $\sigma$,  entering $U$, would accumulate on an orbit of the stable manifold. Therefore the stable manifold could not be completely escaping. The same reasoning holds for the unstable manifold.

Let $E^s_k\oplus_{<}\cdots \oplus_< E^s_1\oplus_{<}E^u_1\oplus_<\cdots \oplus_< E^u_\ell$ be the finest dominated splitting of the flow over
$\sigma$.
For the proof, we will assume, in the rest of the section that the class of $\sigma$ is not trivial, and therefore we are not in the first case of our previous proposition.
In other word, we assume that  there are $a>0,b>0$ so that
$$E^c_{\sigma, U} =E^s_a\oplus_{<}\cdots \oplus_< E^s_1\oplus_{<}E^u_1\oplus_<\cdots \oplus_< E^u_b.$$

We assume that $X$ is multisingular hyperbolic of $s$-index $i$ and we denote by $E\oplus_<F$ the corresponding
dominated splitting of the extended linear Poincar\'e flow over $B(X,U)$.

The following lemma is a direct consequence of \ref{l.central}
\begin{lemm}\label{l.iab} et $X$ be a $C^1$-vector field on a compact manifold and $U\subset M$ a compact region. Assume that
$X$ is multisingular hyperbolic in $U$ and let $i$ denote the dimension of the stable bundle of the reparametrized extended linear Poincar\'e flow.

let $\sigma$  be a zero of $X$.
Then with the notation above,
\begin{itemize}
 \item either $i=dim E\leq dim E^s_k \oplus \dots \oplus  dim E^s_{a+1}$ (i.e. the dimension of $E$ is smaller or equal than the dimension of the biggest stable escaping space).
 \item or $dim M-i-1= dim F\leq dim E^u_\ell\oplus \dots \oplus  dim E^u_{b+1}$ (i.e. the dimension of $F$ is smaller or equal than the dimension of the biggest unstable escaping space).
\end{itemize}
\end{lemm}

According to Lemma~\ref{l.iab} we now assume that $i\leq dim E^s_k+ \oplus +dim E^s_{a+1}$ (the other case is analogous, changing $X$ by $-X$).

\begin{lemm}
 With the hypothesis above, for every $L\in \PP^c_{\sigma,U}$ the projection of $E^c_{\sigma, U}$ on the normal space $N_L$ is contained in $F(L)$.
\end{lemm}
\begin{proof}It is because  the projection of $E^s_k\oplus\cdots \oplus E^s_{a+1}$ has dimension at least  the dimension $i$ of $E$ and hence
contains $E(L)$.  Thus the projection of $E^c_{\sigma, U}$ is transverse to $E$.  As the projection of $E^s_{\sigma,U}$ on $N_l$ defines a
$\psi^t_\cT$-invariant bundle over the $\phi^t_\cT$-invariant compact set $\PP^c_{\sigma,U}$, one concludes that the projection is contained in $F$.
\end{proof}

As a consequence the bundle $F$ is not uniformly expanded on $\PP^c_{\sigma, U}$ for the extended linear Poincar\'e flow. As it is expanded
by the reparametrized flow, this implies $\sigma\in S_F$.

Consider now $L\in E^s_a$. Then $\psi^t_{\cN}$ in restriction to the projection of $E^c_{\sigma,U}$ on $N_L$ consists in multiplying the
natural action of the derivative by the exponential contraction along $L$.  As it is included in $F$, the multisingular hyperbolicity
implies that it is a uniform expansion: this means that
\begin{itemize}
 \item $L$ is the unique contracting direction in $E^s_{\sigma,U}$: in other words, $a=1$ and $dim E^s-a=1$.
 \item the contraction along $a$ is less than the expansion in the $E^u_j$, $j>1$.  In other words $E^c_{\sigma,U}$ is sectionally expanding.
\end{itemize}

For ending the proof of the Proposition~\ref{p.singularites}, it remains to check the $s$-index of $\sigma$: at $L\in E^s_a$ one gets
that $F(L)$ is isomorphic to $E^u_1\oplus\cdots\oplus E^u_\ell$ so that the $s$-index of $\sigma$ is $i+1$, ending the proof.

\section{The multisingular hyperbolicity is a necessary condition for star flows: Proof of theorem 4}
The aim of this section is to prove  Lemma~\ref{l.converse} below:
\begin{lemm}\label{l.converse}
Let $X$ be a generic star vector field on $M$. Consider a chain-recurrent class $C$ of $X$.
Then there is  a filtrating neighborhood  $U$ of  $C$ so that the extended maximal invariant set
$B(X,U)$ is multisingular hyperbolic.
 \end{lemm}

Notice that, as the multisingluar hyperbolicity of $B(X,U)$ is a robust property, Lemma~\ref{l.converse} implies Theorem~\ref{t.converse}.

As already mentioned, the proof of Lemma~\ref{l.converse} consists essentially in recovering  the results in \cite{GSW} and adjusting few of them to the new setting.
So we start by recalling several of the results from or used in \cite{GSW}.

To start we state the following properties of star flows:

\begin{lemm}[\cite{L}\cite{Ma2}]\label{l.uniformlyattheperiod}
For any star vector field $X$ on a closed manifold $M$, there is a $\cC^1$ neighborhood $\cU$ of $X$ and numbers $\eta> 0$
and $T > 0$ such that, for any periodic orbit $\gamma$ of a vector field $Y\in \cU$ and any integer $m>0$, let  $N= N_s\oplus N_u$ be
the stable- unstable splitting of the normal bundle $N$ for the linear Poincar\'e flow $\psi_t^Y$  then:
\begin{itemize}
\item \emph{Domination: }For every $x \in \gamma$ and $t \geq T$, one has $$\frac{\norm{\psi_t^Y\mid_{N_s}}}{\min(\psi_t^Y\mid_{N_u})}\leq e^{-2\eta t}$$
\item\emph{ Uniform hyperbolicity at the period:} at the period If the period $\pi(\gamma)$ is larger than  $T$ then, for every $x\in \gamma$,  one has:

\begin{eqnarray*}
    \Pi^{(m\pi(\gamma)/T)-1}_{i=0}&&\norm{\psi_t^Y\mid_{N_s}(\phi^Y_{iT}(x))}\leq e^{-m\eta \pi(\gamma)}
\end{eqnarray*}
and
\begin{eqnarray*}
    \Pi^{(m\pi(\gamma)/T)-1}_{i=0}&&\min(\psi_t^Y\mid_{N_u}(\phi^Y_{iT}(x)))\geq e^{m\eta \pi(\gamma)}\,.
\end{eqnarray*}
    Here $\min(A)$ is the mini-norm of $A$, i.e., $\min(A) = \norm{A^{-1}}^{-1}$.
\end{itemize}
\end{lemm}

Now we need some
generic properties for flows:
\begin{lemm}[\cite{C}\cite{BGY}]\label{l.generic}
There is a $C^1$-dense $G_{\delta}$ subset $\cG$ in the $C^1$-open set of star flows of $M$ such that, for every $X\in\cG$, one has:
\begin{itemize}
\item Every critical element (zero or periodic orbit) of $X$ is hyperbolic and therefore admits a well defined continuation in a $C^1$-neighborhood of $X$.
\item For every critical element $p$ of $X$, the Chain Recurrent Class $C(p)$ is continuous at $X$ in the Hausdorff topology;
    \item If $p$ and $q$ are two critical elements of $X$, such that $C(p)=C(q)$ then there is a $C^1$ neighborhood $\cU$ of $X$
    such that the chain recurrent class of $p$ and $q$  still coincide for every $Y\in\cU$
            \item For any nontrivial chain recurrent class $C$ of $X$, there exists a sequence of periodic orbits $Q_n$ such that $Q_n$ tends to $C$ in the Hausdorff topology.
\end{itemize}
\end{lemm}

\begin{lemm}[\label{4.2}lemma 4.2 in \cite{GSW}]
Let $X $ be a star flow in $M$ and $\sigma\in Sing(X)$. Assume that the Lyapunov exponents of
$\phi_t(\sigma)$ are $$\lambda_1\leq\dots\leq\lambda_{s-1}\leq\lambda_s <0<\lambda_{s+1}\leq\lambda_{s+1}\leq\dots\leq \lambda_d$$.
If the chain recurrence class$C(\sigma)$ of $\sigma$,
is nontrivial, then:
\begin{itemize}
\item either $\lambda_{s-1}\neq\lambda_s$ or $\lambda_{s+1}\neq\lambda_{s+2}$.
\item if  $\lambda_{s-1}=\lambda_s$, then $\lambda_s +\lambda_{s+1}<0$.
\item if $\lambda_{s+1}=\lambda_{s+2}$, $\lambda_s +\lambda_{s+1}>0$.
\item  if $\lambda_{s-1}\neq\lambda_s$ and $\lambda_{s+1}\neq\lambda_{s+2}$, then $\lambda_s +\lambda_{s+1}\neq0$.
\end{itemize}
\end{lemm}

We say that a singularity $\sigma$ in the conditions of the previous lema is \emph{Lorenz like} of index $s$
and we define the \emph{saddle value} of a singularity as the value
$$sv(\sigma)=\lambda_s +\lambda_{s+1}.$$

Consider a Lorenz like singularity $\sigma$, then:
\begin{itemize}
\item if $sv(\sigma)>0$, we consider the splitting $$T_\sigma M= G^{ss}_{\sigma}\oplus G^{cu}_{\sigma}$$
where (using the notations of Lemma~\ref{4.2}) the space $G^{ss}_\sigma$
corresponds to the Lyapunov exponents $\lambda_1$ to $\lambda_{s-1}$, and $G^{cu}_\sigma$ corresponds to the Lyapunov exponents $\lambda_s, \dots,\lambda_d$.

\item if $sv(\sigma)<0$, we consider the splitting $$T_\sigma M= G^{cs}_\sigma\oplus G^{uu}_\sigma$$
where  the space $G^{cs}_\sigma$
corresponds to the Lyapunov exponents $\lambda_1$ to $\lambda_{s+1}$, and $G^{uu}_\sigma$ corresponds to the Lyapunov exponents $\lambda_{s+2}, \dots,\lambda_d$.
\end{itemize}

\begin{coro}\label{c.Lorenz-like} Let $X$ be a vector field  and $\sigma$ be a Lorenz-like singularity of $X$ and let
$h_\sigma\colon \La_X\times \RR\to (0,+\infty)$ be a cocycle
in the cohomology class $[h_\sigma]$ defined in Section~\ref{ss.reparametrization}.
\begin{enumerate}\item  First assume that  $Ind(\sigma)=s+1$ and $sv(\sigma)>0$. Then the restriction of $\psi_\cN$ over $\PP G^{cu}_\sigma$
admits a dominated splitting $N_L=E_L\oplus F_L$, with $dim (E_L)=s$, for  $L\in \PP G^{cu}_\sigma$.  Furthermore,
\begin{itemize}
\item $E$ is uniformly contracting for $\psi_\cN$
\item $F$ is uniformly expanding for the reparametrized extended linear Poincar\'e flow $ h_\sigma\cdot \psi_\cN$.
\end{itemize}

\item Assume now that  $Ind(\sigma)=s$ and $sv(\sigma)<0$.
One gets a dominated splitting $N_L=E_L\oplus F_L$ for $\phi_\cN$  for $L\in \PP G^{cs}_\sigma$ so that $dim(E_L)= s$, the bundle
$F$ is uniformly expanded under $\psi_\cN$ and $E$ is uniformly contracted by $h_\sigma \cdot\psi_\cN$.
\end{enumerate}
\end{coro}
\begin{proof}We only consider the first case  $Ind(\sigma)=s+1$ and $sv(\sigma)>0$, the other is analogous and can be deduced by reversing the time.

We consider the restriction of $\psi_\cN$ over $\PP G^{cu}_\sigma$, that is, for point $L\in\tilde \La_X$ corresponding to lines contained in $G^{cu}_\sigma$.
Therefore the normal space
$N_L$ can be identified, up to a projection which is uniformly bounded, to the direct sum of $G_\sigma^{ss}$
with the normal space of $L$ in $G^{cu}_\sigma$.

Now we fix $E_L=G_\sigma^{ss}$ and $F_L$ is the normal space of $L$ in $G^{cu}_\sigma$.  As $G_\sigma^{ss}$ and $G^{cu}_\sigma$ are invariant
under the derivative of the flow $\phi_t$, one gets that the
splitting $N_L=E_L\oplus F_L$ is invariant under the extended linear Poincar\'e flow over $\PP G^{cu}_\sigma$.

Let us first prove that this splitting is dominated:

By Lemma \ref{4.2}
if we choose a unit vector $v$ in $E_L$ we know that for any $t>0$  one has
$$\norm{\psi^t_{\cN}(v)}\leq Ke^{t\lambda_{s-1}}\,.$$

Now let us choose a unit vector $u$ in $F_L$, and consider $w_t=\psi_\cN^t(u)\in F_{\phi_\PP^t(L)}$.
Then for any $t>0$, one has
$${\norm{D\phi^{-t}(w_t)}}\leq K'e^{t(-\lambda_s)}\norm{w_t}\,.$$
The extended linear Poincar\'e flow is obtained by projecting the image by the derivative of the flow on the normal bundle.
Since the projection on the normal space does not increase the norm of the vectors, one gets
$${\norm{\psi^{-t}_{\cN}(w_t)}}\leq K'e^{t(-\lambda_s)}\norm{w_t}\,,$$
is other words
$$\frac1{\norm{\psi_\cN^t(u)}}\leq  K'e^{t(-\lambda_s)}$$


Putting together these inequalities one gets:

$$\frac{\norm{\psi^t_{\cN}(v)}}{\norm{\psi^t_{\cN}(u)}}\leq KK'e^{t(\lambda_{s-1}-\lambda_s)}\,.$$
This provides the domination as $\lambda_{s-1}-\lambda_s<0$.

Notice that $E_L= G_\sigma^{ss}$ is uniformly contracted by the extended linear Poincar\'e flow $\psi_\cN$,
because it coincides, on $G_\sigma^{ss}$  and for  $L\in\PP G^{cu}_\sigma$,
with the   differential of the flow $\phi^t$.
For concluding the proof, it remains to show that the reparametrized extended
linear Poincar\'e flow $h_\sigma \cdot \psi_\cN$ expands uniformly the vectors in $F_L$, for $L\in \PP G^{cu}_\sigma$.

Notice that, over the whole  projective space  $\PP_\sigma$, the cocycle $h_{\sigma,t}(L)$ is the rate of expansion of the derivative of $\phi_t$
in the direction of $L$. Therefore
$h_\sigma\cdot \psi_\cN$ is defined as follows: consider a line $D\subset N_L$.  Then the expansion rate of the restriction of
$h_\sigma\cdot \psi_\cN$ to $D$ is the expansion rate of
the area on the plane spanned by $L$ and $D$ by the derivative of $\phi_t$.

The hypothesis $\lambda^s+\lambda^{s+1}>0$ implies that the derivative of $\phi_t$ expands uniformly the area on the planes
contained in $G^{cu}_\sigma$, concluding.

\end{proof}

\begin{lemm}[\label{4.5} Lemma 4.5 and Theorem 5.7 in \cite{GSW}]
Let $X$ be a $C^1$ generic star vector field and let $\sigma \in Sing(X)$.
Then there is a filtrating neighborhood $U$ of $C(\sigma)$ so that,
for every two  periodic points $p,q \subset U$,

$$Ind(p) = Ind(q),$$
Furthermore, for any singularity $\sigma'$ in $U$,
$$Ind(\sigma') = Ind(q) \hspace{0.5cm}\text{ if } sv(\sigma)<0\,, $$
or
$$Ind(\sigma') = Ind(q)+1 \hspace{0.5cm}\text{ if } sv(\sigma)>0\,. $$
\end{lemm}

\begin{lemm}\label{l.lifted-dominated}
There is a dense $G_{\delta}$ set $\cG$ in the set of star flows of $M$ with the following properties:
Let $X$ be in $\cG$, let $C$ be a chain recurrent class of $X$.
Then there is a (small) filtrating neighborhood $U$ of $C$ so that the lifted maximal invariant set $\widetilde{\Lambda}(X,U)$ of $X$ in $U$ has a dominated splitting $\cN=E\oplus_\prec F$ for the
extended linear Poincar\'e flow, so that $E$ extends the stable bundle for every periodic orbit $\gamma$ contained in $U$.
\end{lemm}
\begin{proof}
According to Lemma~\ref{4.5}, the class $C$ admits a filtrating neighborhood $U$ in which the periodic orbits are hyperbolic and with the same index.
On the other hand, according to  Lemma~\ref{l.generic}, every
chain recurrence class in $U$ is accumulated by periodic orbits.
Since $X$ is a star flow, Lemma~\ref{l.uniformlyattheperiod} asserts that the normal bundle over the union of these periodic orbits
admits a dominated splitting for the linear Poincar\'e flow, corresponding to
their stable/unstable splitting. It follows that the union of the corresponding orbits in
the lifted maximal invariant set have a dominated splitting for $\cN$.  Since any dominated splitting defined on an invariant set extends to the
closure of this set,   we have a dominated splitting
on the closure of the lifted periodic orbits, and hence on the whole $\widetilde{\Lambda}(X,U)$.
\end{proof}

Lemma~\ref{l.lifted-dominated} asserts that the lifted maximal invariant set $\tilde\La(X,U)$ admits a dominated splitting.  What we need now is extend this
dominated splitting to the extended maximal invariant  set $$B(X,U)= \tilde\La(X,U)\cup \bigcup_{\sigma_i\in Zero(X)\cap U} \PP^c_{\sigma_i,U}.$$

Now we need the following theorem to have more information on the projective center spaces $\mathbb{P}^c_{\sigma_i,U}$.
\begin{lemm}[\label{4.7} lemma 4.7 in \cite{GSW}]
Let $X$ be a star flow in $M$ and $\sigma$ be a singularity of $X$ such that $C(\sigma)$ is nontrivial.
Then :
\begin{itemize} \item if $sv(\sigma)>0$,  one has:
$$W^{ss}(\sigma)\cap C(\sigma)= \{\sigma\}\,,$$
where $W^{ss}(\sigma)$ is the strong stable manifold associated to the space $G_\sigma^{ss}$.

\item if $sv(\sigma)<0$, then:
$$W^{uu}(\sigma)\cap C(\sigma)=\{\sigma\}\,,$$
where $W^{uu}(\sigma)$ is the strong unstable manifold associated to the space $G_\sigma^{uu}$.
\end{itemize}
\end{lemm}

\begin{rema}\label{r.escapeclass}Consider a vector field $X$ and a hyperbolic singularity $\sigma$ of $X$.
Assume that $W^{ss}(\sigma)\cap C(\sigma)= \{\sigma\}\,,$ for a strong stable manifold $W^{ss}(\sigma)$,
where $C(\sigma)$ is the chain recurrence class of $\sigma$.

Then there is a filtrating neighborhood $U$ of $C(\sigma)$
on which the strong stable manifold $W^{ss}(\sigma)$ is escaping from $U$ (see the definition in Section~\ref{ss.centerspace}).
\end{rema}
\begin{proof} Each orbit in $W^{ss}(\sigma)\setminus \{\sigma\}$ goes out some filtrating neighborhood of $C(\sigma)$ and
the nearby orbits go out of the same filtrating neighborhood. Notice that the space of orbits in $W^{ss}(\sigma)\setminus \{\sigma\}$ is compact, so that we can consider
a finite cover of it by open sets for which the corresponding
orbits go out a same filtrating neighborhood of $C(\sigma)$.
The announced filtrating neighborhood is the intersection of these finitely many filtrating neighborhoods.
\end{proof}

Remark~\ref{r.escapeclass} allows us to consider the \emph{escaping strong stable} and \emph{strong unstable manifold}
of a singularity $\sigma$ without refereing to a specific filtrating neighborhood $U$ of the
class $C(\sigma)$: these notions do not depend on $U$ small enough.
Thus the notion of the \emph{center space $E^c_\sigma=E^c(\sigma,U)$} is also independent of $U$ for $U$ small enough.
Thus we will denote
$$\PP^c_\sigma=\PP^c_{\sigma,U}$$
for $U$ sufficiently small neighborhood of the chain recurrence class $C(\sigma)$.

\begin{rema}\label{r.central}
Lemma~\ref{4.7} together with Remark~\ref{r.escapeclass} implies that:
\begin{itemize}\item if  $sv(\sigma)>0$, then the center space $E^c_\sigma$ is contained in $G^{cu}$
\item if $sv(\sigma)<0$, then $E^c_\sigma\subset G^{cs}$.
\end{itemize}

\end{rema}

\begin{lemm}\label{domination}
Let $X$ be a generic star vector field on $M$. Consider a chain recurrent class $C$ of $X$.
Then there is a neighborhood   $U$ of $C$ so that the extended maximal invariant set $B(X,U)$ has a
 dominated splitting for the extended linear Poincar\'e flow $$\cN_{B(X,U)}=E\oplus_{\prec}F$$
which extends the stable-unstable bundle defined on the lifted maximal invariant set $\widetilde{\La}(X,U)$.
\end{lemm}
\begin{proof} The case where $C$ is not singular is already done.  According to Lemma~\ref{4.5}  there an integer $s$ and  a neighborhood $U$ of $C$ so that
every periodic orbit in $U$ has index $s$ and every singular point $\sigma$ in $U$ is Lorenz like, furthermore either its index is $s$ and $sv(\sigma)<0$
or $\sigma$ has index $s+1$ and $sv(\sigma)>0$.

According to Remark~\ref{r.central}, one has:
$$B(X,U)\subset \tilde \La(X,U)\cup\bigcup_{sv(\sigma_i)<0} \PP G^{cs}_{\sigma_i}\cup \bigcup_{sv(\sigma_i)>0}
\PP G^{cu}_{\sigma_i}$$
By  Corollary~\ref{c.Lorenz-like} and Lemma~\ref{l.lifted-dominated} each of this set admits a dominated splitting  $E\oplus F$ for the extended linear
Poincar\'e flow $\psi_\cN$ with $dim E=s$.

 The uniqueness of the dominated splittings for prescribed dimensions implies that these dominated splitting coincides  on the intersections  concluding.
\end{proof}

 We already proved the existence of a dominated splitting $E\oplus F$, with $dim(E)=s$, for the extended linear Poincar\'e flow over $B(X,U)$ for a small
filtrating neighborhood of $C$, where $s$ is the index of any periodic orbit in $U$.
It remains to show that the extended linear Poincar\'e flow admits a reparametrization which contracts uniformly the bundle $E$ and
a reparametrization which expands the bundle $F$.

Lemma~\ref{4.2} divides the set of singularities in $2$ kinds of singularities, the ones with positive saddle value and the ones with negative saddle value.
We denote  $$S_E:=\{x\in Zero(X)\cap U \text{ such that } sv(x)<0\}\hspace{0,2cm} \text{ and, }$$
$$S_F:=\{x\in Zero(X)\cap U \text{ such that } sv(x)>0\}\,.$$

Recall that Section~\ref{ss.reparametrization} associated  a cocycle $h_\sigma\colon\La_X\to\RR$, whose cohomology class is well defined, to every singular point
$\sigma$.

Now we define $$h_{E}=\Pi_{\sigma\in S_E} h_\sigma \mbox{ and  }h_{F}=\Pi_{\sigma\in S_F} h_\sigma.$$

Now Lemma~\ref{l.converse}  and therefore Theorem~\ref{t.converse} are a direct consequence of the next lemma:

 \begin{lemm}\label{l.contraction}
Let $X$ be a generic star vector field on $M$. Consider a chain recurrent class $C$ of $X$.
Then there is a neighborhood   $U$ of $C$ so that the extended maximal invariant set $B(X,U)$  is such that the normal space has a dominated  splitting $\cN_{B(X,U)}=E\oplus_{\prec}F$ such that the space $E$ (resp. $F$) is uniformly contracting (resp. expanding) for the reparametrized extended linear Poincar\'e flow  $h_E^t\cdot\psi^t_{\cN}$
(resp. $h_F^t\cdot\psi^t_{\cN}$).
\end{lemm}

The proof uses the following theorem by Gan Shi and Wen, which describes the ergodic measures for a star flow.
Given a $C^1$ vector field $X$, an ergodic measure $\mu$ for the flow $\phi_t$,  is said to be \emph{hyperbolic} if either $\mu$ is
supported on a hyperbolic singularity or
$\mu$ has  exactly one zero Lyapunov exponent, whose invariant subspace is spanned by $X$.

\begin{theo}[\label{5.6}lemma  5.6 \cite{GSW}]
Let $X$ be a star flow. Any invariant ergodic measure $\mu$ of the flow $\phi_t$ is a hyperbolic measure. Moreover, if $\mu$ is not the atomic measure
on any singularity, then $supp(\mu) \cap H(P)\neq \emptyset$, where $P$ is a periodic orbit with the index of $\mu$, i.e., the number of negative Lyapunov
exponents of $\mu$(with multiplicity).
\end{theo}

\begin{proof}[ of lemma \ref{l.contraction}]

We argue by contradiction, assuming that the bundle $E$ is not uniformly contracting for $h_E\cdot \psi^t_{\cN}$ over
$B(X,U)$ for every filtrating neighborhood $U$ of the class $C$.

One deduces the following claim:
\begin{clai*} Let $\tilde C(\sigma)\subset \tilde \La(X)$ be the closure in $\PP M$ of the lift of $C(\sigma)\setminus Zero(X)$.
Then, for every $T>0$, there
exists an ergodic invariant measure $\mu_T$ whose support is contained in $\PP^c_\sigma\cup \tilde C(\sigma)$ such that
$$\int
\log\norm{h^T_E.\psi^T_{\cN} \mid_{E}} d\mu(x)\geq 0\,.$$
 \end{clai*}
\begin{proof}
For all $U$, there exist an ergodic measure  $\mu_T$ whose support is contained in $B(X,U)$ such that
$$\int\log\norm{h^T_E.\psi^T_{\cN} \mid_{E}} d\mu_T(x)\geq 0\,.$$
 But note that the class $C$, needs not to be a priori a maximal invariant set in a neighborhood $U$. We fix this by observing the fact that $$\PP^c_\sigma\cup \tilde C(\sigma)\subset B(X,U)$$ for any $U$ as small as we want and actually  we can choose a sequence of neighborhoods $\set{U_n}_{n\in\NN}$  such that $U_n \to C $ and therefore $$\PP^c_\sigma\cup \tilde C(\sigma)=\bigcap_{n\in\NN} B(X,U_n)\,.$$
 This defines a sequence of measures $\mu^n_T \to \mu^0_T $ such that
 $$\int\log\norm{h^T_E.\psi^T_{\cN} \mid_{E}} d\mu^n_T(x)\geq 0\,,$$
 and with supports converging to $\PP^c_\sigma\cup \tilde C(\sigma)$.
 The resulting limit measure $\mu^0_T$, whose support is contained in $\PP^c_\sigma\cup \tilde C(\sigma)$, might not be hyperbolic but it is invariant. We can decompose it in sum of ergodic measures, and so if
 $$\int\log\norm{h^T_E.\psi^T_{\cN} \mid_{E}} d\mu^0_T(x)\geq 0\,,$$
 There must exist an ergodic measure $\mu_T$, in the ergodic decomposition of $\mu^0_T$,
  $$\int\log\norm{h^T_E.\psi^T_{\cN} \mid_{E}} d\mu_T(x)\geq 0\,,$$
  and the support of  $\mu_T$ is contained in $\PP^c_\sigma\cup \tilde C(\sigma)$.
\end{proof}

Recall that for generic star flows, every chain recurrence class in $B(X,U)$ is Hausdorff limit of periodic orbits  of the same index and that
satisfy the conclusion of Lemma~\ref{l.uniformlyattheperiod}. Let $\eta>0$ and  $T_0>0$ be given by Lemma~\ref{l.uniformlyattheperiod}.  We consider
an ergodic measure
$\mu=\mu_T$ for some $T>T_0$.

\begin{clai*}Let  $\nu_n$ be a measure supported on a periodic orbits $\gamma_n$ with period $\pi \gamma_n$ bigger than $T$ , then $\int\log h^T_E d\nu_n(x)=0$.

\end{clai*}
\begin{proof}
By definition of $h^T_E$
$$\log h^T_Ed\nu_n(x)=\log\Pi_{\sigma_i\in S_E}\norm{h^T_{\sigma_i}}d\nu_n(x)\,,$$
so it suffices to prove the claim for a given $h^T_{\sigma_i}$.
For every $x$ in $\gamma$ by the cocycle condition in lemma \ref{l.hcocycle} we have that

\begin{eqnarray*}
 \Pi^{(m\pi(\gamma)/T)-1}_{i=0}&&h^T_{\sigma_i}(\phi^Y_{iT}(x))=h^{(m\pi(\gamma)/T)-1}_{\sigma_i}(x)
\end{eqnarray*}

The norm of the vector field restricted to $\gamma$ is bounded, and therefore $h^{(m\pi(\gamma)/T)-1}_{\sigma_i}(x)$ is bounded for $m\in\NN$ going to infinity.
Then this is also true for $h^T_E$.
Since $\nu_n$ is an ergodic measure, we have that
\begin{eqnarray*}
\int\log h^T_E d\nu_n(x)&=&\lim_{m\to\infty}\frac{1}{m}\sum^{(m\pi(\gamma)/T)-1}_{i=0}\log \left(h^T_E(\phi^Y_{iT}(x))\right)\\
   &=&\lim_{m\to\infty}\frac{1}{m}\log \left(\Pi^{(m\pi(\gamma)/T)-1}_{i=0} h^T_{E}(\phi^Y_{iT}(x))\right)\\
   &=&\lim_{m\to\infty}\frac{1}{m}\log\left(h^{(m\pi(\gamma)/T)-1}_E(x)\right)\\
   &=&0
\end{eqnarray*}
\end{proof}

\begin{clai*}There is a singular point $\sigma_i$ so that $\mu$ is supported on $\PP^c_{\sigma_i}$.
\end{clai*}
\begin{proof}

Suppose that  $\mu$ weights $0$ on $\bigcup_{\sigma_i\in Zero(X)}\PP^c_{\sigma_i}$.
then $\mu$ projects on $M$ on an ergodic measure $\nu$ supported on the class $C(\sigma)$ and such that ut weights $0$ in the singularities, for which
$$\int \log\norm{h_E.\psi^T_{\cN} \mid_{E}} d\nu(x)\geq 0 .$$

Recall that $\psi^T $ is the linear Poincar\'{e} flow, and $h^T_E$ can be defined as a function of $x\in M$ instead of as a function of $L\in\PP M$ outside of an arbitrarily small neighborhood of the singularities.

 However, as $X$ is generic, the ergodic closing lemma implies that  $\nu$ is the weak$*$-limit of measures $\nu_n$ supported on periodic orbits $\gamma_n$
 which converge for the Hausdorff distance to the support of $\nu$. Therefore, for $n$ large enough, the $\gamma_n$ are contained in
 a small filtrating neighborhood of $C(\sigma)$ therefore satisfy
 $$\int
\log\norm{h^T_E.\psi^T \mid_{E}} d\nu_n(x)\leq -\eta.$$

The map $\log\norm{h^T_E.\psi^T\mid_{E}}$ is not continuous. Nevertheless, it is uniformly bounded and the unique discontinuity points are
the singularities of $X$. These singularities have (by assumption) weight $0$ for $\nu$ and thus admit neighborhoods with arbitrarily small weight.
Out of such a neighborhood the map is continuous.  One deduces that
$$\int\log\norm{h^T_E.\psi^T \mid_{E}} d\nu(x)=\lim\int\log\norm{h_E.\psi^T\mid_{E}} d\nu_n(x)$$
and therefore is strictly negative, contradicting the assumption.
This contradiction proves the claim.

\end{proof}

On the other hand, Corollary~\ref{c.Lorenz-like} asserts that $h_E\cdot \psi_\cN$ contracts uniformly the bundle $E$
\begin{itemize}
\item over  the projective space $\PP G^{cs}_{\sigma_i}$,  for $\sigma_i$ with $sv(\sigma_i)<0$ : note that, in this case, $\sigma_i\in S_E$ so that
$h_E$ coincides with $h_{\sigma_i}$ on $\PP G^{cs}_{\sigma_i}$;
\item over $\PP G^{cu}_{\sigma_i}$ for $\sigma_i$ with $sv(\sigma_i)>0$:
note that, in this  case $\sigma_i\notin S_E$ so that  $h^t_E$ is constant equal to $1$ on
$\PP G^{cu}_{\sigma_i}\times \RR$.
\end{itemize}

Recall that $\PP^c_{\sigma_i}$ is contained in $\PP G^{cs}_{\sigma_i}$ (resp. $\PP G^{cu}_{\sigma_i}$)  if $sv(\sigma_i)<0$ (resp. $sv(\sigma_i)>0)$.  One deduces
 that there is $T_1>0$  and $\varepsilon>0$ so that
$$
\log\norm{h_E.\psi^T_{\cN} \mid_{E_L}}\leq -\varepsilon, \quad \forall L\in \PP^c_{\sigma_i} \mbox{ and } T>T_1 .$$

Therefore the measures $\mu_T$, for $T>\sup{T_0,T_1}$ cannot be supported on $\PP^c_{\sigma_i}$, leading to a contradiction.

The expansion for $F$ is proved analogously.

And this finishes the proof of Lemma \ref{l.contraction} and therefore the proof of Lemma~\ref{l.converse} and Theorem~\ref{t.converse}.
\end{proof}

\section{A multisingular hyperbolic set in $\RR^3$}\label{s.example}
This section will build of a chain recurrence class  in $M^3$ containing two singularities of different indexes, that will be multisingular hyperbolic.  However this will not be a robust class, and the singularities will not be robustly related.
Other examples of this kind are exhibited in \cite{BaMo}.
 A robust example is announced by the second author on a 5-dimensional manifold,


We add this example since it illustrates de situation in the simplest way we could.

\begin{theo}\label{t.wexample}
There exists a vector field $X$  in $S^2\times S^1$ with an isolated chain recurrent class  $\La$ such that :
\begin{itemize}

\item There are 2 singularities in  $\La$. They are  Lorenz like and of different index.
\item  There is cycle between the singularities. The cycle and the singularities  are the only orbits in $\La$.
\item The set $\La$ is multisingular hyperbolic.

\end{itemize}
\end{theo}

To begin with the proof of the theorem, let us start with the construction of a vector field $X$, that we will later show that it has the properties of the Theorem \ref{t.wexample}.

\begin{figure}[htb]
\begin{center}
\includegraphics[width=0.35\linewidth]{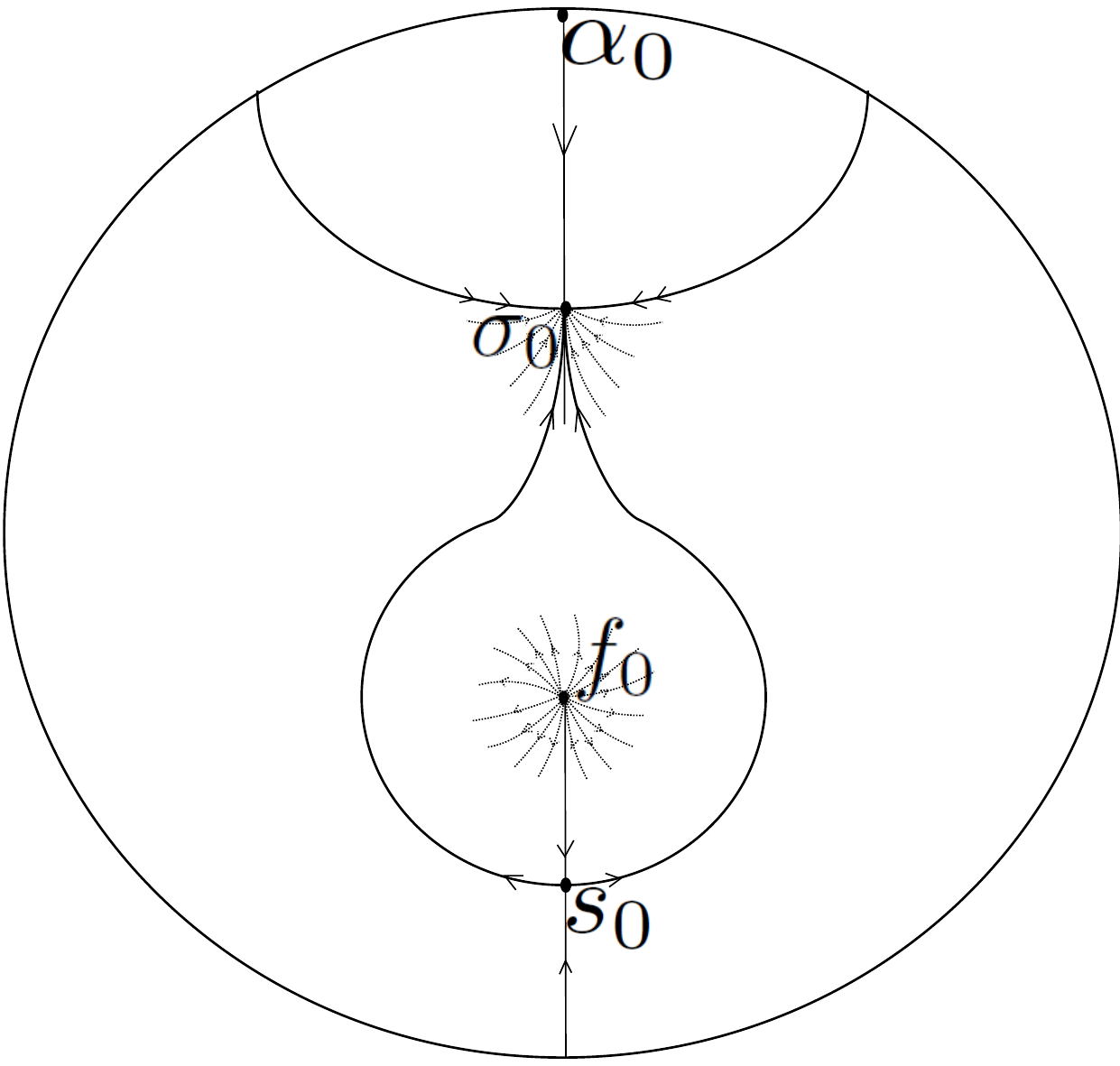}
\end{center}
\caption{The vector field in $S^2$ }
\label{s}
\end{figure}

We consider a vector field  in $S^2$ having:
 \begin{itemize}
 \item A source $f_0$
such that the basins of  repulsion  of $f_0$ is a disc bounded by a cycle $\Gamma$ formed by the  unstable manifold of a saddle $s_0$ and a sink $\sigma_0$.
 \item A source $\alpha_0$ in the other component limited by $\Gamma$.
 \item We require that the tangent at  $\sigma_0$ splits into 2 spaces, one having a stronger contraction than the other. The strongest direction is tangent to $\Gamma$ at $\sigma_0$.

\end{itemize}

Note that the unstable manifold of $s_0$, is formed by two orbits. This two orbits have their $\omega$-limit in $\sigma_0$, and as they approach  $\sigma_0$, they become tangent to the weak stable direction,
(see figure \ref{s}).

Now we consider $S^2$ embedded in  $S^3$, and we define a vectorfield $X_0$ in $S^3$ that is normally hyperbolic at  $S^2$, in fact we have $S^2$  times a strong contraction, and 2 extra sinks that we call $\omega_0$ and $P_0$ completing the dynamics (see figure  \ref{hola}).

\begin{figure}[htb]
\begin{center}
\includegraphics[width=0.44\linewidth]{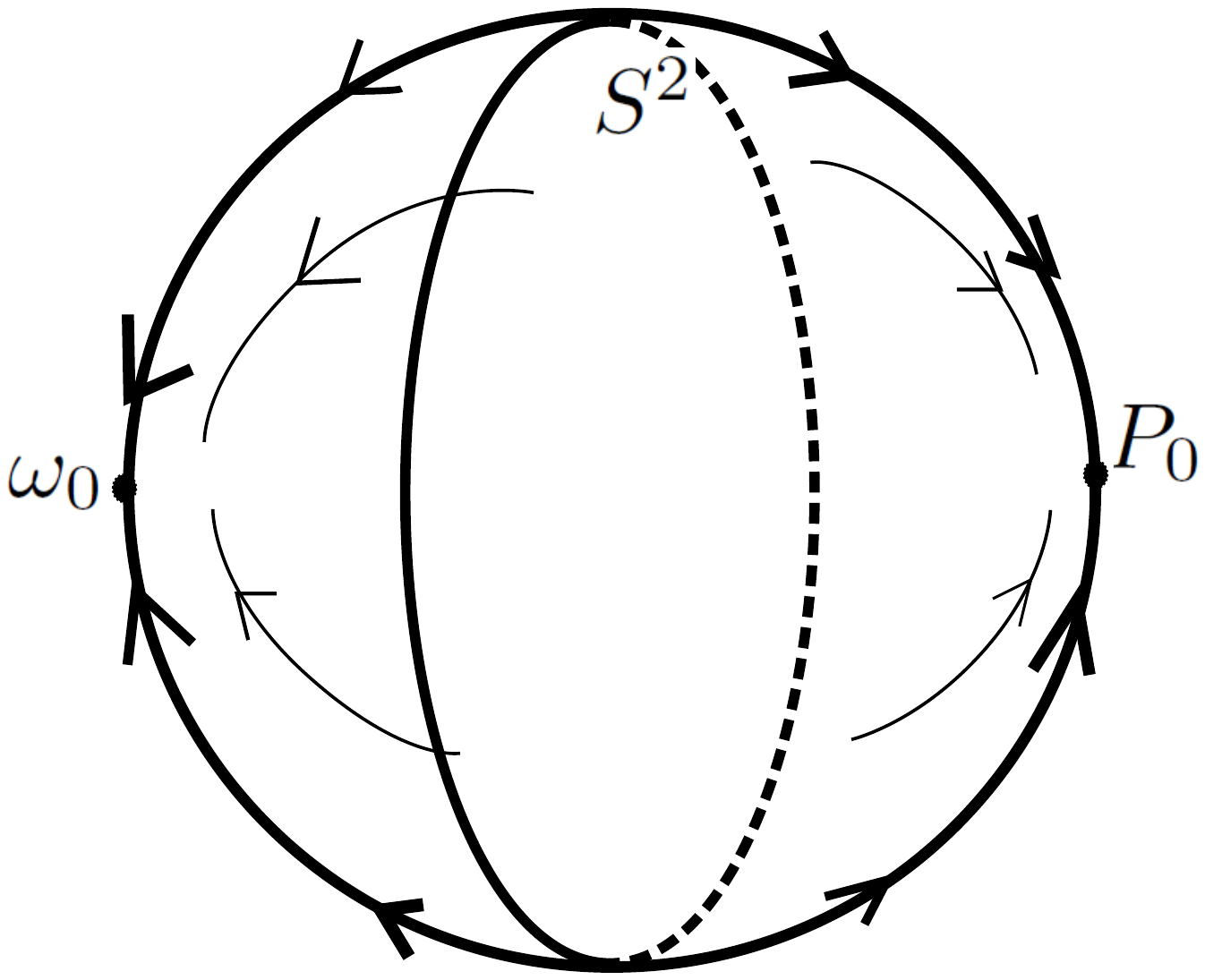}
\end{center}
\caption{ $S^2$ normally repelling  in $S^3$ }
\label{hola}
\end{figure}

Note that  $\sigma_0$  is now a saddle. We require that the weaker contraction at $\sigma_0$ is weaker than the expansion. So $\sigma_0$  is Lorenz like.

Now we remove a neighborhood of $f_0$ and $P_0$. The resulting manifold is diffeomorphic  to  $S^2\times [-1,1]$ and the vector field $X_0$ will be  pointing out at $A_{0}=S^2\times \set{1}$ (corresponding to removing a neighborhood of $P_0$) and entering at $B_{0}=S^2\times\set{-1}$ (corresponding to removing a neighborhood of $f_0$)(see figure \ref{X0}).

\begin{figure}[htb]
\begin{center}
\includegraphics[width=0.6\linewidth]{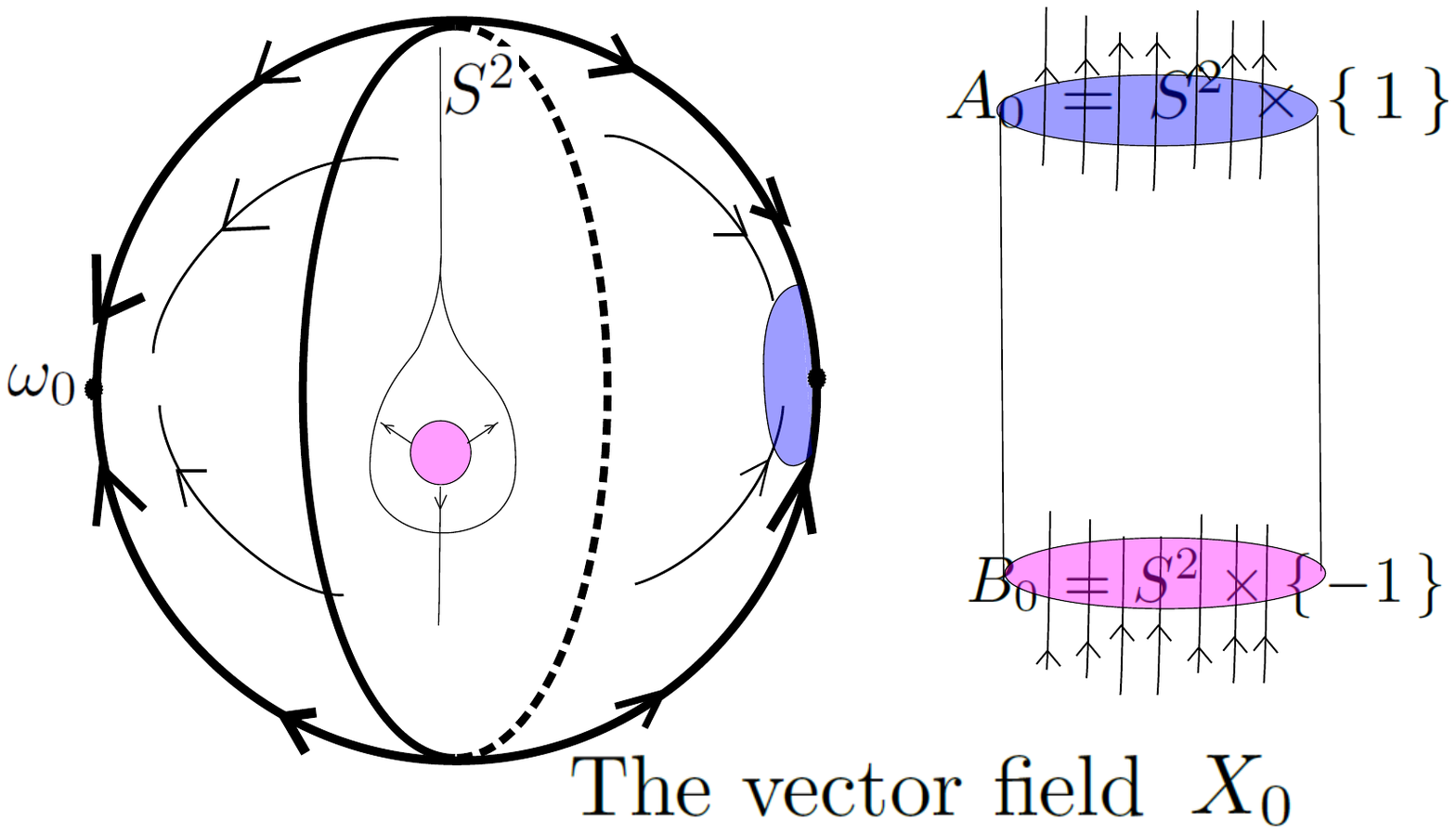}
\end{center}
\caption{ Removing  a neighborhood of $f_0$ and $P_0$  }
\label{X0}
\end{figure}

Now we consider an other copy of   $S^2\times [-1,1]$  with a vector field  $X_1$ that is the reverse time of  $X_0$.
Therefore  $X_1$ has a saddle called   $\sigma_1$  that has a strong expansion, a weaker expansion and a contraction, and   is Lorenz like.
It also has a sink called $\alpha_1$ a source called $\omega_1$ and saddle called $s_1$.

The vector field $X_1$ is entering at $A_{1}=S^2\times \set{1}$ and pointing out at $B_{1}=S^2\times\set{-1}$.

We can now paste $X_1$ and $X_0$  along their boundaries ($A_{0}$ with $A_{1}$ and the other two). Since both vector fields are transversal to the boundaries we can  obtain a $C^1$ vector field $X$ in the resulting manifold that is diffeomorphic to $S^2\times S^1$.

 We do not  paste any of the boundaries  with the identity.
  We describe first the gluing map of  $A_{0}$ with $A_{1}$. We paste them by a  rotation so that $$\left(\overline{W^u(\alpha_0)}\cap A_{0}\right)^c \text{ and } W^u( s_0)\cap A_{0}$$ are mapped to $$W^s(\alpha_1)\cap \left(A_{1}\right)\,.$$

 We require as well that $W^u(\sigma_0)\cap A_{0}$ is mapped to $W^s(\sigma_1)\cap A_{1}\,,$
We will later require an extra condition on this gluing map, which is a generic condition, and that will guarantee the multisingular hyperbolicity.

\begin{figure}[htb]
\begin{center}
\includegraphics[width=0.55\linewidth]{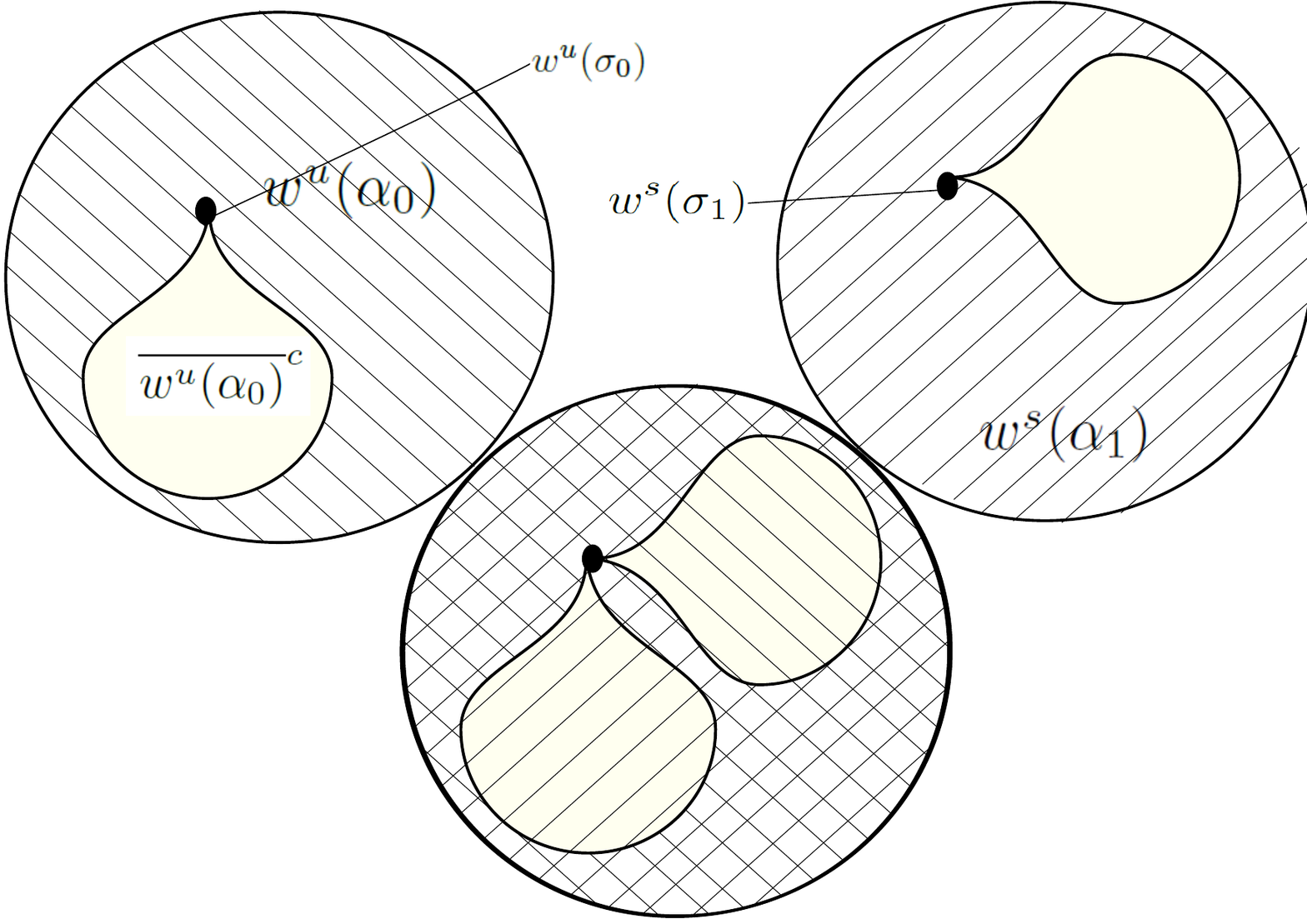}
\end{center}
\caption{ Pasting $S^2\times\set{1}$ with $S^2\times\set{1}$.}
\end{figure}

 To glue $B_{0}$ with $B_{1}$, let us first observe  that this boundaries were formed by removing a neighborhood of $f_0$ and $f_1$. Then by construction $\overline{W^s(\sigma_0)}\cap B_{0}\,$ is a circle that we will call $\cC_0$.  We can also define the corresponding $\cC_1$.
  Note that all points in $\cC_0$, except for one, are in $W^s(\sigma_0)$ while there is one point $l$ in $\cC_0$  that is in $W^s(s_0)$. We paste $B_{0}$ with $B_{1}$, mapping $\cC_0$ to cut transversally $\cC_1$ in points of  $W^s(\sigma_0)$ and  $W^u(\sigma_1)$.

Note that the resulting vector field $X$ has a cycle between two Lorenz like singularities $\sigma_0$ and $\sigma_1$.

\begin{figure}[htb]
\begin{center}
\includegraphics[width=0.50\linewidth]{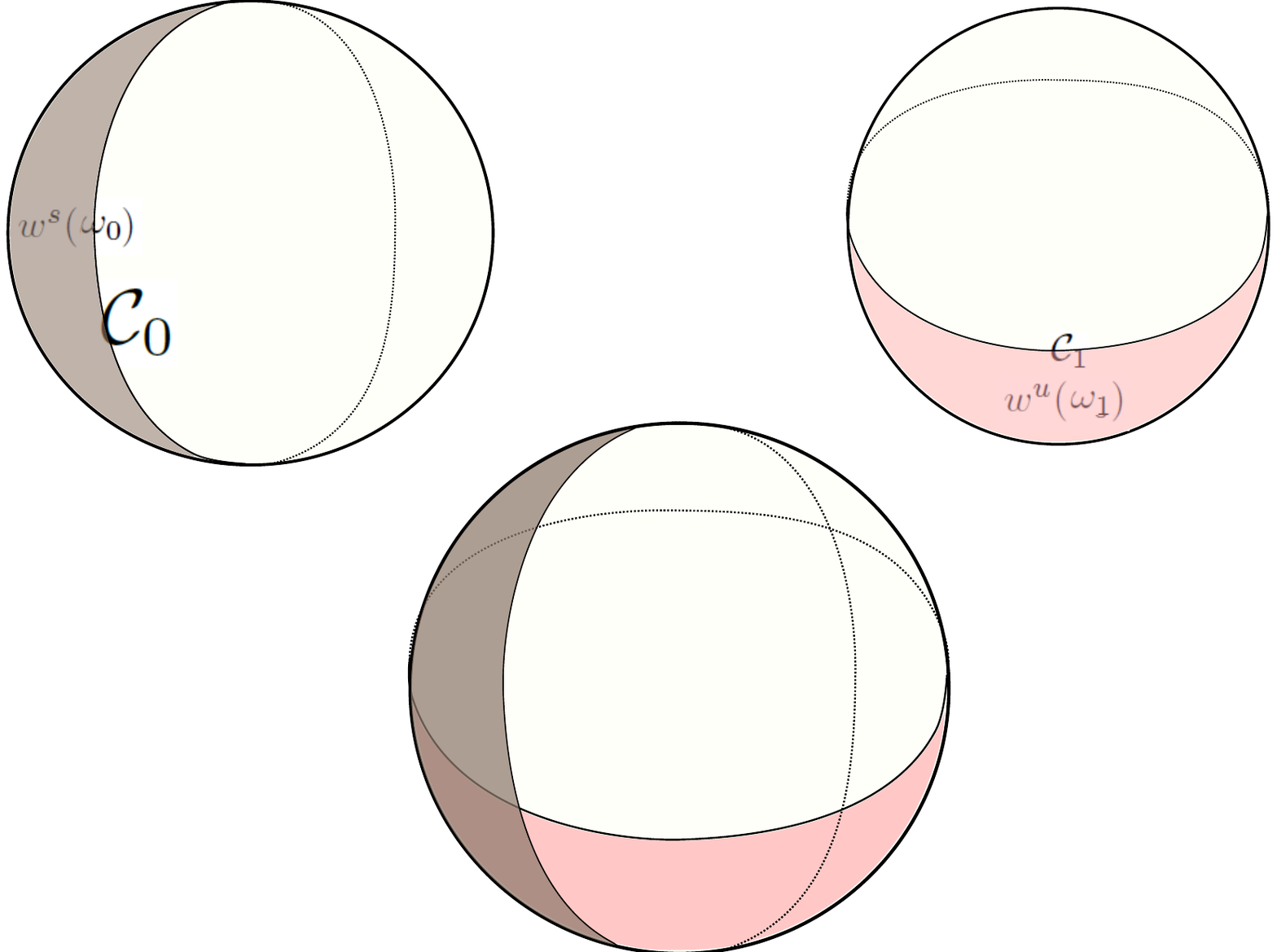}
\end{center}
\caption{ Pasting $S^2\times\set{-1}$ with $S^2\times\set{-1}$.}
\end{figure}

\pagebreak


\begin{lemm} The vector field $X$ defined above is such that the cycle and the singularities are the only chain recurrent points.
\end{lemm}
\proof
\begin{itemize}
\item All the recurrent orbits by $X_0$ in $S^3$ are the singularities. Once we remove the neighborhoods of the 2 singularities we obtain the manifold with boundary  $S^2\times [-1,1]$. The  the only other orbits of the vector field $X$ (that results from pasting $X_0$ and $X_1$) with hopes of being recurrent need to cut the boundaries.
\item The points in $A_0$ that are in  $\overline{W^u(\alpha_0)}^c\ W^u(\sigma_0)$ are wondering since they are mapped to the stable manifold of the sink $\alpha_1$
 \item The points in $B_0\cap W^u(\alpha_0)$ are wondering.
\end{itemize}
As a conclusion, the only point in $A_0$ whose orbit could be recurrent is the one in $$B_0\cap W^u(\sigma_0)$$.
Let us now look at the points in  $B_0$.
There is a circle $\cC_0$, corresponding to $\overline{W^s(\sigma_0)}\cap B_0\,$ that divides  $ B_0$ in 2 components.
One of this components is the basin of the sink $\omega_0$ and the other is what used to be the basin of  $P_0$.
So we have the following options:
\begin{itemize}
\item The points that are in the  basin of the sink $\omega_0$ are not chain recurrent.
\item The points that are in  what used to be the basin of  $P_0$ are either mapped into the basin of $\omega_1$ or are sent to  what used to be the basin of  $P_1$. Note that this points cross $A_0$ for the past, and since they are not in the stable manifold of $\sigma_1$ they are wondering.
\item  Some points in $\cC_0$ will be mapped to  the basin of $\omega_1$, others to   what used to be the basin of  $P_1$, and others to $\cC_1$.
In the two first cases those points are wondering
\end{itemize}

To sum up,
\begin{itemize}
\item The only recurrent orbits that cross $B_0$, are in the intersection of $\cC_0$ with$\cC_1$.
\item The only recurrent orbits that cross $A_0$, are in the intersection of $W^u(\sigma_0)$ with $W^s(\sigma_1)$.
\item The only recurrent orbits that do not cross the boundaries of $S^2\times [-1,1]$ are singularities.
\end{itemize}
This proves our lemma.
 \endproof

For the Lorenz singularity $\sigma_0$ of $X$ which is of positive saddle value and such that $T_{\sigma_0}M=E^{ss}\oplus E^{s}\oplus E^{uu}$, we define $B_{\sigma_0}\subset \PP M$ as $$B_{\sigma_0}=\pi_{\PP}\left( E^{s}\oplus E^{uu}\right)\,.$$

For the Lorenz singularity $\sigma_1$ of $X$ which is of positive saddle value and such that $T_{\sigma_1}M=E^{ss}\oplus E^{u}\oplus E^{uu}$, we define $B_{\sigma_1}\subset \PP M$ as $$B_{\sigma_1}=\pi_{\PP}\left(  E^{ss}\oplus E^{u}\right)\,.$$
Let $a$, $b$ and $c$ be points that are one in each of the 3 regular orbits forming the cycle between the two singularities of $X$. We call $a$ to the one such that the $\alpha$-limit of $a$ is $\sigma_0$. We define  $L_a=S_{X}(a)$ $L_b=S_{X}(b)$ and $L_c=S_{X}(c)$. We also note $O(L_a)$, $O(L_b)$ and $O(L_c)$ as the orbits of $L_a$, $L_b$ and $L_c$ by $\phi^t_{\PP}$.
\begin{prop}\label{p.unaorbitamas}
Suppose that $X$ is a vector field defined above.
Then there exist an open set $U$ containing the orbits of $a,b$ and $c$ and the saddles  $\sigma_0$  and $\sigma_1$, such that the extended maximal invariant set $B(U,X)$  is $$B(U,X)=B_{\sigma_0}\cup B_{\sigma_1}\cup O(L_a)\cup O(L_b)\cup O(L_c)\,.$$
\end{prop}

\begin{proof}
The  2 orbits of strong stable manifold  of $\sigma_0$ go by construction to $\alpha_0$ for the past. This implies that the strong stable manifold  is escaping.
The fact that there is a cycle tells us that there are no other escaping directions, therefore the center space is formed by the weak stable and the unstable spaces. By definition $B_{\sigma_0}=\PP^c_{\sigma_0}$.
Analogously we see that $B_{\sigma_1}=\PP^c_{\sigma_1}$.
Since the cycle formed by the orbits of $a,b$ and $c$ and the saddles  $\sigma_0$  and $\sigma_1$ is an isolated chain recurrence class, we can chose $U$ small enough so that this chain-class  is the maximal invariant set in $U$.
This proves our proposition.
\end{proof}

\begin{lemm} We can choose a  vector field $X$ defined above is multisingular hyperbolic in $U$.
\end{lemm}

\proof

The reparametrized linear Poincar\'e flow is Hyperbolic in restriction to the bundle over $B_{\sigma_0}\cup B_{\sigma_1}$ and of index one. We consider the set $B_{\sigma_0}\cup B_{\sigma_1}\cup O(L_a)$.

The strong stable space at $\sigma_0$ is the stable space for
 the reparametrized linear Poincar\'e flow.
 There is a well defined stable space in the linearized neighborhood of $\sigma_0$ and since the stable space is invariant for the future,
there is  a one dimensional stable flag  that extends along the orbit of $a$.
We can reason analogously with the strong unstable manifold of $\sigma_1$ and conclude that there is an unstable flag extending through the orbit of $a$. We ca choose the  gluing maps of  $S^2\times\set{-1}$ to $S^2\times\set{-1}\,,$ so that the stable and unstable flags in the orbit of $a $ intersect transversally.This is because this condition is open and dense in the possible gluing maps of  $S^2\times\set{-1}$ to $S^2\times\set{-1}\,,$  with the properties mentioned above.
 Therefore the set $B_{\sigma_0}\cup B_{\sigma_1}\cup O(L_a)$ is hyperbolic for the  reparametrized linear Poincar\'e flow.

Analogously we prove that $B_{\sigma_0}\cup B_{\sigma_1}\cup O(L_a)\cup O(L_b)\cup O(L_c)$  is hyperbolic for the  reparametrized linear Poincar\'e flow, and since from proposition \ref{p.unaorbitamas} there exist a $U$ such that,
 $$B_{\sigma1}\cup B_{\sigma2}\cup O(L_a)\cup O(L_b)\cup O(L_c)=B(U,X)\,.$$
 Then $X$ is multisingular hyperbolic in $U$.

\endproof

The example in \cite{BaMo} consists on two singular hyperbolic sets(negatively and positively)
  $H_{-}$ and $H_{+}$ of different indexes, and wandering orbits going from one to the other. Since they are singular hyperbolic $H_{-}$ and $H_{+}$  are multisingular hyperbolic sets of the same index. Moreover, the stable and unstable flags (for the reparametrized  linear Poincare flow ) along the orbits joining $H_{-}$ and $H_{+}$  intersect transversally. This is also true for $H_{-}$.

 With all this ingredients we can prove (in a similar way as we just did with the more simple example above ) that the chain  recurrence class containing  $H_{-}$ and $H_{+}$ in  \cite{BaMo} is multisingular hyperbolic, while it was shown by the authors that it is not singular hyperbolic.

 \section*{Acknowledgements}
 We wold like to thank  for their comments and suggestions, to Martin Sambario, Rafael Potrie and  Sylvain Crovisier for their comments and suggestions.
 We would also like to thank Lan Wen, Shaobo Gan, Yi Shi and the attendants of the dynamical system seminar at Peking University for their interest and encouragement.
 The second author was supported by the École doctorale Carnot Pasteur,  Centro de Matem\'aticas UdelaR, ANII FCE,and CAP UdelaR.

\end{document}